\documentclass{article}
\usepackage[utf8]{inputenc}
\usepackage{Preamble}
\usepackage{xcolor}

\theoremstyle{plain}\newtheorem{theorem}{Theorem}[section]
\theoremstyle{plain}\newtheorem{proposition}[theorem]{Proposition}
\theoremstyle{plain}\newtheorem{lemma}[theorem]{Lemma}
\theoremstyle{plain}\newtheorem{corollary}[theorem]{Corollary}
\theoremstyle{definition}\newtheorem{definition}[theorem]{Definition}
\theoremstyle{definition}\newtheorem{remark}[theorem]{Remark}
\theoremstyle{definition}
\theoremstyle{definition}\newtheorem{example}[theorem]{Example}

\title{Minimizing Harmonic Maps on the Unit Ball with Tangential Anchoring}
\date{\today}

\begin{document}

\author[1]{Lia Bronsard\footnote{bronsard@mcmaster.ca }}
\author[2]{Andrew Colinet\footnote{acolinet@mun.ca}}
\author[3]{Dominik Stantejsky\footnote{dominik.stantejsky@univ-lorraine.fr}}

\affil[1]{Department of Mathematics and Statistics, McMaster University, Hamilton, ON L8S 4L8 Canada}
\affil[2]{Department of Computer Science, Memorial University of Newfoundland, St. John's, NL A1C 5S7 Canada}
\affil[3]{Universit\'e de Lorraine, Institut \'Elie Cartan de Lorraine, UMR 7502 CNRS,  54506 Vand\oe uvre-l\`es-Nancy, France}

\maketitle

\begin{abstract}
    Since the seminal work of Schoen-Uhlenbeck, many authors have studied     properties of harmonic maps satisfying Dirichlet boundary conditions.
    In this article, we instead investigate regularity and symmetry of $\mathbb{S}^2-$valued minimizing harmonic maps subject to a tangency constraint in the model case of the unit ball in $\mathbb{R}^{3}$.
    In particular, we obtain a monotonicity formula respecting tangentiality on a curved boundary in order to show optimal regularity up to the boundary.
    We introduce novel sufficient conditions under which the minimizer must exhibit symmetries.
    Under a symmetry assumption, we present a delineation of the singularities of minimizers, namely that a mimimizer has exactly two point singularities, located on the boundary at opposite points. \\
\linebreak
\textbf{Keywords:}  Harmonic maps, regularity, monotonicity formula, tangential boundary condition, equivariance. \hfill \hfill
\linebreak
\textbf{MSC2020:} 
49K20, 
49Q20, 
35B38, 
35A16, 
35B07, 
49S05. 
\end{abstract}




\section{Introduction}
\label{sec:intro}

In this article, we study minimizing harmonic maps $u:\Omega\to \mathbb{S}^2$, $\Omega \coloneqq B_1(0)\subset\mathbb{R}^3$, with natural tangential boundary conditions that arise in the study of nematic liquid crystal droplets in the one-constant approximation of the Oseen-Frank model \cite{LaVol}.
That is, we investigate minimizers of
\begin{equation}\label{def:dir}
    E(w)
    \coloneqq\frac{1}{2}\int_{\Omega} |\nabla w|^2 \dx x \,,
\end{equation}
subject to the boundary conditions
\begin{equation}\label{eq:tanbc}
    w(x)\cdot\nu(x)=0\qquad \text{for }\mathcal{H}^{2}-\text{almost every }x\in\partial\Omega,
\end{equation}
where $\nu$ is the outward unit normal to $\partial\Omega$.
Minimizers satisfy
\begin{equation}\label{eq:MinPDE}
    \begin{aligned}[c]
    	\begin{cases}
    		\hfill -\Delta{}u(x)=|\nabla{}u(x)|^{2}u(x)\phantom{0}&\text{for }x\in\Omega,\\
    		\hfill u(x)\cdot\nu(x)=0\phantom{|\nabla{}u(x)|^{2}u(x)}&\text{for }x\in\partial\Omega,\\
    		\hfill \partial_{\nu}u_{T}(x)=0\phantom{|\nabla{}u(x)|^{2}u(x)}&\text{for }x\in\partial\Omega,
    	\end{cases}
     \end{aligned}
\end{equation}
where $u_T=u-u_\nu\nu$ \ds{is the tangential part of $u$}.
Contrarily to Dirichlet conditions, the map $u$ is not a prescribed tangential map on $\partial\Omega$, but enjoys the freedom to be any unit vector field in the tangent plane and this introduces new technical difficulties that will be described below.
From topological considerations, the tangent bundle of the sphere $T\mathbb{S}^2$
is not parallelizable, and thus the constraint \eqref{eq:tanbc} cannot be satisfied
everywhere on $\mathbb{S}^2$ \cite{Mil}.
Hence our boundary conditions cannot be smooth on the entirety of $\mathbb{S}^{2}$.

Our goal is to describe the minimizing configurations $u$ and, in particular, to characterize its regularity (smoothness), symmetry, and the location of its defects.

For the regularity, the goal is to show the optimal smoothness that was first obtained in \cite{SU,GiGi} that $u$ is smooth, except on a discrete singular set (see \cite{HaKiLi} for the general Oseen-Frank energy).
The classical approach that we follow here, consists of three main steps:
First, one obtains a monotonicity formula for an appropriate energy density using inner and outer variations of the energy.
This allows one to analyze the fine structure of the support of the Dirichlet energy measure.
Second, equipped with a monotonicity formula, a careful investigation of the properties of the singular set is undertaken to rule out singularities of dimension $1$.
Finally, using a blow up procedure, we argue that the minimizer must be $0-$homogeneous around singularities, thus showing that the singular set consists only of isolated points.

While the overall strategy remains the same, there are a number of obstacles
caused by the fact that our analysis includes the boundary, in which we restrict the boundary data $u|_{\partial \Omega}$ to lie in the tangent plane:
\begin{enumerate}[label=(\roman*)]
    \item
        Much of our argument is built around the idea of exchanging boundary
        information for interior information of an extended function.
        This was done previously in various settings, see \cite{ABC,BCS,DiMiPi,Sch}.
        As a result, an immediate obstacle is the ambiguity in selecting a
        suitable extension.
        A lot of care is needed to match the behavior of a minimizer of
        \eqref{def:dir} to sufficient degree in order to preserve the elliptic
        structure of the associated PDEs.
        Extending by a spherical reflection  yields an extended function that satisfies a PDE similar to the harmonic map equation and is a critical point of a reflected energy functional that agrees with the Dirichlet energy integral to highest order.
        At the same time, the extended function is \emph{not} guaranteed to be a minimizer for the unconstrained extended energy, and so the regularity
        could not be obtained from \cite{Luck}.
    \item
        Another obstacle is that the standard approaches to rigorously
        establishing a monotonicity formula make use of inner variations of the
        minimizer of \eqref{def:dir} by smooth compactly supported perturbations
        of $\Omega$ (see \cite{GiMe}).
        This presents an additional difficulty since we do not have minimality for the extended energy over $B_{2}(0)$, but only minimality of the extension on $\Omega$ and $B_{2}(0)\setminus\Omega$ separately.
        As a result, there is a need for the perturbation to preserve the tangency
        condition imposed along $\partial\Omega$ in order to make use of
        minimality of each of the two energies.
        The classical approach for inner variation consists of competitors of the form $u(\psi_{t}(x))$ for a family of diffeomorphisms
        $\{\psi_{t}\}_{t>0}$.
        However, these competitors do \emph{not} satisfy the boundary condition \eqref{eq:tanbc}.
        To ensure tangency we consider a suitable family of interior
        preserving perturbations and a transformation that preserves both tangency and norm, see Definition~\ref{def:innervartan} in Section~\ref{subsec:tang_inner_var}.
        This allows us to simultaneously obtain test functions for the energies
        on $\Omega$ and $B_{2}(0)\setminus\Omega$.
        While this has been achieved in \cite{Sch} in the case of flat
        boundaries, we are able to construct those diffeomorphisms in the
        considerably harder case of a curved boundary.
    \item
        To establish monotonicity properties of the $1$-density of the Dirichlet
        energy for points near the boundary \cite{Ev2} simply considers balls
        of sufficiently small radii that are contained in $\Omega$.
        This is not possible when analyzing behavior up to the boundary.
        We construct an explicit family of diffeomorphisms to which
        our inner variation formula applies to, see Section~\ref{subsubsec:rigorous_calc}.
        In particular, this family of diffeomorphisms is built inside
        a coordinate ball which is centred at a point $x_{0}$ on
        the boundary and chosen to match the
        map $x\mapsto{}(x-x_{0})$ to leading order.
        However, this results in a number of additional error terms when investigating the derivative of the $1$-density
        of the Dirichlet energy coming from the approximation to
        $x\mapsto{}(x-x_{0})$ and the additional terms in the PDE
        for the extension.
        The additional error terms are of a lower order and so the monotonicity formula we obtain retains the same character for small radii, similar to
        \cite{DiMiPi,Sch}.
        For example, we a priori obtain upper semi-continuity of the energy
        density only up to a constant, not with the precise constant $1$.
        We note that the error terms that are present in our inner and outer
        variations are preventing us from applying other approaches for
        regularity, such as \cite{Ev2} since the additional terms, although of lower order, can not easily be seen to vanish after the appropriate rescaling.
\end{enumerate}

The symmetry of the minimizers is also difficult to obtain due to the fact that standard symmetrization techniques relying on rearrangements do not preserve the pointwise norm of the function and thus other methods need to be developed, see for example \cite{BBCH,diFraSlaZar,DiFrFiSl,KaSha,Sa,SaSha2,SaSha}.
Other techniques rely on the explicit comparison with a minimizer to deduce symmetry, see e.g.\ \cite{FHL,MiPo}.
We present two approaches consisting of a decomposition of $u$ in cylindrical coordinates and a symmetrization of the coefficients that preserves the norm.
We are not able to prove that $u$ is symmetric in the sense that $u\cdot\ee_\theta=0$ everywhere and $u$ is equivariant with respect to rotations about an axis.
However, we are able to obtain two distinct conditions which imply symmetry or a partial symmetry.
These approaches do not only apply to the harmonic map setting, but could also be used e.g.\ for the Ginzburg-Landau energy, see Remark~\ref{rem:GL}.
We conjecture that the symmetry holds also without these conditions.
The symmetry would imply that for any minimizer there are two and only two point defects on the boundary, at opposite points on the sphere and no interior defects.

The analysis of minimizers in the constrained class of symmetric functions can be reduced to functions defined on two-dimensional slices of the three dimensional ball $B_1(0)$.
In order to show that those minimizers cannot have interior defects, we adapt the techniques invoked in \cite{ABL} to establish sharp bounds on the angle between $u$ and the axis of symmetry $\ee_3$, preventing $u$ from turning towards $-\ee_3$. 
This shows that in the setting of symmetric functions there are exactly two point defects, located on the boundary, as expected.

Another natural question concerns the value of the minimal energy, which is known to be $4\pi$ in the case of a Dirichlet condition $u=\frac{x}{|x|}$ on $\mathbb{S}^2$, see \cite{FHL}.
We provide an upper bound by constructing an admissible map given by the unit normal vector field of a family of spheres with varying radius and center that intersect the boundary and the axis $\ee_3$ perpendicularly, see Figure~\ref{fig:competitor}.
Then, a lower bound is derived by passing into a coordinate system using the level sets of the minimizer.
Using variations in a slightly larger class, one can compute explicit solutions to the associated Euler-Lagrange equations which provide us with explicit parameters that allow us to further estimate the energy from below. 

Our results lead naturally to the following open questions:
\begin{itemize}
    \item
        How can one prove that the above sufficient conditions for symmetry hold?
    \item
        Can one extend the analysis to weak anchoring, i.e.\ replace the
        condition $u\cdot\nu=0$ by a surface term of the form
        $\int_{\partial \Omega} |u\cdot\nu|^2\dx\mathcal{H}^2$ ? 
    \item
        Replacing the One-constant Oseen-Frank model for liquid crystals by
        e.g.\ the general Oseen-Frank model or the Landau-de Gennes model, can we obtain analogous results? 
\end{itemize}



In a forthcoming work, the case $u\cdot\nu=\alpha$, for $\alpha\in (0,1)$ will be considered.

\subsection{Equivariance}
\label{subsec:equivariance}

In order to discuss symmetries in solutions over $\Omega$ we introduce
cylindrical coordinates, $(\rho,\theta,z)$, defined on
$(0,1)\times(0,2\pi)\times(-1,1)$ by
\begin{equation*}
    x=(\rho\cos(\theta),\rho\sin(\theta),z).
\end{equation*}
These coordinates can be inverted to obtain
\begin{align*}
    \rho=\sqrt{x_{1}^{2}+x_{2}^{2}},\hspace{15pt}
    \theta=
    \begin{cases}
        \arccos\Bigl(\frac{x_{1}}{\sqrt{x_{1}^{2}+x_{2}^{2}}}\Bigr)&
        \text{if }x_{2}>0,\\
        \pi&\text{if }x_{2}=0,\\
        2\pi-\arccos\Bigl(\frac{x_{1}}{\sqrt{x_{1}^{2}+x_{2}^{2}}}\Bigr)&
        \text{if }x_{2}<0,
    \end{cases}\hspace{15pt}
    z=x_{3}.
\end{align*}
The unit vectors $\mathbf{e}_{\rho}$, $\mathbf{e}_{\theta}$, and $\mathbf{e}_{\varphi}$
are defined by
\begin{equation*}
    \mathbf{e}_{\rho}\coloneqq(\cos(\theta),\sin(\theta),0),\hspace{15pt}
    \mathbf{e}_{\theta}\coloneqq(-\sin(\theta),\cos(\theta),0),\hspace{15pt}
    \mathbf{e}_{z}\coloneqq\mathbf{e}_{3}.
\end{equation*}
Using these unit vectors we can represent a function $u\colon\Omega\to\mathbb{S}^{2}$
as
\begin{equation*}
    u(\rho,\theta,z)=u(x(\rho,\theta,z))
    =u_{\rho}(\rho,\theta,z)\mathbf{e}_{\rho}
    +u_{\theta}(\rho,\theta,z)\mathbf{e}_{\theta}
    +u_{z}(\rho,\theta,z)\mathbf{e}_{z}
\end{equation*}
where $u_{a}$ for $a\in\{\rho,\theta,z\}$ is defined by
\begin{equation*}
    u_{a}(\rho,\theta,z)\coloneqq{}u(x(\rho,\theta,z))\cdot\mathbf{e}_{a}.
\end{equation*}

\begin{definition}[Equivariance]\label{def:equivariance}
A function $u\colon{}\Omega\xrightarrow{}\mathbb{R}^3$ is called \emph{equivariant} around an axis $\ell$ passing through the origin if for all rotations $R\in SO(3)$ around $\ell$ (i.e.\ $R$ leaves $\ell$ invariant) it holds that
\begin{align*}
u(Rx) \ = \ Ru(x) \, .
\end{align*}
We will later consider the axis $\ell$ given by the $\ee_3-$axis, in which case the rotation matrices $R$ are of the form $R_\theta$, where
\begin{align*}
R_\theta 
\ = \ 
\begin{pmatrix}
\cos(\theta) & -\sin(\theta) & 0 \\
\sin(\theta) & \cos(\theta)  & 0 \\
0 			  & 0 			   & 1
\end{pmatrix}
\end{align*}
for some angle $\theta\in [0,2\pi)$.
\end{definition}

A function, $u$, written in cylindrical coordinates as $u= u_\rho\ee_\rho + u_\theta \ee_\theta + u_z\ee_z$ is equivariant w.r.t. the $\ee_z-$axis if and only if all coefficient functions of $u$ are independent of $\theta$. The only dependence of $u$ on $\theta$ is via the basis vectors $\ee_r,\ee_\theta$. Symbolically, this means that
\begin{align*}
    \partial_\theta u_a 
    \ \equiv \ 0 
    \qquad \forall a\in\{r,\theta,z\} \, .
\end{align*}

\subsection{Main Results}
\label{subsec:main_results}

Our main results regarding minimizers of \eqref{def:dir} can be stated as follows:

\begin{theorem}\label{thm:main_regularity}
Let $u$ be a global minimizer of \eqref{def:dir} satisfying the boundary conditions \eqref{eq:tanbc}.
Then $u$ has only finitely many point defects and is smooth up to the boundary away from those defects. 
\end{theorem}

The remaining results require some concepts regarding equivariance.
We refer the reader to Subsection~\ref{subsec:equivariance} for the
necessary background.

\begin{theorem}\label{thm:main_symmetry}
Let $u$ be a global minimizer of \eqref{def:dir} satisfying the boundary conditions \eqref{eq:tanbc}.
\begin{enumerate}
    \item If in addition 
    \begin{align*}
        \T 
        \ \coloneqq \
        \int_{\Omega} \frac{1}{\rho^2}(u_\rho\partial_\theta u_\theta - u_\theta\partial_\theta u_\rho) \dx x
        \ \geq \ 
        0
        \, ,
    \end{align*}
    then $u$ is equivariant about the $\ee_3-$axis. 
    \item If $u\cdot\ee_\theta=0$ on $\partial \Omega$, then $u\cdot\ee_\theta=0$ in the whole ball $B_1(0)$ and $u$ is equivariant about the $\ee_3-$axis.
\end{enumerate}
\end{theorem}

\begin{theorem}\label{thm:main_no_interior_defects}
Let $u$ be a global minimizer of \eqref{def:dir} satisfying the boundary conditions \eqref{eq:tanbc} such that $u_\theta=0$.
Then $u$ has exactly two point defects located at the North and South pole. 
\end{theorem}

\begin{remark}[Ginzburg-Landau]\label{rem:GL}
    Parts of our analysis not only apply to the $\mathbf{S}^2-$valued setting of harmonic maps, but also carry over to similar problems, for example the minimization of the  Ginzburg-Landau energy
    \begin{align*}
        E_{\eps}^{GL}(u)
        \ \coloneqq \
        \int_{\Omega} \frac12 |\nabla u|^2 + \frac{1}{4\eps^2}(|u|^2-1)^2 \dx x
        \, .
    \end{align*}
    In particular, the reflection of the energy in Section~\ref{sec:refl_E_reg} and the symmetrization procedure of Proposition~\ref{prop:symmetrization} in Section~\ref{sec:sym_equiv} do not require that $|u|=1$.
\end{remark}

\begin{remark}[General domains]\label{rem:generalization}
    We believe that the methods developed in this article for the particular geometry of $B_1(0)$ can also be employed to other, sufficiently smooth and symmetric domains. 
    For example, in the two-dimensional setting, the reflection-extension approach for tangential boundary conditions has successfully been employed in curved geometry in \cite{ABC,BCS}.
\end{remark}

\begin{remark}[Energy in 3D]\label{rem:3D_O1_bound}
    There are two main difficulties for proving results like Theorem \ref{thm:main_no_interior_defects} in 3D:
    \begin{enumerate}
        \item The loss of complex structure as the domain and range can no longer be seen as subsets of $\mathbb{C}$.
        The loss of complex structure can be partially compensated using tools from geometry as in Section~\ref{sec:refl_E_reg}.
        \item Point defects have finite energy. \ds{This implicates that it is no longer sufficient to study the Jacobian $\det(\nabla u)$ since the energy of the defects that are measured by the Jacobian  are of the same order as the interaction energy between the defects, that were of different order in the two-dimensional setting. }
    \end{enumerate}
    Regarding the energy bound, in the situation of this article we can pose a competitor and  show that $E(u_0)=5\pi - \tfrac{\pi^3}{4}\approx 7.956<\infty$ and $u_0$ has two boundary defects at $(0,0,\pm 1)$, see Lemma~\ref{lem:upper_bound_competitor}. 
    Articles concerned with the radial boundary condition $u\cdot\nu=1$ can make use of the explicit shape of the minimizer, known to be $u_*(x)=\tfrac{x}{|x|}$ \cite{BrCoLi,Jäger1983,FHL}.
\end{remark}

\paragraph{Organization of the paper.}
In Section~\ref{sec:refl_E_reg} we prove the regularity result Theorem~\ref{thm:main_regularity} using several steps:
First, we introduce the extended minimizer $\overline{u}$ that we reflect across the boundary $\partial\Omega$ (Subsection~\ref{subsec:reflprelim}).
Second, the show that $\overline{u}$ is a critical point of a reflected energy and thus satisfies a PDE on a larger domain, see Subsection~\ref{subsec:ext_outer_var}.
Third, a monotonicity formula is derived in Subsection~\ref{subsec:monoform} based on tangential inner variations that are introduced in Subsection~\ref{subsec:tang_inner_var}.
This can in turn be used to prove the regularity statement in Subsection~\ref{subsec:fullreg}.
All symmetry and equivariance results are proven in Section~\ref{sec:sym_equiv}.
Assuming $u_\theta=0$, we continue the analysis in Section~\ref{sec:bdry_defects}, showing that no interior defects can occur.
The last part of the article is devoted to the estimation of the minimal energy, see Section~\ref{sec:profile}.

\paragraph{Acknowledgments.}
LB received funding from a NSERC Discovery Grant.
This work was largely completed while AC and DS were affiliated with McMaster University.

\section{Regularity}
\label{sec:refl_E_reg}

In this section we prove Theorem~\ref{thm:main_regularity}, i.e.\ that minimizers have only finitely many point defects and are otherwise analytic.
These types of results have been known since the works of Schoen-Uhlenbeck \cite{SU} and Giaquinta-Giusti \cite{GiGi}.

Since we do not have a fixed Dirichlet condition at the boundary, we cannot directly apply their results.
To overcome this issue, we extend our energy as well as our minimizers
outside the domain $\Omega$, while still preserving minimality on the
exterior energy.
This extension idea has been successfully applied to solutions of partial differential equations \cite{BCS,DiMiPi,Sch}. 
Our extension is somewhat special since it takes place on the energy level, but as we note in Remark \ref{rem:refl_E_PDE_ELE} we could have also reflected the Euler-Lagrange equation instead as we preserve the variational character of the equation. 
Note, however, that the extended function is only minimal on the inside and on the outside separately, not minimizing both energies jointly. 
If this was the case, one could apply the results in \cite{Luck} to obtain regularity.
The surprising result that we obtain in Corollary~\ref{cor:outerpde}, is that the extended function is a critical point of the combined energies.\\

As we will be interested in the boundary conditions introduced in
\eqref{eq:tanbc} we define a corresponding subspace of
$W^{1,2}(\Omega;\mathbb{S}^{2})$ by
\begin{equation*}
    W_{T}^{1,2}(\Omega;\mathbb{S}^{2})\coloneqq
    \bigl\{w\in{}W^{1,2}(\Omega;\mathbb{S}^{2}):w\text{ satisfies }
    \eqref{eq:tanbc}
    \bigr\}.
\end{equation*}
Occasionally we will also make use of the corresponding subspace of
$W^{1,2}(\Omega;\mathbb{R}^{3})$.
We also introduce the tangential part, denoted $w_{T}$, defined by
\begin{equation*}
	w_{T}(x)\coloneqq\biggl[I_{3}-\frac{xx^{T}}{|x|^{2}}\biggr]w(x)
\end{equation*} 
for $x\in\Omega\setminus\{0\}$.
We notice that since $u$ is a minimizer of \eqref{def:dir}
then it is a critical point among outer variations and hence
\begin{equation}\label{eq:weakform}
	\int_{\Omega}\!{}\nabla{}u(x):\nabla\varphi(x)=\int_{\Omega}\!{}|\nabla{}u(x)|^{2}u(x)\cdot\varphi(x)
\end{equation}
for all $\varphi\in{}W_{T}^{1,2}(\Omega;\mathbb{R}^{3})\cap{}L^{\infty}(\Omega;\mathbb{R}^{3})$.
Notice that \eqref{eq:weakform} shows that the weak Laplacian of $u$,
$-\Delta{}u$, is an $L^{1}$ function.\\

Next, we define, as in \cite{BCS}, a weak formulation of the normal derivative
of the tangential part that is satisfied by minimizers of \eqref{def:dir}.
For any function
$w\in{}W_{T}^{1,2}(\Omega;\mathbb{R}^{3})\cap{}L^{\infty}(\Omega;\mathbb{R}^{3})$
satisfying $-\Delta{}w\in{}L^{1}(\Omega;\mathbb{R}^{3})$ we may define
the normal derivative of the tangential part,
$\partial_{\nu}w_{T}\in{}\bigl(W^{\frac{1}{2},2}(\partial\Omega;\mathbb{R}^{3})\cap{}L^{\infty}(\partial\Omega;\mathbb{R}^{3})\bigr)^{*}$,
by
\begin{equation*}
    \bigl<\partial_{\nu}w_{T},\phi\bigr>\coloneqq
    \int_{\Omega}\!{}(-\Delta{}w)\cdot\Phi
    -\int_{\Omega}\!{}\nabla{}w:\nabla\Phi
\end{equation*}
where $\phi\in{}W^{\frac{1}{2},2}(\partial\Omega;\mathbb{R}^{3})\cap{}L^{\infty}(\partial\Omega;\mathbb{R}^{3})$ satisfies \eqref{eq:tanbc} and
$\Phi\in{}W_{T}^{1,2}(\Omega;\mathbb{R}^{3})\cap{}L^{\infty}(\Omega;\mathbb{R}^{3})$ satisfies $\Phi_{T}\vert_{\partial\Omega}=\phi$.
Note that this is independent of the choice of extension since if
$\Phi_{1},\Phi_{2}$ are two such extensions of $\phi$ then
$\Phi\coloneqq\Phi_{1}-\Phi_{2}\in{}W_{0}^{1,2}(\Omega;\mathbb{R}^{3})$
and hence we would have
\begin{equation*}
    \int_{\Omega}\!{}(-\Delta{}w)\cdot\Phi
    =\int_{\Omega}\!{}\nabla{}w:\nabla\Phi.
\end{equation*}
In our case, this definition gives that $\partial_{\nu}u_{T}=0$ due to
\eqref{eq:weakform}.

\subsection{Reflection Preliminaries}\label{subsec:reflprelim}

In this subsection, we introduce the notation and basic concepts regarding
an extension of a harmonic map $u\colon\Omega\to\mathbb{S}^{2}$ in
$W_{T}^{1,2}(\Omega;\mathbb{S}^{2})$ in a way
that preserves both the elliptic structure and the modulus of the map.

The extension of a minimizer $u$ of \eqref{def:dir} takes place in two steps.
First, we invert the domain at the boundary using the involution defined
in \eqref{def:sphereinv} so that we can extend the definition of $u$ to
$B_{2}(0)$.
In a second step, we switch the sign on the radial component of $u$ using
the map defined in \eqref{def:radialswitch}.
Note that this does not change the fact that the new function is still
$\mathbb{S}^2-$valued. This step is crucial to ensure the continuity of
derivatives across the boundary $\partial\Omega$.\\

We start by defining
$A\colon\mathbb{R}^{3}\setminus\{0\}\to{}M_{3\times3}(\mathbb{R})$
by
\begin{equation}\label{def:radialswitch}
	A(x)\coloneqq{}I_{3}-2\frac{xx^{T}}{|x|^{2}}
\end{equation}
where $I_{3}$ is the $3\times3$ identity matrix.
We observe that for $x\in\mathbb{R}^{3}\setminus\{0\}$ we have the following
properties:
\begin{align}
	A(x)A(x)&=I_{3},\label{eq:involutionA}\\
	|A(x)|^{2}&=3,\\
	\partial_{x_{i}}A(x)&=-\frac{2}{|x|^{2}}\bigl[\mathbf{e}_{i}x^{T}+x\mathbf{e}_{i}^{T}\bigr]+\frac{4x_{i}}{|x|^{4}}xx^{T},
	\label{eq:partialA}\\
	\Delta{}A(x)&=\frac{12}{|x|^{4}}xx^{T}-\frac{4}{|x|^{2}}I_{3}.
\end{align}
Next, we let
$\iota\colon\mathbb{R}^{3}\setminus\{0\}\to\mathbb{R}^{3}\setminus\{0\}$
denote inversion about the sphere defined as
\begin{equation}\label{def:sphereinv}
    \iota(x)\coloneqq\frac{x}{|x|^{2}}.
\end{equation}
We note that this choice of $\iota$ preserves the differential equation to highest order. Other choices, such as $\iota(x)=(\frac2{|x|}-1)x$, lead to additional contributions in the second order derivatives of $u$.
We notice that $\iota$ satisfies:
\begin{align}
	\nabla{}\iota(x)&=\frac{1}{|x|^{2}}\biggl[I_{3}-2\frac{xx^{T}}{|x|^{2}}\biggr],\\
	\det(\nabla{}\iota(x))&=\frac{-1}{|x|^{6}}.\label{eq:covdet}
\end{align}
For a function $u\colon\Omega\to\mathbb{S}^{2}$ we define an extension,
$\overline{u}\colon{}B_{2}(0)\to\mathbb{S}^{2}$, by spherical
inversion as
\begin{equation}\label{eq:extension}
	\overline{u}(x)\coloneqq
	\begin{cases}
		u(x)&x\in\Omega,\\
		A(x)u\bigl(\iota(x)\bigr)&x\in{}B_{2}(0)\setminus\Omega.
	\end{cases}
\end{equation}
Note that preservation of the modulus of $\overline{u}$ on
$B_{2}(0)\setminus\Omega$ follows from \eqref{eq:involutionA}.
For convenience, we let, for maps $v\colon\Omega\to\mathbb{S}^{2}$,
$\widetilde{v}\colon{}B_{2}(0)\setminus\Omega\to\mathbb{S}^{2}$ denote
the map
\begin{equation}\label{def:tildeu}
    \widetilde{v}(x)\coloneqq{}A(x)v\bigl(\iota(x)\bigr).
\end{equation}
We also note that following identities
\begin{align}
    \overline{u}(x)&=A(x)\overline{u}\bigl(\iota(x)\bigr)
    \label{eq:invol}\\
	\partial_{x_{i}}\bigl[u(\iota(x))\bigr]&=\frac{1}{|x|^{2}}\partial_{x_{i}}u(\iota(x))
	-\frac{2x_{i}}{|x|^{4}}\nabla{}u(\iota(x))x
    \label{eq:partialinv}\\
	\nabla\bigl[u(\iota(x))\bigr]
	&=\frac{1}{|x|^{2}}\nabla{}u\bigl(\iota(x)\bigr)A(x)\label{eq:gradinversion}\\
	\bigl|\nabla\bigl[u\bigl(\iota(x)\bigr)\bigr]\bigr|^{2}&=\frac{1}{|x|^{4}}\bigl|\nabla{}u\bigl(\iota(x)\bigr)\bigr|^{2}.
	\label{eq:normsq}
\end{align}
Next, we introduce the exterior energy, $\widetilde{E}$ defined by
\begin{equation}\label{def:extenergy}
    \begin{aligned}[c]
    	\widetilde{E}(w)\coloneqq4\pi+
    	\frac{1}{2}&\int_{B_{2}(0)\setminus\Omega}\!{}\frac{1}{|x|^{2}}
    	\biggl[
    	|\nabla{}w(x)|^{2}
    	+\frac{4}{|x|^{4}}(w(x)\cdot{}x)^{2}\\
    	&+\frac{4}{|x|^{2}}w(x)^{T}\nabla{}w(x)^{T}x
        -\frac{4}{|x|^{2}}(w(x)\cdot{}x)\text{div}(w)(x)\biggr]
    \end{aligned}
\end{equation}
for functions $w\in{}W^{1,2}(B_{2}(0)\setminus\Omega;\mathbb{S}^{2})$.
As we will be particularly interested in relating harmonic maps on $\Omega$
to minimizers of $\widetilde{E}$ we will now introduce a suitable function
space.
For each $g\in{}W^{\frac{1}{2},2}(\partial{}B_{2}(0);\mathbb{S}^{2})$
we define the function space
\begin{equation*}
	\mathcal{W}_{g}\coloneqq
	\bigl\{w\in{}W^{1,2}(B_{2}(0)\setminus\Omega;\mathbb{S}^{2})\cap{}L^{\infty}(B_{2}(0)\setminus\Omega;\mathbb{S}^{2})
    :w\vert_{\partial{}B_{2}(0)}=g,\,
	w\text{ satisfies }\eqref{eq:tanbc}\bigr\}.
\end{equation*}
Of particular interest will be the case when
$g=\overline{u}\vert_{\partial{}B_{2}(0)}$ which we will abbreviate,
for convenience, to $\mathcal{W}_{u}$.
\\
Finally, we introduce the combined energy, $\overline{E}$, defined
for $w\in{}W^{1,2}(B_{2}(0);\mathbb{R}^{3})$ as
\begin{equation*}
    \overline{E}(w)\coloneqq
    E(w)+\widetilde{E}(w).
\end{equation*}
This will be important as we intend to ``glue" together compatible functions
defined on $\Omega$ and $B_{2}(0)\setminus\Omega$ in order to make boundary
vortices interior.

\subsection{Extension of the Outer Variation}
\label{subsec:ext_outer_var}

In this subsection we show that $u$ is a minimizer of \eqref{def:dir}
if and only if $\widetilde{u}$ is a minimizer of \eqref{def:extenergy}.
As a result, $\overline{u}$ satisfies a PDE on $B_{2}(0)\setminus\Omega$.
In addition, we demonstrate that the PDEs satisfied separately on
$\Omega$ and $B_{2}(0)\setminus\Omega$ ``glue" together due to the way
the extension from \eqref{eq:extension} was defined.\\

\begin{proposition}\label{prop:refl_E}
    A function $u$ is a minimizer of \eqref{def:dir} over
    $W_{T}^{1,2}(\Omega;\mathbb{S}^{2})$ if and only if
    $\widetilde{u}$ is defined as in \eqref{def:tildeu} above is a minimizer
    of the energy \eqref{def:extenergy} over $\mathcal{W}_{u}$.
\end{proposition}

\begin{proof}
Starting from \eqref{def:dir}, we can perform a change of variables $x=\iota(y)$ with Jacobian computed in \eqref{eq:covdet} to find that
for all functions $u\colon\Omega\to\mathbb{S}^{2}$ we have
\begin{align*}
    E\bigl(u,A_{\frac{1}{2},1}\bigr)
    \ &= \
    \frac12\int_{B_{2}(0)\setminus\Omega} \frac{1}{|y|^6}
    \left| \nabla u\bigl(\iota(y)\bigr) \right|^2 \dx y
    \ = \
    \frac12\int_{B_{2}(0)\setminus\Omega} \frac{1}{|y|^2}\left| \nabla \left[A(y) \tilde{u}(y)\right] \right|^2 \dx y
    \, .
\end{align*}
Calculating $\nabla \left[A(y) \tilde{u}(y)\right]$  by appealing to
\eqref{eq:partialA}, one can see that
\begin{align*}
    E\bigl(u,A_{\frac{1}{2},1}\bigr)
    \ &= \
    \frac12 \int_{B_{2}(0)\setminus\Omega}
    \frac{1}{|y|^2}\biggl( 
    |\nabla \tilde{u}|^2
    - \frac{4}{|y|^2}(y\cdot\tilde{u})\dif(\tilde{u})
    + \frac{4}{|y|^2} y\cdot (\nabla\tilde{u})\tilde{u}
    + \frac{4}{|y|^4} (y\cdot\tilde{u})^2
    + \frac{4}{|y|^2} 
     \biggr)
    \dx y \\
    \ &= \
    \widetilde{E}(\widetilde{u}).
\end{align*}
Minimality of $u$ for $E$ is therefore equivalent to minimality of $\tilde{u}$ for $\widetilde{E}$ as each competitor for one of them can be transformed into a competitor for the other via the involutive mappings $\iota$ and $A$.
\end{proof}

Since $\widetilde{u}$ is a minimizer of $\widetilde{E}$ then it is also a
critical point of $\widetilde{E}$ among outer variations.
As a result, we have the following corollary:
\begin{corollary}\label{cor:outerpde}
	Suppose that $\varphi\in{}\mathcal{W}_{0}$.
	Then we have
	\begin{align*}
		\int_{B_{2}(0)\setminus\Omega}\!{}\frac{1}{|x|^{2}}\nabla\widetilde{u}(x):\nabla\varphi(x)
		=&\int_{B_{2}(0)\setminus\Omega}\!{}\frac{1}{|x|^{2}}\biggl[
		-\frac{4}{|x|^{4}}(\widetilde{u}(x)\cdot{}x)x^{T}
		+\frac{4}{|x|^{4}}(\widetilde{u}(x)\cdot{}x)^{2}\widetilde{u}(x)^{T}
		+|\nabla\widetilde{u}(x)|^{2}\widetilde{u}(x)^{T}\\
		&-\frac{4}{|x|^{2}}x^{T}\nabla\widetilde{u}(x)
		+\frac{4}{|x|^{2}}\bigl[\widetilde{u}(x)^{T}\nabla\widetilde{u}(x)^{T}x\bigr]\widetilde{u}(x)^{T}
		+\frac{4}{|x|^{2}}{\rm{div}}(\widetilde{u})(x)x^{T}\\
		&-\frac{4}{|x|^{2}}(\widetilde{u}(x)\cdot{}x){\rm{div}}(\widetilde{u})(x)\widetilde{u}(x)^{T}
	    \biggr]\varphi(x).
	\end{align*}
	In addition, we have
	\begin{align*}
		\int_{B_{2}(0)\setminus\Omega}\!{}\nabla\widetilde{u}(x):\nabla\varphi(x)
		=&\int_{B_{2}(0)\setminus\Omega}\!{}\biggl[
		-\frac{4}{|x|^{4}}(\widetilde{u}(x)\cdot{}x)x^{T}
		+\frac{4}{|x|^{4}}(\widetilde{u}(x)\cdot{}x)^{2}\widetilde{u}(x)^{T}
		+|\nabla\widetilde{u}(x)|^{2}\widetilde{u}(x)^{T}\\
		&-\frac{4}{|x|^{2}}x^{T}\nabla\widetilde{u}(x)
		+\frac{4}{|x|^{2}}\bigl[\widetilde{u}(x)^{T}\nabla\widetilde{u}(x)^{T}x\bigr]\widetilde{u}(x)^{T}
		+\frac{4}{|x|^{2}}{\rm{div}}(\widetilde{u})(x)x^{T}\\
		&-\frac{4}{|x|^{2}}(\widetilde{u}(x)\cdot{}x){\rm{div}}(\widetilde{u})(x)\widetilde{u}(x)^{T}
        -\frac{2}{|x|^{2}}x^{T}\nabla\widetilde{u}(x)^{T}
		\biggr]\varphi(x).
	\end{align*}
\end{corollary}

\begin{proof}
	Notice that the second identity follows from the first by treating $\frac{1}{|x|^{2}}\varphi(x)$ as the test function.
	In order to prove the first identity we consider $\varphi$ satisfying
    the hypothesis of the corollary.
	We can calculate that
	\begin{align*}
		\frac{\mathrm{d}}{\mathrm{d}t}\bigg|_{t=0}\widetilde{E}\biggl(\frac{\widetilde{u}(x)+t\varphi(x)}{|\widetilde{u}(x)+t\varphi(x)|}\biggr)
		&=\int_{B_{2}(0)\setminus\Omega}\!{}\frac{1}{|x|^{2}}\biggl[\frac{4}{|x|^{4}}\widetilde{u}(x)^{T}xx^{T}\biggr]\varphi(x)
		-\int_{B_{2}(0)\setminus\Omega}\!{}\frac{1}{|x|^{2}}\biggl[\frac{4}{|x|^{4}}(\widetilde{u}(x)\cdot{}x)^{2}
		(\widetilde{u}(x)\cdot{}\varphi(x))\biggr]\\
		+&\int_{B_{2}(0)\setminus\Omega}\!{}
        \frac{1}{|x|^{2}}\biggl[-|\nabla\widetilde{u}(x)|^{2}\widetilde{u}(x)^{T}\biggr]\varphi(x)
		+\int_{B_{2}(0)\setminus\Omega}\!{}\frac{1}{|x|^{2}}\nabla\widetilde{u}(x):\nabla{}\varphi(x)\\
		+&\frac{1}{2}\int_{B_{2}(0)\setminus\Omega}\!{}
		\frac{4}{|x|^{4}}\varphi(x)^{T}\nabla\widetilde{u}(x)^{T}x
		-\int_{B_{2}(0)\setminus\Omega}\!{}
		\frac{4}{|x|^{4}}(\widetilde{u}(x)\cdot{}\varphi(x))\widetilde{u}(x)^{T}\nabla\widetilde{u}(x)^{T}x\\
		+&\frac{1}{2}\int_{B_{2}(0)\setminus\Omega}\!{}\frac{4}{|x|^{4}}\widetilde{u}(x)^{T}
		\nabla{}\varphi(x)^{T}x
        -\frac{1}{2}\int_{B_{2}(0)\setminus\Omega}\!{}\frac{4}{|x|^{4}}
		(\varphi(x)\cdot{}x)\text{div}(\widetilde{u})(x)\\
        +&\int_{B_{2}(0)\setminus\Omega}\!{}\frac{4}{|x|^{4}}
        (\widetilde{u}(x)\cdot{}x)(\widetilde{u}(x)\cdot\varphi(x))
        \text{div}(\widetilde{u})(x)
		-\frac{1}{2}\int_{B_{2}(0)\setminus\Omega}\!{}\frac{4}{|x|^{4}}(\widetilde{u}(x)\cdot{}x)\text{div}(\varphi)(x)
	\end{align*}
	Integrating by parts and using that $u$ satisfies \eqref{eq:tanbc} as well as
    that $\varphi(x)=0$ on $\partial{}B_{2}(0)$ we see that
	\begin{align*}
		-\frac{1}{2}\int_{B_{2}(0)\setminus\Omega}\!{}\frac{4}{|x|^{4}}(\widetilde{u}(x)\cdot{}x)\text{div}(\varphi)(x)
		=&\frac{1}{2}\int_{B_{2}(0)}\!{}\frac{4}{|x|^{4}}\varphi(x)^{T}\nabla\widetilde{u}(x)^{T}x
		+\frac{1}{2}\int_{B_{2}(0)\setminus\Omega}\!{}\frac{4}{|x|^{4}}(\widetilde{u}(x)\cdot{}\varphi(x))\\
		&-\frac{1}{2}\int_{B_{2}(0)\setminus\Omega}\!{}\frac{16}{|x|^{6}}(\widetilde{u}(x)\cdot{}x)(\varphi(x)\cdot{}x).
	\end{align*}
	A second integration by parts using \eqref{eq:tanbc} and
    $\varphi(x)=0$ on $\partial{}B_{2}(0)$ results in
	\begin{align*}
		\frac{1}{2}\int_{B_{2}(0)\setminus\Omega}\!{}\frac{4}{|x|^{4}}\widetilde{u}(x)^{T}\nabla{}\varphi(x)^{T}x
		=&\frac{-1}{2}\int_{B_{2}(0)\setminus\Omega}\!{}\frac{4}{|x|^{4}}(\varphi(x)\cdot{}x)
		\text{div}(\widetilde{u})(x)+\frac{1}{2}\int_{B_{2}(0)\setminus\Omega}\!{}\frac{16}{|x|^{6}}(\widetilde{u}(x)\cdot{}x)
		(\varphi(x)\cdot{}x)\\
		&-\frac{1}{2}\int_{B_{2}(0)\setminus\Omega}\!{}\frac{4}{|x|^{4}}(\widetilde{u}(x)\cdot{}\varphi(x)).
	\end{align*}
	Putting these into our calculation of the outer variation and using Proposition
    \ref{prop:refl_E} gives the desired result.
\end{proof}

\begin{remark}[Reflection of the Energy vs.\ reflection of the Euler-Lagrange equations]\label{rem:refl_E_PDE_ELE}
Determining the PDE associated to the second equation of Corollary~\ref{cor:outerpde}, we obtain
\begin{equation}\label{rem:ELE_tilde_n}
\begin{aligned}[c]
    -\Delta \widetilde{u}
    \ &= \
    - \frac{2}{|y|^2} (y\cdot\nabla) \widetilde{u}
    -\frac{4}{|y|^4} (\widetilde{u}\cdot y) y
    -\frac{4}{|y|^2} y(\nabla \widetilde{u})
    + \frac{4}{|y|^2} \dif(\widetilde{u}) y \\
    &\quad+ \left( |\nabla \widetilde{u}|^2
    + \frac{4}{|y|^4} (\widetilde{u}\cdot y)^2
    - \frac{4}{|y|^2} (y\cdot \widetilde{u})\dif(\widetilde{u})
    - \frac{4}{|y|^2} \right) \widetilde{u}
    \, .
\end{aligned}
\end{equation}
On the other hand, one could calculate the Euler-Lagrange equations for $u$ by varying \eqref{def:dir} yielding the well-known harmonic map equation
\begin{align*}
    -\Delta u
    \ &= \
    |\nabla u|^2 u
    \, .
\end{align*}
Reflecting this PDE across $\partial\Omega$ and rewriting it in terms of $\widetilde{u}$, we also obtain \eqref{rem:ELE_tilde_n}.
In other words, the following diagram commutes:
\[
\begin{xy}
\xymatrix@=4em{ 
    E \ar[r]^{\widetilde{u}=A(u\circ\iota)} \ar[d]_{\delta} & \widetilde{E} \ar[d]^{\delta}  \\
    \mathrm{ELE} \ar[r]_{\widetilde{u}=A(u\circ\iota)} &\widetilde{\mathrm{ELE}}
}
\end{xy}
\]
\end{remark}

Next we show that the PDEs satisfied on $\Omega$ and
$B_{2}(0)\setminus\Omega$ ``glue" together.
As a result, we establish $\overline{u}$ satisfies a PDE on $B_{2}(0)$.
The novelty here is that the test functions no longer need to satisfy
the tangential requirements of $W_{T}^{1,2}(\Omega;\mathbb{S}^{2})$
or $\mathcal{W}_{u}$.
The approach to proving this is to appeal to the reflective symmetry
across $\partial\Omega$ and the Change of Variables Theorem as in
\cite{BCS, DiMiPi, Sch}.
This method is elegant as it evades the technical problem of defining the normal derivative of $u$ along $\partial\Omega$.
Note that our minimization problem involving the Dirichlet energy,
\eqref{def:dir}, only gives
meaning to $\partial_{\nu}u_{T}$ as noted in \eqref{eq:MinPDE}.
\begin{lemma}\label{lem:global_PDE}
	Suppose $u$ is a minimizer for $\eqref{def:dir}$ and $\overline{u}$
    is the extension to $B_{2}(0)$ defined in \eqref{eq:extension}.
	Then, $\overline{u}$ satisfies
	\begin{align*}
		&\int_{\Omega}\!{}\nabla\overline{u}(x):\nabla\varphi(x)
		+\int_{B_{2}(0)\setminus\Omega}\!{}
        \frac{1}{|x|^{2}}\nabla\overline{u}(x):\nabla\varphi(x)\\
		=&\int_{\Omega}\!{}|\nabla\overline{u}(x)|^{2}
        \overline{u}(x)\cdot\varphi(x)
		+\int_{B_{2}(0)\setminus\Omega}\!{}\frac{1}{|x|^{2}}
        \biggl[|\nabla\overline{u}(x)|^{2}\overline{u}(x)
		-\frac{4}{|x|^{4}}(\overline{u}(x)\cdot{}x)x^{T}
		+\frac{4}{|x|^{4}}(\overline{u}(x)\cdot{}x)^{2}\overline{u}(x)^{T}
		\\
		&-\frac{4}{|x|^{2}}x^{T}\nabla\overline{u}(x)
		+\frac{4}{|x|^{2}}\bigl[\overline{u}(x)^{T}\nabla\overline{u}(x)^{T}
        x\bigr]\overline{u}(x)^{T}
		+\frac{4}{|x|^{2}}{\rm{div}}(\overline{u})(x)x^{T}\\
		&-\frac{4}{|x|^{2}}(\overline{u}(x)\cdot{}x)
        {\rm{div}}(\overline{u})(x)\overline{u}(x)^{T}
		\biggr]\varphi(x)
	\end{align*}
	where $\varphi\in{}W_{0}^{1,2}(B_{2}(0);\mathbb{R}^{3})\cap{}L^{\infty}(B_{2}(0);\mathbb{R}^{3})$.
\end{lemma}

\begin{proof}
	We may suppose that $\text{supp}(\varphi)\subseteq{}B_{2}(0)\setminus{}B_{\frac{1}{2}}(0)$ since we could otherwise use a partition of unity
    subordinate to the cover of $B_{2}(0)$ given by
    $\{B_{2}(0)\setminus{}B_{\frac{1}{2}}(0),B_{\frac{3}{4}}(0)\}$.
	We rewrite $\varphi$ as
	\begin{equation*}
		\varphi=\varphi_{e}+\varphi_{o}
	\end{equation*}
	where $\varphi_{e}(x)\coloneqq\frac{1}{2}[\varphi(x)+A(x)\varphi(\iota(x))]$ and
	$\varphi_{o}(x)\coloneqq\frac{1}{2}[\varphi(x)-A(x)\varphi(\iota(x))]$.
	We also observe that $\varphi_{e}\in{}W_{T}^{1,2}(\Omega,\mathbb{R}^{3})\cap{}\mathcal{W}_{u}$.
	From this observation, \eqref{eq:weakform}, and Corollary \ref{cor:outerpde} we conclude that
	\begin{align*}
		\int_{\Omega}\!{}\nabla\overline{u}:\nabla\varphi
		+\int_{B_{2}(0)\setminus\Omega}\!{}\frac{1}{|x|^{2}}\nabla\overline{u}:\nabla\varphi
		&=\int_{\Omega}\!{}\nabla\overline{u}:\nabla\varphi_{o}+\int_{B_{2}(0)\setminus\Omega}\!{}\frac{1}{|x|^{2}}\nabla\overline{u}:\nabla\varphi_{o}
		+\int_{\Omega}\!{}|\nabla\overline{u}|^{2}\overline{u}\cdot\varphi_{e}\\
		+&\int_{B_{2}(0)\setminus\Omega}\!{}\frac{1}{|x|^{2}}\biggl[|\nabla\overline{u}|^{2}\overline{u}
		-\frac{4}{|x|^{4}}(\overline{u}(x)\cdot{}x)x^{T}
		+\frac{4}{|x|^{4}}(\overline{u}(x)\cdot{}x)^{2}\overline{u}(x)^{T}
		\\
		-&\frac{4}{|x|^{2}}x^{T}\nabla\overline{u}(x)
		+\frac{4}{|x|^{2}}\bigl[\overline{u}(x)^{T}\nabla\overline{u}(x)^{T}x\bigr]\overline{u}(x)^{T}
		+\frac{4}{|x|^{2}}\text{div}(\overline{u})(x)x^{T}\\
		-&\frac{4}{|x|^{2}}(\overline{u}(x)\cdot{}x)\text{div}(\overline{u})(x)\overline{u}(x)^{T}
		\biggr]\varphi_{e}(x).
	\end{align*}
	Using \eqref{eq:invol} and that
    $\varphi_{o}(x)=-A(x)\varphi_{o}(\iota(x))$ we obtain
	\begin{equation*}
		\int_{B_{2}(0)\setminus\Omega}\!{}\frac{1}{|x|^{2}}
        \nabla\overline{u}(x):\nabla\varphi_{o}(x)
		=
		-\int_{B_{2}(0)\setminus\Omega}\!{}\frac{1}{|x|^{2}}\nabla{}\bigl[A(x)\overline{u}(\iota(x))\bigr]:
		\nabla\bigl[A(x)\varphi_{o}(\iota(x))\bigr].
	\end{equation*}
	Using \eqref{eq:partialA} and \eqref{eq:partialinv} we find that
	\begin{align*}
	\partial_{x_{i}}\bigl[A(x)\overline{u}(\iota(x))\bigr]
	=&\biggl[-\frac{2}{|x|^{2}}\bigl[\mathbf{e}_{i}x^{T}+x\mathbf{e}_{i}^{T}\bigr]+\frac{4x_{i}}{|x|^{4}}xx^{T}\biggr]
	\overline{u}(\iota(x))\\
	&+A(x)\biggl[\frac{1}{|x|^{2}}\partial_{x_{i}}\overline{u}(\iota(x))
	-\frac{2x_{i}}{|x|^{4}}\nabla{}\overline{u}(\iota(x))x\biggr]
	\end{align*}
	and similarly for $A(x)\varphi_{o}(\iota(x))$.
    Hence,
	\begin{align*}
		\nabla\bigl[A(x)
        \overline{u}(\iota(x))\bigr]:
        &\nabla\bigl[A(x)\varphi_{o}(\iota(x))\bigr]
		=\frac{1}{|x|^{4}}\nabla\overline{u}(\iota(x)):
        \nabla\varphi_{0}(\iota(x))\\
		&+\frac{2}{|x|^{4}}\varphi_{o}(\iota(x))^{T}\nabla{}\overline{u}(\iota(x))^{T}x
		-\frac{2}{|x|^{4}}\text{div}(\overline{u})(\iota(x))\bigl(x\cdot\varphi_{o}(\iota(x))\bigr)\\
		&-\frac{2}{|x|^{4}}\bigl(x\cdot{}\overline{u}(\iota(x))\bigr)
		\text{div}(\varphi_{o})(\iota(x))
		+\frac{2}{|x|^{4}}\overline{u}(\iota(x))^{T}\nabla\varphi_{o}(\iota(x))^{T}x\\
		&+\frac{4}{|x|^{4}}\bigl(x\cdot{}\overline{u}(\iota(x))\bigr)\bigl(x\cdot{}\varphi_{o}(\iota(x))\bigr)
		+\frac{4}{|x|^{2}}\overline{u}(\iota(x))\cdot\varphi_{o}(\iota(x)).
	\end{align*}
	From this we see, using that $\varphi_{0}$ is supported in $B_{2}(0)\setminus{}B_{\frac{1}{2}}(0)$ as well as  the change of variables
	$y=\iota(x)$ and \eqref{eq:covdet}, that
	\begin{align*}
		-\int_{B_{2}(0)\setminus\Omega}\!{}\frac{1}{|x|^{2}}\nabla{}\bigl[A(x)\overline{u}(\iota(x))\bigr]:&
		\nabla\bigl[A(x)\varphi_{o}(\iota(x))\bigr]
		=-\int_{\Omega}\!{}\nabla{}\overline{u}(y):\nabla\varphi_{o}(y)\\
		&-2\int_{\Omega}\!{}\varphi_{o}(y)^{T}\nabla{}\overline{u}(y)^{T}\iota(y)
		+2\int_{\Omega}\!{}\text{div}(\overline{u}(y))\bigl(\iota(y)\cdot\varphi_{o}(y)\bigr)\\
		&+2\int_{\Omega}\!{}\text{div}(\varphi_{o}(y))\bigl(\iota(y)\cdot{}\overline{u}(y)\bigr)
		-2\int_{\Omega}\!{}\overline{u}(y)^{T}\nabla\varphi_{o}(y)^{T}\iota(y)\\
		&-4\int_{\Omega}\!{}\bigl(\iota(y)\cdot{}\overline{u}(y)\bigr)\bigl(\iota(y)\cdot{}\varphi_{o}(y)\bigr)
		-4\int_{\Omega}\!{}\frac{1}{|y|^{2}}\overline{u}(y)\cdot\varphi_{o}(y).
	\end{align*}
	Integrating by parts twice and using that $\overline{u}$ has no normal
    part on $\partial\Omega$ leads to
	\begin{align*}
		-\int_{B_{2}(0)\setminus\Omega}\!{}\frac{1}{|x|^{2}}\nabla{}\bigl[A(x)\overline{u}(\iota(x))\bigr]:
		\nabla\bigl[A(x)\varphi_{o}(\iota(x))\bigr]
		=&-\int_{\Omega}\!{}\nabla{}\overline{u}(y):\nabla\varphi_{o}(y)-4\int_{\Omega}\!{}\frac{1}{|y|^{2}}\overline{u}(y)\cdot\varphi_{o}(y)\\
        &-4\int_{\Omega}\!{}\bigl(\iota(y)\cdot{}\overline{u}(y)\bigr)\bigl(\iota(y)\cdot{}\varphi_{o}(y)\bigr)\\
        &-4\int_{\Omega}\!{}\varphi_{o}(y)^{T}\nabla{}\overline{u}(y)^{T}\iota(y)
		+4\int_{\Omega}\!{}\text{div}(\overline{u}(y))\bigl(\iota(y)\cdot\varphi_{o}(y)\bigr).
	\end{align*}
	Next using \eqref{eq:invol}, \eqref{eq:partialA}, and
    \eqref{eq:partialinv} and making the change of variables $y=\iota(x)$
    we obtain
	\begin{align*}
		\int_{B_{2}(0)\setminus\Omega}\!{}\frac{1}{|x|^{2}}|\nabla\overline{u}(x)|^{2}\overline{u}(x)\cdot\varphi_{e}(x)
		=&\frac{1}{2}\int_{B_{2}(0)\setminus\Omega}\!{}\frac{1}{|x|^{2}}|\nabla\overline{u}(x)|^{2}\overline{u}(x)\cdot\varphi(x)\\
		&+\frac{1}{2}\int_{\Omega}\!{}|\nabla\overline{u}(y)|^{2}\overline{u}(y)\cdot\varphi(y)\\
		&+2\int_{\Omega}\!{}\frac{1}{|y|^{4}}(y\cdot\overline{u}(y))^{2}(\overline{u}(y)\cdot\varphi(y))
		+2\int_{\Omega}\!{}\frac{1}{|y|^{2}}\overline{u}(y)\cdot\varphi(y)\\
		&+2\int_{\Omega}\!{}\frac{1}{|y|^{2}}\bigl[\overline{u}(y)^{T}\nabla\overline{u}(y)^{T}y\bigr]\overline{u}(y)\cdot\varphi(y)\\
		&-2\int_{\Omega}\!{}\frac{1}{|y|^{2}}(y\cdot\overline{u}(y))\text{div}(\overline{u}(y))(\overline{u}(y)\cdot\varphi(y)).
	\end{align*}
	A similar calculation gives
	\begin{align*}
		\frac{1}{2}\int_{\Omega}\!{}
		|\nabla\overline{u}|^{2}
		\overline{u}\cdot{}A(x)\varphi(\iota(y))
		=&\frac{1}{2}\int_{B_{2}(0)\setminus\Omega}\!{}\frac{1}{|x|^{2}}|\nabla\overline{u}(x)|^{2}\overline{u}(x)\cdot\varphi(x)\\
		&+2\int_{B_{2}(0)\setminus\Omega}\!{}\frac{1}{|x|^{6}}(x\cdot\overline{u}(x))^{2}(\overline{u}(x)\cdot\varphi(x))
		+2\int_{B_{2}(0)\setminus\Omega}\!{}\frac{1}{|x|^{4}}\overline{u}(x)\cdot\varphi(x)\\
		&+2\int_{B_{2}(0)\setminus\Omega}\!{}\frac{1}{|x|^{4}}\bigl[\overline{u}(x)^{T}\nabla\overline{u}(x)^{T}x\bigr]
		(\overline{u}(x)\cdot\varphi(x))\\
		&-2\int_{B_{2}(0)\setminus\Omega}\!{}\frac{1}{|x|^{4}}(x\cdot\overline{u}(x))\text{div}(\overline{u}(x))(\overline{u}(x)\cdot\varphi(x)).
	\end{align*}
	which leads to
	\begin{align*}
		\int_{\Omega}\!{}\nabla\overline{u}:\nabla\varphi
		+\int_{B_{2}(0)\setminus\Omega}\!{}\frac{1}{|x|^{2}}\nabla\overline{u}:\nabla\varphi
		=&2\int_{\Omega}\!{}\frac{1}{|y|^{2}}\overline{u}(y)\cdot{}A(y)\varphi(\iota(y))
		+\int_{\Omega}\!{}|\nabla\overline{u}|^{2}\overline{u}\cdot\varphi\\
		&+\int_{B_{2}(0)\setminus\Omega}\!{}\frac{1}{|x|^{2}}|\nabla\overline{u}(x)|^{2}\overline{u}(x)\cdot\varphi(x)\\
		&+\int_{B_{2}(0)\setminus\Omega}\!{}\frac{1}{|x|^{2}}\biggl[
		-\frac{4}{|x|^{4}}(\overline{u}(x)\cdot{}x)x^{T}
		+\frac{4}{|x|^{4}}(\overline{u}(x)\cdot{}x)^{2}\overline{u}(x)^{T}
		\\
		&-\frac{4}{|x|^{2}}x^{T}\nabla\overline{u}(x)
		+\frac{4}{|x|^{2}}\bigl[\overline{u}(x)^{T}\nabla\overline{u}(x)^{T}x\bigr]\overline{u}(x)^{T}
		+\frac{4}{|x|^{2}}\text{div}(\overline{u})(x)x^{T}\\
		&-\frac{4}{|x|^{2}}(\overline{u}(x)\cdot{}x)\text{div}(\overline{u})(x)\overline{u}(x)^{T}
		\biggr]\varphi_{e}(x)\\
		&+2\int_{\Omega}\!{}\frac{1}{|y|^{4}}(y\cdot\overline{u}(y))^{2}(\overline{u}(y)\cdot\varphi(y))\\
		&+2\int_{\Omega}\!{}\frac{1}{|y|^{2}}\bigl[\overline{u}(y)^{T}\nabla\overline{u}(y)^{T}y\bigr]\overline{u}(y)\cdot\varphi(y)\\
		&-2\int_{\Omega}\!{}\frac{1}{|y|^{2}}(y\cdot\overline{u}(y))\text{div}(\overline{u}(y))(\overline{u}(y)\cdot\varphi(y))\\
		&+2\int_{B_{2}(0)\setminus\Omega}\!{}\frac{1}{|x|^{6}}(x\cdot\overline{u}(x))^{2}(\overline{u}(x)\cdot\varphi(x))
		+2\int_{B_{2}(0)\setminus\Omega}\!{}\frac{1}{|x|^{4}}\overline{u}(x)\cdot\varphi(x)\\
		&+2\int_{B_{2}(0)\setminus\Omega}\!{}\frac{1}{|x|^{4}}\bigl[\overline{u}(x)^{T}\nabla\overline{u}(x)^{T}x\bigr]
		(\overline{u}(x)\cdot\varphi(x))\\
		&-2\int_{B_{2}(0)\setminus\Omega}\!{}\frac{1}{|x|^{4}}(x\cdot\overline{u}(x))\text{div}(\overline{u}(x))(\overline{u}(x)\cdot\varphi(x))\\
        &-4\int_{\Omega}\!{}\varphi_{o}^{T}(y)\nabla\overline{u}(y)^{T}i(y)
        +4\int_{\Omega}\!{}\text{div}(\overline{u})(y)(i(y)\cdot\varphi_{o}(y))\\
        &-4\int_{\Omega}\!{}(i(y)\cdot\overline{u}(y))
        (i(y)\cdot{}\varphi_{o}(y)).
	\end{align*}
	Using \eqref{eq:invol}, making the change of variables
    $x=\iota(y)$, and using that $\varphi$ is supported in
    $B_{2}(0)\setminus{}B_{\frac{1}{2}}(0)$ leads to
	\begin{equation}\label{eq:extpde_cov1}
		2\int_{\Omega}\!{}\frac{1}{|y|^{2}}\overline{u}(y)\cdot{}A(y)\varphi(\iota(y))
		=2\int_{B_{2}(0)\setminus\Omega}\!{}\frac{1}{|x|^{4}}
        \overline{u}(x)\cdot\varphi(x).
	\end{equation}
	We note the following identities:
	\begin{align*}
		\overline{u}(\iota(x))^{T}\nabla\overline{u}(\iota(x))^{T}\iota(x)
		&=-2+\frac{2}{|x|^{2}}(x\cdot\overline{u}(x))^{2}-x^{T}\nabla\overline{u}(x)\overline{u}(x)\\
		\text{div}(\overline{u})
        	(\iota(x))&=-4(x\cdot\overline{u}(x))+|x|^{2}\text{div}(\overline{u}(x)).
	\end{align*}
	Using the above identities, the change of coordinates $y=\iota(x)$,
    equation \eqref{eq:covdet}, equation \eqref{eq:gradinversion} applied
    with $\overline{u}$, and that the support of $\varphi$ is in
    $B_{2}(0)\setminus{}B_{\frac{1}{2}}(0)$ we obtain
	\begin{align}
		2\int_{\Omega}\!{}\frac{1}{|y|^{4}}(y\cdot\overline{u}(y))^{2}(\overline{u}(y)\cdot\varphi(y))&=
		2\int_{B_{2}(0)\setminus\Omega}\!{}\frac{1}{|x|^{6}}\!{}(x\cdot\overline{u}(x))^{2}
		\biggl(\overline{u}(x)\cdot{}A(x)\varphi(\iota(x))\biggr)
        \label{eq:extpde_cov2}\\
		2\int_{\Omega}\!{}\frac{1}{|y|^{2}}\bigl[\overline{u}(y)^{T}\nabla\overline{u}(y)^{T}y\bigr]\overline{u}(y)\cdot\varphi(y)=&
		2\int_{B_{2}(0)\setminus\Omega}\!{}\frac{1}{|x|^{4}}\biggl(\frac{2}{|x|^{2}}(x\cdot\overline{u}(x))^{2}-2\biggr)
		\biggl(\overline{u}(x)\cdot{}A(x)\varphi(\iota(x))\biggr)
        \nonumber\\
		&-2\int_{B_{2}(0)\setminus\Omega}\!{}\frac{1}{|x|^{4}}[x^{T}\nabla\overline{u}(x)\overline{u}(x)]
		\biggl(\overline{u}(x)\cdot{}A(x)\varphi(\iota(x))\biggr)
        \label{eq:extpde_cov3}\\
		-2\int_{\Omega}\!{}\frac{1}{|y|^{2}}(y\cdot\overline{u}(y))\text{div}(\overline{u}(y))(\overline{u}(y)\cdot\varphi(y))=&
		-8\int_{B_{2}(0)\setminus\Omega}\!{}\frac{1}{|x|^{6}}(x\cdot\overline{u}(x))^{2}
		\biggl(\overline{u}(x)\cdot{}A(x)\overline{u}(\iota(x))\biggr)
        \nonumber\\
		&+2\int_{B_{2}(0)\setminus\Omega}\!{}\frac{1}{|x|^{4}}(x\cdot\overline{u}(x))\text{div}(\overline{u}(x))
		\biggl(\overline{u}(x)\cdot{}A(x)\varphi(\iota(x))\biggr).
        \label{eq:extpde_cov4}
	\end{align}
    Finally, making the change of variables $x=i(y)$ and using
    \eqref{eq:covdet} leads to
    \begin{align}
        -4\int_{\Omega}\!{}
        (i(y)\cdot\overline{u}(y))(i(y)\cdot\varphi_{o}(y))
        &=4\int_{B_{2}(0)\setminus\Omega}\!{}\frac{1}{|x|^{6}}
        (x\cdot\overline{u}(x))(x\cdot\varphi_{o}(x)),
        \label{eq:intpde_cov1}\\
        4\int_{\Omega}\!{}\text{div}(\overline{u})(y)
        (i(y)\cdot\varphi_{o}(y))
        =&-16\int_{B_{2}(0)\setminus\Omega}\!{}\frac{1}{|x|^{6}}
        (x\cdot\overline{u}(x))(x\cdot\varphi_{o}(x))
        \label{eq:intpde_cov2}\\
        &+4\int_{B_{2}(0)\setminus\Omega}\!{}\frac{1}{|x|^{4}}
        \text{div}(\overline{u})(x)(x\cdot\varphi_{o}(x)),\nonumber\\
        -4\int_{B_{2}(0)\setminus\Omega}\!{}\varphi_{o}(y)^{T}
        \nabla\overline{u}(y)^{T}i(y)
        =&8\int_{B_{2}(0)\setminus\Omega}\!{}\frac{1}{|x|^{4}}
        (x\cdot\overline{u}(x))(x\cdot\varphi_{o}(x))
        \label{eq:intpde_cov3}\\
        &-8\int_{B_{2}(0)\setminus\Omega}\!{}\frac{1}{|x|^{4}}
        (\overline{u}(x)\cdot\varphi_{o}(x))
        -4\int_{B_{2}(0)\setminus\Omega}\!{}\varphi_{o}(x)^{T}
        \nabla\overline{u}(x)^{T}x.\nonumber
    \end{align}
	Applying identities \eqref{eq:extpde_cov1}, \eqref{eq:extpde_cov2},
    \eqref{eq:extpde_cov3}, \eqref{eq:extpde_cov4}, \eqref{eq:intpde_cov1},
    \eqref{eq:intpde_cov2}, and \eqref{eq:intpde_cov3} we obtain:
	\begin{align*}
		&\int_{\Omega}\!{}\nabla\overline{u}(x):\nabla\varphi(x)
		+\int_{B_{2}(0)\setminus\Omega}\!{}\frac{1}{|x|^{2}}\nabla\overline{u}(x):\nabla\varphi(x)\\
		=&\int_{\Omega}\!{}|\nabla\overline{u}(x)|^{2}\overline{u}(x)\cdot\varphi(x)
		+\int_{B_{2}(0)\setminus\Omega}\!{}\frac{1}{|x|^{2}}\biggl[|\nabla\overline{u}(x)|^{2}\overline{u}(x)
		-\frac{4}{|x|^{4}}(\overline{u}(x)\cdot{}x)x^{T}
		+\frac{4}{|x|^{4}}(\overline{u}(x)\cdot{}x)^{2}\overline{u}(x)^{T}
		\\
		&-\frac{4}{|x|^{2}}x^{T}\nabla\overline{u}(x)
		+\frac{4}{|x|^{2}}\bigl[\overline{u}(x)^{T}\nabla\overline{u}(x)^{T}x\bigr]\overline{u}(x)^{T}
		+\frac{4}{|x|^{2}}\text{div}(\overline{u})(x)x^{T}
		-\frac{4}{|x|^{2}}(\overline{u}(x)\cdot{}x)\text{div}(\overline{u})(x)\overline{u}(x)^{T}
		\biggr]\varphi(x)
	\end{align*}
\end{proof}

%

\subsection{Tangential Inner Variation}
\label{subsec:tang_inner_var}

In this subsection, we show that the extension satisfies an inner variation
formula similar to the one demonstrated in \cite{Ev2}.
However, in order to account for behaviour near the boundary, we modified
the family of diffeomorphisms that are considered.
This has the effect of adding some additional error terms.\\

We begin by introducing a suitable family of diffeomorphisms which
separately preserves the interior and exterior domains.
\begin{definition}[Tangential Inner Variation]\label{def:innervartan}
    For a fixed $\delta>0$ we consider a function
    $\varphi\colon{}B_{2}(0)\times[0,\delta]\to{}B_{2}(0)$ which is $C^{3}$ and,
    for each $t\in[0,\delta]$, $x\mapsto\varphi(x,t)$ is a diffeomorphism of
    $B_{2}(0)$ satisfying
    \begin{enumerate}
    	\item\label{eq:preserveint}
    		$\varphi(\Omega\times\{t\})=\Omega\text{ for each }t\in[0,\delta]$,
    	\item\label{eq:idatzero}
    		$\varphi(\cdot,0)\coloneqq{}\text{id}_{B_{2}(0)}(\cdot)$,
    	\item\label{eq:compactperturb}
    		$x\mapsto\varphi(x,t)-\text{id}_{B_{2}(0)}(x)$ has compact support in $B_{2}(0)$ for all  $t\in[0,\delta]$.
    \end{enumerate}
    As we will be most interested in this family with fixed $t\in[0,\delta]$ we introduce the notation
    \begin{equation*}
    	\varphi_{t}(\cdot)\coloneqq\varphi(\cdot,t).
    \end{equation*}
    In addition, since $\varphi_{t}$ is a diffeomorphism for each $t\in[0,\delta]$ we let $\psi_{t}\coloneqq\varphi_{t}^{-1}$.
    By assumption, $t\mapsto\varphi_{t}(x)$ is $C^{3}$ for each $x\in{}B_{2}(0)$.
    As a result, we can define $T\colon{}B_{2}(0)\to{}B_{2}(0)$ by
    \begin{equation}\label{def:infinnervar}
    	T(x)\coloneqq\frac{\mathrm{d}}{\mathrm{d}t}\bigg\vert_{t=0}\varphi_{t}(x).
    \end{equation}
\end{definition}
Since $\nabla_{x}\varphi_{t}(x)\in{}GL_{3}(\mathbb{R})$ for all $x\in{}B_{2}(0)$ and $t\in[0,\delta]$, and $|\overline{u}|=1$ for
$\mathcal{L}^{3}$-almost every point in $B_{2}(0)$ then, by Sard's Theorem,
we can define, for each
$t\in[0,\delta]$, $\overline{u}_{t}\colon{}B_{2}(0)\to\mathbb{S}^{2}$ for $\mathcal{L}^{3}$-almost every point by
\begin{equation*}
	\overline{u}_{t}(y)\coloneqq\frac{\nabla_{x}\varphi_{t}(\psi_{t}(y))\overline{u}(\psi_{t}(y))}
	{|\nabla_{x}\varphi_{t}(\psi_{t}(y))\overline{u}(\psi_{t}(y))|}
\end{equation*}
where we use the coordinate $y$ to emphasize that we are sampling from the domain of $\psi_{t}$.
We now demonstrate that
\begin{equation*}
    \{\overline{u}_{t}\}_{0<t<\delta}\coloneqq
    \{\overline{u}_{t}:0<t<\delta\}
\end{equation*}
defines a family of test functions within both
$W_{T}^{1,2}(\Omega;\mathbb{S}^{2})$ and $\mathcal{W}_{u}$.
This will be necessary in order to use minimality of $u$ and $\widetilde{u}$
to $E$ and $\widetilde{E}$ respectively.

\begin{lemma}\label{lem:tang-var-admissible}
    Suppose $u\colon\Omega\to\mathbb{S}^{2}$ is a minimizer of
    \eqref{def:dir} over $W_{T}^{1,2}(\Omega;\mathbb{S}^{2})$ and
    $\overline{u}$ is defined as in \eqref{eq:extension}.
    Then for $\delta>0$ we have
    \begin{equation*}
        \{\overline{u}_{t}\}_{0<t<\delta}
        \subseteq{}W_{T}^{1,2}(\Omega;\mathbb{S}^{2})\cap{}
        \mathcal{W}_{u}
    \end{equation*}
    for all $t\in[0,\delta]$ and $\overline{u}_{0}=\overline{u}$
    for $\mathcal{L}^{3}-$almost every point in $B_{2}(0)$.
\end{lemma}

\begin{proof}
    Notice that when $t=0$ we have, since $\nabla_{x}\varphi_{0}=I_{3}$ and $\psi_{0}=\text{id}_{\mathbb{R}^{3}}$ that
    \begin{equation*}
    	\overline{u}_{0}(y)=\frac{\overline{u}(y)}{|\overline{u}(y)|}=\overline{u}(y)
    \end{equation*}
    for all points where $|\overline{u}|=1$.
    In particular, this holds for $\mathcal{L}^{3}$-almost every point in $B_{2}(0)$.
    It is shown in Lemma A.1 of Appendix A of \cite{KiMaSt}, for $x\in\partial\Omega$ and $t\ge0$, that
    \begin{equation*}
    	\nu_{\partial[\varphi_{t}(\Omega)]}(\varphi_{t}(x))=
    	\frac{\nabla_{x}\psi_{t}(\varphi_{t}(x))^{T}\nu_{\partial\Omega}(x)}{|\nabla_{x}\psi_{t}(\varphi_{t}(x))^{T}\nu_{\partial\Omega}(x)|}
    \end{equation*}
    where $\nu_{\partial\Omega}(x)$ denotes the outward unit normal to $\partial\Omega$ at $x\in\partial\Omega$ and similarly for
    $\nu_{\partial[\varphi_{t}(\Omega)]}$.
    We verify that $\overline{u}_{t}$ preserves the tangentiality condition along $\partial\Omega$.
    Observe that for $x\in\Omega$ and $t\in[0,\delta]$ we have,
    since $\varphi_{t}(\Omega)=\Omega$ for all $0\le{}t\le\frac{\delta}{2}$,
    \begin{align}
    	\overline{u}_{t}(y)\cdot\nu_{\partial\Omega}(y)
        &=\overline{u}_{t}(y)\cdot\nu_{\partial[\varphi_{t}(\Omega)]}(y) \nonumber \\
    	&=\biggl[\frac{\nabla_{x}\varphi_{t}(\psi_{t}(y))\overline{u}(\psi_{t}(y))}
    	{|\nabla_{x}\varphi_{t}(\psi_{t}(y))\overline{u}(\psi_{t}(y))|}\biggr]\cdot
    	\biggl[\frac{\nabla_{x}\psi_{t}(y)^{T}\nu_{\partial\Omega}(\psi_{t}(y))}{|\nabla_{x}\psi_{t}(y)^{T}\nu_{\partial\Omega}(\psi_{t}(y))|}\biggr] \nonumber \\
    	&=\frac{\overline{u}(\psi_{t}(y))^{T}\nabla_{x}\varphi_{t}(\psi_{t}(y))^{T}\nabla_{x}\psi_{t}(y)^{T}\nu_{\partial\Omega}(\psi_{t}(y))}
    	{|\nabla_{x}\varphi_{t}(\psi_{t}(y))\overline{u}(\psi_{t}(y))||\nabla_{x}\psi_{t}(y)\nu_{\partial\Omega}(\psi_{t}(y))|} \label{lem:tang-var-admissible:eq} \\
    	&=\frac{\overline{u}(\psi_{t}(y))^{T}\nu_{\partial\Omega}(\psi_{t}(y))}
    	{|\nabla_{x}\varphi_{t}(\psi_{t}(y))\overline{u}(\psi_{t}(y))||\nabla_{x}\psi_{t}(y)\nu_{\partial\Omega}(\psi_{t}(y))|} \nonumber \\
    	&=0 \, , \nonumber
    \end{align}
    where we have applied the Inverse Function Theorem as well as the assumption that $u$ satisfies \eqref{eq:tanbc}.
    We conclude that $\overline{u}_{t}$ satisfies \eqref{eq:tanbc} for all
    $t\in[0,\delta]$.
    Next, observe that since, for each $t\in[0,\delta]$, $\varphi_{t}\in{}C^{3}$ then by the Inverse Function Theorem we have that
    $\psi_{t}\in{}C^{3}$.
    From this it follows that $\overline{u}_{t}\in{}W^{1,2}(\Omega;\mathbb{S}^{2})$ for each $t\in[0,\delta]$ and hence
    $\overline{u}_{t}\in{}W_{T}^{1,2}(\Omega;\mathbb{S}^{2})$ for each $t\in[0,\delta]$.

    Finally, notice that since $\psi_{t}(y)=y$ for all $y$ in a
    neighbourhood of $\partial{}B_{2}(0)$ and $t\in[0,\delta]$ then $\nabla_{y}\psi_{t}(y)=I_{3}$ for all
    $y\in\partial{}B_{2}(0)$ and $t\in[0,\delta]$ and so $\overline{u}_{t}$ satisfies
    \begin{equation*}
    	\overline{u}_{t}(y)=\frac{\overline{u}(y)}{|\overline{u}(y)|}=\overline{u}(y)
    \end{equation*}
    for all $y\in{}\partial{}B_{2}(0)$ and $t\in[0,\delta]$.
    From this we conclude that $\overline{u}_{t}\in\mathcal{W}_{u}$ for all $t\in[0,\delta]$.
\end{proof}

Next, we use the admissible family of functions
$\{\overline{u}_{t}\}_{0<t<\delta}$ in order to calculate an inner
variation of the extended energy $\overline{E}$.
\begin{proposition}[Tangential Inner Variations]\label{prop:inner_variations_tang}
    Suppose $u\colon\Omega\to\mathbb{S}^{2}$ is a minimizer of
    \eqref{def:dir} and $\overline{u}$ is defined as in \eqref{eq:extension}.
    Suppose also that $T$ is given by \eqref{def:infinnervar} where
    $\varphi_{t}$ is a family defined as in Definition \ref{def:innervartan}.
    Then there are functions $\mathcal{F}$ and $\mathcal{G}$ such that
\begin{align*}
	0
	=&\int_{\Omega}\!{}\biggl\{-\nabla_{x}u(x):\nabla_{x}T(x)\nabla_{x}\overline{u}(x)+\frac{1}{2}|\nabla{}\overline{u}|^{2}\dif(T(x))\biggr\}\\
	&-\int_{B_{2}(0)\setminus\Omega}\!{}\frac{1}{|x|^{2}}\nabla_{x}\overline{u}(x):\nabla_{x}T(x)\nabla_{x}\overline{u}(x))
	+\frac{1}{2}\int_{B_{2}(0)\setminus\Omega}
	\frac{1}{|x|^{2}}|\nabla_{x}\overline{u}(x)|^{2}\dif(T(x))\\
	&+\int_{B_{2}(0)\setminus\Omega}\!{}\mathcal{F}(x,\overline{u}(x),\nabla_{x}\overline{u}(x))\cdot{}T(x)
	+\int_{B_{2}(0)\setminus\Omega}\!{}\mathcal{G}(x,\overline{u}(x),\nabla_{x}\overline{u}(x)):\nabla_{x}T(x)
\end{align*}
and
\begin{equation*}
	|\mathcal{F}(x,z,p)|\le{}C(1+|p|^{2}),\hspace{20pt}|\mathcal{G}(x,z,p)|\le{}C(1+|p|)
    \, .
\end{equation*}
\end{proposition}

\begin{proof}
    In order to prove the desired formula we compute
    \begin{equation*}
        \frac{\mathrm{d}}{\mathrm{d}t}\bigg|_{t=0}
        \overline{E}(\overline{u})
        =\frac{\mathrm{d}}{\mathrm{d}t}\bigg|_{t=0}E(u)+
        \frac{\mathrm{d}}{\mathrm{d}t}\bigg|_{t=0}
        \widetilde{E}(\widetilde{u}).
    \end{equation*}
    To compute this we consider separately the energy over $\Omega$ and over
    $B_{2}(0)\setminus\Omega$.
    We begin with the interior portion of the energy.\\

    By minimality of $u$ in $W_{T}^{1,2}(\Omega;\mathbb{S}^{2})$ we have
    \begin{equation*}
    	0=\frac{\mathrm{d}}{\mathrm{d}t}\bigg\vert_{t=0}E(\overline{u}_{t})
    	  =\frac{\mathrm{d}}{\mathrm{d}t}\bigg\vert_{t=0}\int_{\Omega}\!{}\frac{1}{2}|\nabla_{y}\overline{u}_{t}(y)|^{2}\mathrm{d}y.
    \end{equation*}
    We notice that by an application of the Change of Variables Theorem,
    the assumption that $\varphi_{t}(\Omega)=\Omega$, the Inverse Function
    Theorem, as well as the Chain rule we have
    \begin{align*}
    	\frac{1}{2}\int_{\Omega}\!{}|\nabla_{y}\overline{u}_{t}|^{2}
    	&=\frac{1}{2}\int_{\Omega}\!{}|\nabla_{x}\bigl(\overline{u}_{t}(\varphi_{t}(x))\bigr)\nabla_{y}\psi_{t}(\varphi_{t}(x))|^{2}|\det(\nabla_{x}\varphi_{t}(x))|\\
    	&=\frac{1}{2}\int_{\Omega}\!{}\biggl|\nabla_{x}\biggl(
    	\frac{\nabla_{x}\varphi_{t}(x)u(x)}{|\nabla_{x}\varphi_{t}(x)u(x)|}\biggr)\nabla_{y}\psi_{t}(\varphi_{t}(x))\biggr|^{2}
    	\det(\nabla_{x}\varphi_{t}(x)).
    \end{align*}
    Notice that differentiating $\det(\nabla_{x}\varphi_{t}(x))$ with respect to $t$ and evaluating at $t=0$ gives
    \begin{equation*}
    	\frac{\mathrm{d}}{\mathrm{d}t}\bigg\vert_{t=0}\det(\nabla_{x}\varphi_{t}(x))
    	=\text{tr}\biggl(\nabla_{x}\frac{\mathrm{d}}{\mathrm{d}t}\bigg\vert_{t=0}\varphi_{t}(x)\biggr)
    	=\text{tr}(\nabla_{x}T(x))
    	=\text{div}(T(x)).
    \end{equation*}
    Differentiating the other integrand with respect to $t$ and evaluating at $t=0$ gives
    \begin{align*}
    	&\frac{\mathrm{d}}{\mathrm{d}t}\bigg\vert_{t=0}\frac{1}{2}\biggl|\nabla_{x}\biggl(
    	\frac{\nabla_{x}\varphi_{t}(x)\overline{u}(x)}{|\nabla_{x}\varphi_{t}(x)\overline{u}(x)|}\biggr)\nabla_{y}\psi_{t}(\varphi_{t}(x))\biggr|^{2}\\
    	=&\text{tr}\biggl(\nabla_{x}\overline{u}(x)^{T}\nabla_{x}\overline{u}(x)\nabla_{x}\frac{\mathrm{d}}{\mathrm{d}t}\bigg\vert_{t=0}\psi_{t}(x)\biggr)
    	+\text{tr}\biggl(\nabla_{x}\overline{u}(x)^{T}\nabla_{x}\biggl(\nabla_{x}\frac{\mathrm{d}}{\mathrm{d}t}\bigg\vert_{t=0}\varphi_{t}(x)\overline{u}(x)\biggr)\biggr)\\
    	&-\text{tr}\biggl(\nabla_{x}\overline{u}(x)^{T}
    	\nabla_{x}\biggl(\overline{u}(x)\text{tr}\biggl(u(x)^{T}\nabla_{x}\frac{\mathrm{d}}{\mathrm{d}t}\bigg\vert_{t=0}\varphi_{t}(x)\overline{u}(x)\biggr)\biggr)\biggr)
    \end{align*}
    where we have used that $\partial_{y_{i}y_{j}}^{2}\psi_{0}(y)=0$ for all $i,j=1,2,3$ and that $\nabla_{x}\varphi_{0}(x)=I_{3}$.
    Using the definition of $T$, that $\varphi_{t}$ and $\psi_{t}$ are inverses of each other, as well as that $\nabla_{y}\psi_{0}(y)=I_{3}$ gives
    \begin{equation*}
    	\frac{\mathrm{d}}{\mathrm{d}t}\bigg\vert_{t=0}\psi_{t}(y)+T(y)=0.
    \end{equation*}
    From this and our previous calculations we obtain
    \begin{align*}
    	\frac{\mathrm{d}}{\mathrm{d}t}\bigg\vert_{t=0}\frac{1}{2}\biggl|\nabla_{x}\biggl(
    	\frac{\nabla_{x}\varphi_{t}(x)\overline{u}(x)}{|\nabla_{x}\varphi_{t}(x)\overline{u}(x)|}\biggr)\nabla_{y}\psi_{t}(\varphi_{t}(x))\biggr|^{2}
    	=&-\text{tr}(\nabla_{x}\overline{u}(x)^{T}\nabla_{x}\overline{u}(x)\nabla_{x}T(x))\\
        &+\text{tr}(\nabla_{x}\overline{u}(x)^{T}\nabla_{x}(\nabla_{x}T(x)u(x)))\\
    	&-\text{tr}(\nabla_{x}\overline{u}(x)^{T}\nabla_{x}(\overline{u}(x)\text{tr}(\overline{u}(x)^{T}\nabla_{x}T(x)\overline{u}(x)))).
    \end{align*}
    Combining all our calculations now gives that
    \begin{align*}
    	\frac{\mathrm{d}}{\mathrm{d}t}\bigg\vert_{t=0}\frac{1}{2}\int_{\Omega}\!{}|\nabla_{y}\overline{u}_{t}|^{2}
    	=&\int_{\Omega}\!{}\biggl\{-\nabla_{x}\overline{u}(x):\nabla_{x}\overline{u}(x)\nabla_{x}T(x)+\nabla_{x}\overline{u}(x):\nabla_{x}(\nabla_{x}T(x)\overline{u}(x))\\
    	&-\nabla_{x}\overline{u}(x):\nabla_{x}[(\overline{u}(x)\cdot\nabla_{x}T(x)\overline{u}(x))\overline{u}(x)]+\frac{1}{2}|\nabla_{x}\overline{u}|^{2}\text{div}(T(x))\biggr\}.
    \end{align*}
    We observe that the restrictions imposed on $\varphi$ and $\overline{u}$ permit the map
    $x\mapsto(\overline{u}(x)\cdot\nabla_{x}T(x)\overline{u}(x))\overline{u}(x)$ to be a suitable test
    function in \eqref{eq:weakform}.
    As a result, we obtain
    \begin{equation*}
    	\int_{\Omega}\!{}\nabla_{x}\overline{u}(x):\nabla_{x}\bigl[(\overline{u}(x)\cdot\nabla_{x}T(x)\overline{u}(x))\overline{u}(x)\bigr]
    	=\int_{\Omega}\!{}|\nabla_{x}\overline{u}(x)|^{2}(\overline{u}(x)\cdot\nabla_{x}T(x)\overline{u}(x))
    \end{equation*}
    where we have used that $|\overline{u}|=1$ for $\mathcal{L}^{3}$-almost every point in $\Omega$.
    Using this calculation and minimality of $u$ for \eqref{def:dir}
    we end up with
    \begin{align*}
    	0
    	=&\int_{\Omega}\!{}\biggl\{-\nabla_{x}\overline{u}(x):\nabla_{x}\overline{u}(x)\nabla_{x}T(x)+\nabla_{x}\overline{u}(x):\nabla_{x}(\nabla_{x}T(x)\overline{u}(x))\\
    	&-|\nabla_{x}\overline{u}(x)|^{2}(\overline{u}(x)\cdot\nabla_{x}T(x)\overline{u}(x))+\frac{1}{2}|\nabla_{x}\overline{u}|^{2}\text{div}(T(x))\biggr\}.
    \end{align*}

    Next, we consider the exterior domain case.
    Notice that since $\overline{u}_{t}\in\mathcal{W}_{u}$ for all $t\in[0,\delta]$
    and since $\overline{u}_{0}\vert_{B_{2}(0)\setminus\Omega}=\overline{u}\vert_{B_{2}(0)\setminus\Omega}=\widetilde{u}$
    is a minimizer of $\widetilde{E}$ by
    Proposition \ref{prop:refl_E}, then
    \begin{equation*}
    	0=\frac{\mathrm{d}}{\mathrm{d}t}\bigg\vert_{t=0}\widetilde{E}(\overline{u}_{t})
    \end{equation*}
By the Change of Variables Theorem applied with $\varphi_{t}$ we can rewrite
    $\widetilde{E}(\overline{u}_{t})$ as
    \begin{align*}
    	\widetilde{E}(\overline{u}_{t})=&\frac{1}{2}\int_{B_{2}(0)\setminus\Omega}\!{}\frac{1}{|\varphi_{t}(x)|^{2}}
    	\biggl[|\nabla_{y}\overline{u}_{t}(\varphi_{t}(x))|^{2}+\frac{4}{|\varphi_{t}(x)|^{4}}(\overline{u}_{t}(\varphi_{t}(x))\cdot{}\varphi_{t}(x))^{2}\\
    	&+\frac{4}{|\varphi_{t}(x)|^{2}}\overline{u}_{t}(\varphi_{t}(x))^{T}\nabla_{y}\overline{u}_{t}(\varphi_{t}(x))^{T}\varphi_{t}(x)
    	-\frac{4}{|\varphi_{t}(x)|^{2}}(\overline{u}_{t}(\varphi_{t}(x))\cdot{}\varphi_{t}(x))\text{div}(\overline{u}_{t})(\varphi_{t}(x))\biggr]
    	\det(\nabla_{x}\varphi_{t}(x)).
    \end{align*}
    Combining the following identities
    \begin{equation*}
    	\overline{u}_{t}(\varphi_{t}(x))=\frac{\nabla_{x}\varphi_{t}(x)\overline{u}(x)}{|\nabla_{x}\varphi_{t}(x)\overline{u}(x)|},
        \hspace{20pt}
        \nabla_{y}\overline{u}_{t}(\varphi_{t}(x))=\nabla_{x}\biggl(\frac{\nabla_{x}\varphi_{t}(x)\overline{u}(x)}{|\nabla_{x}\varphi_{t}(x)\overline{u}(x)|}\biggr)\nabla_{y}\psi_{t}(\varphi_{t}(x)).
    \end{equation*}
    and using the Inverse Function Theorem we can further rewrite
    $\widetilde{E}(\overline{u}_{t})$ as
    \begin{align*}
    	\frac{1}{2}\int_{B_{2}(0)\setminus\Omega}\!{}\frac{1}{|\varphi_{t}(x)|^{2}}\biggl[
    	&\biggl|\nabla_{x}\biggl(\frac{\nabla_{x}\varphi_{t}(x)\overline{u}(x)}{|\nabla_{x}\varphi_{t}(x)\overline{u}(x)|}\biggr)
    	\nabla_{y}\psi_{t}(\varphi_{t}(x))\biggr|^{2}\\
    	&+\frac{4}{|\varphi_{t}(x)|^{4}|\nabla_{x}\varphi_{t}(x)\overline{u}(x)|^{2}}\Bigl(\overline{u}(x)^{T}\nabla_{x}\varphi_{t}(x)^{T}\varphi_{t}(x)\Bigr)^{2}\\
    	&+\frac{4}{|\varphi_{t}(x)|^{2}|\nabla_{x}\varphi_{t}(x)\overline{u}(x)|}\overline{u}(x)^{T}\biggl[\nabla_{x}\biggl(\frac{\nabla_{x}\varphi_{t}(x)\overline{u}(x)}{|\nabla_{x}\varphi_{t}(x)\overline{u}(x)|}\biggr)\biggr]^{T}\varphi_{t}(x)\\
    	&-\frac{4}{|\varphi_{t}(x)|^{2}|\nabla_{x}\varphi_{t}(x)\overline{u}(x)|}\overline{u}(x)^{T}\nabla_{x}\varphi_{t}(x)^{T}\varphi_{t}(x)
    	\text{tr}\biggl(
    	\nabla_{x}\biggl(\frac{\nabla_{x}\varphi_{t}(x)\overline{u}(x)}{|\nabla_{x}\varphi_{t}(x)\overline{u}(x)|}\biggr)\nabla_{y}\psi_{t}(\varphi_{t}(x))\biggr)\biggr].
    \end{align*}
    Differentiating in $t$, evaluating at $t=0$, and evaluating using minimality gives
    \begin{align*}
    	0=&-\int_{B_{2}(0)\setminus\Omega}\!{}\frac{x\cdot{}T(x)}{|x|^{4}}
    	\biggl[|\nabla_{x}\overline{u}(x)|^{2}+\frac{4}{|x|^{4}}(\overline{u}(x)\cdot{}x)^{2}
    	+\frac{4}{|x|^{2}}\overline{u}(x)^{T}\nabla_{x}\overline{u}(x)^{T}x
    	-\frac{4}{|x|^{2}}(\overline{u}(x)\cdot{}x)\text{div}(\overline{u})(x)\biggr]\\
    	&+\int_{B_{2}(0)\setminus\Omega}\!{}\frac{1}{|x|^{2}}\biggl[
    	-\nabla_{x}\overline{u}(x):\nabla_{x}\overline{u}(x)\nabla_{x}T(x)+\nabla_{x}\overline{u}(x):\nabla_{x}(\nabla_{x}T(x)\overline{u}(x))\\
    	&-\nabla_{x}\overline{u}(x):\nabla_{x}(\overline{u}(x)\text{tr}(\overline{u}(x)^{T}\nabla_{x}T(x)\overline{u}(x)))\\
    	&-\frac{16x\cdot{}T(x)}{|x|^{6}}(\overline{u}(x)\cdot{}x)^{2}
    	-\frac{8\overline{u}(x)^{T}\nabla_{x}T(x)^{T}\overline{u}(x)}{|x|^{4}}(\overline{u}(x)\cdot{}x)^{2}
    	+\frac{8}{|x|^{4}}(\overline{u}(x)\cdot{}x)\Bigl[\overline{u}(x)^{T}\nabla_{x}T(x)^{T}x+\overline{u}(x)\cdot{}T(x)\Bigr]\\
    	&-\frac{8x\cdot{}T(x)}{|x|^{4}}\overline{u}(x)^{T}\nabla_{x}\overline{u}(x)^{T}x
    	-\frac{4\overline{u}(x)^{T}\nabla_{x}T(x)\overline{u}(x)}{|x|^{2}}
    	+\frac{4}{|x|^{2}}\overline{u}(x)^{T}\Bigl[\nabla_{x}(\nabla_{x}T(x)\overline{u}(x))\Bigr]^{T}x\\
    	&-\frac{4}{|x|^{2}}\overline{u}(x)^{T}\biggl[\nabla_{x}\biggl(\bigl(\overline{u}(x)^{T}\nabla_{x}T(x)\overline{u}(x)\bigr)\overline{u}(x)\biggr)\biggr]^{T}x
    	+\frac{4}{|x|^{2}}\overline{u}(x)^{T}\nabla_{x}\overline{u}(x)^{T}T(x)\\
    	&+\frac{8x\cdot{}T(x)}{|x|^{4}}(\overline{u}(x)\cdot{}x)\text{div}(\overline{u})(x)
    	+\frac{4\overline{u}(x)^{T}\nabla_{x}T(x)\overline{u}(x)}{|x|^{2}}(\overline{u}(x)\cdot{}x)\text{div}(\overline{u})(x)
    	-\frac{4}{|x|^{2}}\overline{u}(x)^{T}\nabla_{x}T(x)^{T}x\text{div}(\overline{u})(x)\\
    	&-\frac{4}{|x|^{2}}(\overline{u}(x)\cdot{}T(x))\text{div}(\overline{u})(x)
    	-\frac{4}{|x|^{2}}(\overline{u}(x)\cdot{}x)\text{div}\bigl(\nabla_{x}T(x)\overline{u}(x)\bigr)\\
    	&+\frac{4}{|x|^{2}}(\overline{u}(x)\cdot{}x)\text{tr}\Bigl(\nabla_{x}\Bigl(\bigl(\overline{u}(x)^{T}\nabla_{x}T(x)\overline{u}(x)\bigr)\overline{u}(x)\Bigr)\Bigr)
    	-\frac{4}{|x|^{2}}(\overline{u}(x)\cdot{}x)\text{tr}\Bigl(\nabla_{x}\overline{u}(x)\Bigl[-\nabla_{x}T(x)\Bigr]\Bigr)
    	\biggr]\\
    	&+\frac{1}{2}\int_{B_{2}(0)\setminus\Omega}
    	\frac{1}{|x|^{2}}\biggl[|\nabla_{x}\overline{u}(x)|^{2}+\frac{4}{|x|^{2}}(\overline{u}(x)\cdot{}x)^{2}
    	+\frac{4}{|x|^{2}}\overline{u}(x)^{T}\nabla_{x}\overline{u}(x)^{T}x-\frac{4}{|x|^{2}}(\overline{u}(x)\cdot{}x)\text{div}(\overline{u})(x)\biggr]\text{div}(T(x)).
    \end{align*}
    Notice that since $\varphi(\cdot,t)-\text{Id}_{\mathbb{R}^{3}}(\cdot)$ has compact support in $B_{2}(0)$ then $T(x)=0$ for all
    $x\in\partial{}B_{2}(0)$.
    We observe that since $\varphi\in{}C^{3}$ and $\overline{u}\in{}W_{T}^{1,2}(\Omega;\mathbb{S}^{2})$ then
    \begin{equation*}
    	\text{tr}(\overline{u}(x)^{T}\nabla_{x}T(x)\overline{u}(x))\overline{u}(x)\in{}W_{T}^{1,2}(B_{2}(0)\setminus\Omega;\mathbb{R}^{3})\cap{}L^{\infty}(B_{2}(0)\setminus\Omega;\mathbb{R}^{3})
    \end{equation*}
    and hence
    \begin{equation*}
    	\text{tr}(\overline{u}(x)^{T}\nabla_{x}T(x)\overline{u}(x))\overline{u}(x)\in\mathcal{W}_{u}.
    \end{equation*}
    We conclude by Corollary \ref{cor:outerpde} that
    \begin{align*}
    	&-\int_{B_{2}(0)\setminus\Omega}\!{}\frac{1}{|x|^{2}}\nabla_{x}\overline{u}(x):\nabla_{x}(\overline{u}(x)\text{tr}(\overline{u}(x)^{T}\nabla_{x}T(x)\overline{u}(x)))\\
    	=&-\int_{B_{2}(0)\setminus\Omega}\!{}\frac{1}{|x|^{2}}\biggl[
    		-\frac{4}{|x|^{4}}(\overline{u}(x)\cdot{}x)x^{T}
    		+\frac{4}{|x|^{4}}(\overline{u}(x)\cdot{}x)^{2}\overline{u}(x)^{T}
    		+|\nabla{}\overline{u}(x)|^{2}\overline{u}(x)^{T}\\
    		&-\frac{4}{|x|^{2}}x^{T}\nabla{}\overline{u}(x)
    		+\frac{4}{|x|^{2}}\bigl[\overline{u}(x)^{T}\nabla{}\overline{u}(x)^{T}x\bigr]\overline{u}(x)^{T}
    		+\frac{4}{|x|^{2}}\text{div}(\overline{u})(x)x^{T}\\
    		&-\frac{4}{|x|^{2}}(\overline{u}(x)\cdot{}x)\text{div}(\overline{u})(x)\overline{u}(x)^{T}
    	\biggr]\overline{u}(x)\text{tr}(\overline{u}(x)^{T}\nabla_{x}T(x)\overline{u}(x)).
    \end{align*}
    Of particular note is that the highest order term in the gradient of $\overline{u}$ satisfies
    \begin{equation*}
    	|\nabla{}\overline{u}(x)|^{2}\overline{u}(x)^{T}\overline{u}(x)\text{tr}(\overline{u}(x)^{T}\nabla_{x}T(x)\overline{u}(x))=|\nabla{}\overline{u}(x)|^{2}\text{tr}(\overline{u}(x)^{T}\nabla_{x}T(x)\overline{u}(x))
    \end{equation*}
    which matches the corresponding interior term.
    Next, we notice that by \eqref{eq:gradinversion}, the Change of
    Variables Theorem applied with $\iota$, and \eqref{eq:invol} we obtain
    \begin{align*}
    	\int_{B_{2}(0)\setminus\Omega}\!{}\frac{1}{|x|^{2}}
    	\nabla_{x}\overline{u}(x):\nabla_{x}\bigl(\nabla_{x}T(x)\overline{u}(x)\bigr)
    	=&\int_{\Omega}\!{}\nabla_{x}\overline{u}(x):
    	\nabla_{x}\biggl[A(x)\nabla_{x}T(\iota(x))\overline{u}(\iota(x))\biggr]\\
    	&+\sum_{i=1}^{3}\int_{\Omega}\!{}\partial_{x_{i}}A(x)\overline{u}(x)\mathbf{e}_{i}^{T}:
    	\nabla_{x}\biggl[\nabla_{x}T(\iota(x))\overline{u}(\iota(x))\biggr]\\
    	&-\sum_{i=1}^{3}\int_{\Omega}\!{}\nabla_{x}\overline{u}(x):
    	\partial_{x_{i}}A(x)\nabla_{x}T(\iota(x))\overline{u}(\iota(x))\mathbf{e}_{i}^{T}.
    \end{align*}
    Combining this with the corresponding interior term we find that
    \begin{align*}
    	&\int_{\Omega}\!{}\nabla_{x}\overline{u}(x)\nabla_{x}\bigl(\nabla_{x}T(x)\overline{u}(x)\bigr)+
    	\int_{B_{2}(0)\setminus\Omega}\!{}\frac{1}{|x|^{2}}
    	\nabla_{x}\overline{u}(x):\nabla_{x}\bigl(\nabla_{x}T(x)\overline{u}(x)\bigr)\\
    	=&2\int_{\Omega}\!{}\nabla_{x}\overline{u}(x):\nabla_{x}\bigl[\nabla_{x}T(\cdot)\overline{u}(\cdot)\bigr]_{e}(x)
    	+\sum_{i=1}^{3}\int_{\Omega}\!{}\partial_{x_{i}}A(x)\overline{u}(x)\mathbf{e}_{i}^{T}:
    	\nabla_{x}[\nabla_{x}T(\iota(x))\overline{u}(\iota(x))]\\
    	&-\sum_{i=1}^{3}\int_{\Omega}\!{}\nabla_{x}\overline{u}(x):
    	\partial_{x_{i}}A(x)\nabla_{x}T(\iota(x))\overline{u}(\iota(x))\mathbf{e}_{i}^{T}
    \end{align*}
    where for a function $\Psi$ taking values in $\mathbb{R}^{3}$ we define $\Psi_{e}$ by
    \begin{equation*}
    	\Psi_{e}(x)\coloneqq\frac{1}{2}\Bigl(\Psi(x)+A(x)\Psi\bigl(\iota(x)\bigr)\Bigr).
    \end{equation*}
    We observe that for $x\in\partial\Omega$ we have, since $\nu(x)=x$ on $\partial\Omega$, that
    \begin{equation*}
    	\bigl[\nabla_{x}T(\cdot)\overline{u}(\cdot)\bigr]_{e}(x)\cdot\nu(x)
    	=\frac{1}{2}\bigl[\nu(x)^{T}\nabla_{x}T(x)\overline{u}(x)-\nu(x)^{T}\nabla_{x}T(x)\overline{u}(x)\bigr]=0.
    \end{equation*}
    Since $\bigl[\nabla_{x}T(\cdot)\overline{u}(\cdot)\bigr]_{e}\in{}W^{1,2}(\Omega;\mathbb{R}^{3})\cap{}L^{\infty}(\Omega;\mathbb{R}^{3})$
    then by \eqref{eq:weakform}, \eqref{eq:invol}, the Change of
    Variables Theorem applied with $\iota$, and \eqref{eq:normsq} we have
    \begin{align*}
    	&\int_{\Omega}\!{}\nabla_{x}\overline{u}(x):\nabla_{x}\bigl(\nabla_{x}T(x)\overline{u}(x)\bigr)
    	+
    	\int_{B_{2}(0)\setminus\Omega}\!{}\frac{1}{|x|^{2}}
    	\nabla_{x}\overline{u}(x):\nabla_{x}\bigl(\nabla_{x}T(x)\overline{u}(x)\bigr)\\
    	=&\int_{\Omega}\!{}|\nabla_{x}\overline{u}(x)|^{2}\overline{u}(x)^{T}\nabla_{x}T(x)\overline{u}(x)
    	+\int_{B_{2}(0)\setminus\Omega}\!{}\frac{1}{|x|^{2}}\bigl|\nabla_{x}\overline{u}(x)\bigr|^{2}
    	\overline{u}(x)^{T}\nabla_{x}T(x)\overline{u}(x)\\
    	&+2\sum_{i=1}^{3}\int_{B_{2}(0)\setminus\Omega}\!{}\frac{1}{|x|^{2}}\Bigl[A(x)\nabla_{x}\overline{u}(x):\partial_{x_{i}}A(x)\overline{u}(x)\mathbf{e}_{i}^{T}\Bigr]
    	\overline{u}(x)^{T}\nabla_{x}T(x)\overline{u}(x)\\
    	&+\int_{B_{2}(0)\setminus\Omega}\!{}\frac{1}{|x|^{2}}\biggl|\sum_{i=1}^{3}\partial_{x_{i}}A(x)\overline{u}(x)\mathbf{e}_{i}^{T}\biggr|^{2}
    	\overline{u}(x)^{T}\nabla_{x}T(x)\overline{u}(x)\\
    	&+\sum_{i=1}^{3}\int_{\Omega}\!{}\partial_{x_{i}}A(x)\overline{u}(x)\mathbf{e}_{i}^{T}:
    	\nabla_{x}[\nabla_{x}T(\iota(x))\overline{u}(\iota(x))]
    	-\sum_{i=1}^{3}\int_{\Omega}\!{}\nabla_{x}\overline{u}(x):
    	\partial_{x_{i}}A(x)\nabla_{x}T(\iota(x))\overline{u}(\iota(x))\mathbf{e}_{i}^{T}.
    \end{align*}
    Notice that integrating by parts as well as using \eqref{eq:partialA}
    and that $\overline{u}$ satisfies \eqref{eq:tanbc} we obtain
    \begin{equation*}
    	\sum_{i=1}^{3}\int_{\Omega}\!{}\partial_{x_{i}}A(x)\overline{u}(x)\mathbf{e}_{i}^{T}:
    	\nabla_{x}[\nabla_{x}T(\iota(x))\overline{u}(\iota(x))]
    	=-\sum_{i=1}^{3}\int_{\Omega}\partial_{x_{i}}\Bigl[\partial_{x_{i}}A(x)\overline{u}(x)\Bigr]
    	\cdot\nabla_{x}T(\iota(x))\overline{u}(\iota(x)).
    \end{equation*}
    Similarly, we observe that by
    multiple applications of integration by parts,
    using that $T$ has compact support in $B_{2}(0)$, and using that $\overline{u}$ satisfies \eqref{eq:tanbc} gives the following four identities
    \begin{align*}
    	&\int_{B_{2}(0)\setminus\Omega}\!{}\frac{4}{|x|^{4}}\overline{u}(x)^{T}\Bigl[\nabla_{x}\bigl(\nabla_{x}T(x)\overline{u}(x)\bigr)\Bigr]^{T}x
        =-\sum_{i=1}^{3}\int_{B_{2}(0)\setminus\Omega}\!{}\partial_{x_{i}}\biggl(\frac{4\overline{u}^{i}(x)x}{|x|^{4}}\biggr)\cdot\nabla_{x}T(x)\overline{u}(x)\\
    	&-\int_{B_{2}(0)\setminus\Omega}\!{}
    	\frac{4}{|x|^{4}}\overline{u}(x)^{T}=\sum_{i=1}^{3}\int_{B_{2}(0)\setminus\Omega}\!{}\partial_{x_{i}}\biggl(\frac{4\overline{u}(x)x_{i}}{|x|^{4}}\biggr)\cdot
    	\Bigl[\bigl(\overline{u}(x)^{T}\nabla_{x}T(x)\overline{u}(x)\bigr)\overline{u}(x)\Bigr]\\
        &-\int_{B_{2}(0)\setminus\Omega}\!{}\frac{4}{|x|^{4}}(\overline{u}(x)\cdot{}x)\text{div}\bigl(\nabla_{x}T(x)\overline{u}(x)\bigr)
    	=\int_{B_{2}(0)\setminus\Omega}\!{}\nabla_{x}\biggl[\frac{4}{|x|^{4}}\bigl(\overline{u}(x)\cdot{}x\bigr)\biggr]\cdot\nabla_{x}T(x)\overline{u}(x)\\
        &\int_{B_{2}(0)\setminus\Omega}\!{}\frac{4}{|x|^{4}}(\overline{u}(x)\cdot{}x)\text{tr}\Bigl(\nabla_{x}\Bigl(\bigl(\overline{u}(x)^{T}\nabla_{x}T(x)\overline{u}(x)\bigr)\overline{u}(x)\Bigr)\Bigr)
        \\
        &=-\int_{B_{2}(0)\setminus\Omega}\!{}\nabla_{x}\biggl(\frac{4}{|x|^{4}}(\overline{u}(x)\cdot{}x)\biggr)\cdot
    	\Bigl(\bigl(\overline{u}(x)^{T}\nabla_{x}T(x)\overline{u}(x)\bigr)\overline{u}(x)\Bigr)
    \end{align*}
    Combining this with our previous calculations we see that there are
    functions $\mathcal{F}$ and $\mathcal{G}$ such that
    \begin{align*}
    	0
    	=&\int_{\Omega}\!{}\biggl\{-\nabla_{x}\overline{u}(x):\nabla_{x}\overline{u}(x)\nabla_{x}T(x)+\frac{1}{2}|\nabla{}\overline{u}|^{2}\text{div}(T(x))\biggr\}\\
    	&-\int_{B_{2}(0)\setminus\Omega}\!{}\frac{1}{|x|^{2}}\nabla_{x}\overline{u}(x):\nabla_{x}\overline{u}(x)\nabla_{x}T(x)
    	+\frac{1}{2}\int_{B_{2}(0)\setminus\Omega}
    	\frac{1}{|x|^{2}}|\nabla_{x}\overline{u}(x)|^{2}\text{div}(T(x))\\
    	&+\int_{B_{2}(0)\setminus\Omega}\!{}\mathcal{F}(x,\overline{u}(x),\nabla_{x}\overline{u}(x))\cdot{}T(x)
    	+\int_{B_{2}(0)\setminus\Omega}\!{}\mathcal{G}(x,\overline{u}(x),\nabla_{x}\overline{u}(x)):\nabla_{x}T(x)
    \end{align*}
    and
    \begin{equation*}
    	|\mathcal{F}(x,z,p)|\le{}C(1+|p|^{2}),\hspace{20pt}|\mathcal{G}(x,z,p)|\le{}C(1+|p|).
    \end{equation*}
\end{proof}

\subsection{Monotonicity Formula and Partial Regularity}\label{subsec:monoform}

In this subsection, we demonstrate that an up to the boundary version of the
Monotonicity Formula holds for a minimizer of \eqref{def:dir} in
$W_{T}^{1,2}(\Omega;\mathbb{S}^{2})$.
This extends the interior Monotonicity Formula demonstrated in
\cite{Ev2}.\\

\subsubsection{Formal Calculation}\label{subsec:formalcalc}
    Before we begin with a rigorous derivation of the Monotonicity Formula
    we provide a formal calculation which acts as intuition for the
    rigorous result.
    While not completely rigorous this calculation demonstrates that one
    can expect the $1$-Hausdorff density of the energy density
    measure to satisfy an ``almost monotone" property resembling the
    classical Monotonicity Formula presented in equation $1.4$ of \cite{Ev2}.
    In order to make this rigorous, the main modification is to ensure item
    $1$ of Definition \ref{def:innervartan} is satisfied while preserving
    $C^{3}-$smoothness.\\

    For $\delta,r>0$ sufficiently small we consider the function
    $T_{\delta,r}\colon{}B_{2}(0)\to\mathbb{R}^{3}$ defined by
    \begin{equation*}
    	T_{\delta,r}(x)\coloneqq\eta_{\delta,r}(|x-x_{0}|)(x-x_{0})
    \end{equation*}
    where
    \begin{equation*}
    	\eta_{\delta,r}(s)\coloneqq\min\biggl\{1,\max\biggl\{0,\frac{r+\delta-s}{\delta}\biggr\}\biggr\}.
    \end{equation*}
    Observe that
    \begin{align*}
    	\nabla{}T_{\delta,r}(x)&=\eta_{\delta,r}'(|x-x_{0}|)\frac{(x-x_{0})(x-x_{0})^{T}}{|x-x_{0}|}+\eta_{\delta,r}(|x-x_{0}|)I_{3}\\
    	\text{div}(T_{\delta,r})(x)&=\eta_{\delta,r}'(|x-x_{0}|)|x-x_{0}|+3\eta_{\delta,r}(|x-x_{0}|).
    \end{align*}
    By Proposition \ref{prop:inner_variations_tang}, applying the Coarea
    Formula, using that $\eta_{\delta,r}'=\frac{-1}{\delta}$ on
    $A_{r,r+\delta}(x_{0})$, and letting $\delta\to0^{+}$
    \begin{align*}
    	0
    	&=\int_{\partial{}B_{r}(x_{0})}\!{}\min\Big\{1,\frac{1}{|x|^2}\Big\}\biggl\{r|\partial_{n}\overline{u}(x)|^{2}
    	-\frac{r}{2}|\nabla_{x}u|^{2}\biggr\}
    	+\int_{B_{r}(x_{0})}\!{}\frac{1}{2}\min\Big\{1,\frac{1}{|x|^2}\Big\} |\nabla_{x}\overline{u}|^{2}\\
    	&\qquad
        +\int_{B_{r}(x_{0})\cap(B_{2}(0)\setminus\Omega)}\!{}\mathcal{F}(x,\overline{u}(x),\nabla_{x}\overline{u}(x))\cdot{}(x-x_{0})\\
    	&\qquad
        -\frac{1}{r}\int_{\partial{}B_{r}(x_{0})\cap(B_{2}(0)\setminus\Omega)}\!{}\mathcal{G}(x,\overline{u}(x),\nabla_{x}\overline{u}(x)):
    	(x-x_{0})(x-x_{0})^{T}\\
    	&\qquad
        +\int_{B_{r}(x_{0})\cap(B_{2}(0)\setminus\Omega)}\!{}\mathcal{G}(x,\overline{u}(x),\nabla_{x}\overline{u}(x)):I_{3}.
    \end{align*}
We conclude that 
\begin{align*}
	r^{2}\frac{\mathrm{d}}{\mathrm{d}r}\biggl\{
    \frac{1}{2r}
    \int_{B_{r}(x_{0})}\!{}\min\Big\{1,\frac{1}{|x|^2}\Big\}|\nabla_{x}\overline{u}|^{2}
    \biggr\}
	=&\int_{\partial{}B_{r}(x_{0})}\!{}r\ \min\Big\{1,\frac{1}{|x|^2}\Big\}|\partial_{n}\overline{u}(x)|^{2} \\
	&+\int_{B_{r}(x_{0})\cap(B_{2}(0)\setminus\Omega)}\!{}\mathcal{F}(x,\overline{u}(x),\nabla_{x}\overline{u}(x))\cdot{}(x-x_{0})\\
	&-\frac{1}{r}\int_{\partial{}B_{r}(x_{0})\cap(B_{2}(0)\setminus\Omega)}\!{}\mathcal{G}(x,\overline{u}(x),\nabla_{x}\overline{u}(x)):
	(x-x_{0})(x-x_{0})^{T}\\
	&+\int_{B_{r}(x_{0})\cap(B_{2}(0)\setminus\Omega)}\!{}\mathcal{G}(x,\overline{u}(x),\nabla_{x}\overline{u}(x)):I_{3}.
\end{align*}
Using that $|\mathcal{F}(x,z,p)|\le{}C_{1}(1+|p|^{2})$,
$|\mathcal{G}(x,z,p)|\le{}C_{2}(1+|p|)$,
and the Arithmetic-Geometric Inequality we find that 
\begin{align*}
	r^{2}\frac{\mathrm{d}}{\mathrm{d}r}\biggl\{\frac{1}{2r}\int_{B_{r}(x_{0})}\!{}\min\Big\{1,\frac{1}{|x|^2}\Big\}|\nabla_{x}\overline{u}|^{2}
	\biggr\}
	\ge&\int_{\partial{}B_{r}(x_{0})}\!{}r\ \min\Big\{1,\frac{1}{|x|^2}\Big\}|\partial_{n}\overline{u}(x)|^{2} \\
	&-(C_{1}r+2\sqrt{3}C_{2})\int_{B_{r}(x_{0})\cap(B_{2}(0)\setminus\Omega)}\!{}\biggl\{1+|\nabla_{x}\overline{u}|^{2}\biggr\} \\
    &-C_{2}r\int_{\partial{}B_{r}(x_{0})\cap(B_{2}(0)\setminus\Omega)}\!{}\biggl\{1+|\nabla_{x}\overline{u}|\biggr\}.
\end{align*}
After a further rewrite and substituting constants we obtain
\begin{align*}
	\frac{\mathrm{d}}{\mathrm{d}r}\biggl\{\frac{1}{2r}\int_{B_{r}(x_{0})}\!{}\min\Big\{1,\frac{1}{|x|^2}\Big\}|\nabla_{x}\overline{u}|^{2}
	\biggr\}
	\ge&\overline{C}_{0}\biggl[\frac{1}{r}\int_{\partial{}B_{r}(x_{0})}\!{}\min\Big\{1,\frac{1}{|x|^2}\Big\}|\partial_{n}\overline{u}(x)|^{2}
	\biggr]
	-\overline{C}_{1}\\
	&-\overline{C}_{2}
	\biggl[\frac{1}{2r}\int_{B_{r}(x_{0})}\!{}\min\Big\{1,\frac{1}{|x|^2}\Big\}|\nabla_{x}\overline{u}|^{2}
	\biggr].
\end{align*}
Letting $f_{x_{0}}\colon(0,1\slash4)\to[0,\infty)$ be defined by
\begin{equation*}
	f_{x_{0}}(r)\coloneqq\frac{1}{2r}\int_{B_{r}(x_{0})}\!{}\min\Big\{1,\frac{1}{|x|^2}\Big\}|\nabla_{x}\overline{u}|^{2}
    \, ,
\end{equation*}
we can bound the previous equation from below to obtain
\begin{equation*}
	f_{x_{0}}'(r)\ge-\overline{C}_{2}f_{x_{0}}(r)-\overline{C}_{1}.
\end{equation*}
From this we conclude that
\begin{equation*}
	\frac{\mathrm{d}}{\mathrm{d}r}\biggl[e^{\overline{C}_{2}r}f_{x_{0}}(r)\biggr]\ge-\overline{C}_{1}e^{\overline{C}_{2}r}
	\ge-\overline{C}_{1}e^{\overline{C}_{2}}\eqqcolon-\overline{C}_{3}.
\end{equation*}
Thus, $r\mapsto{}e^{\overline{C}_{2}r}f_{x_{0}}(r)+\overline{C}_{3}r$ is
non-decreasing.

\subsubsection{Rigorous Calculation}
\label{subsubsec:rigorous_calc}

Before we proceed with a rigorous proof, we introduce some notation that
will be used.
Since the arguments for the Monotonicity Formula from \cite{Ev2} hold
for points in $\Omega$ and $B_{2}(0)\setminus\overline{\Omega}$ for
sufficiently small balls then we will focus on points
$x_{0}\in\partial\Omega$.
In addition, for notational convenience, we will assume without loss of
generality that $x_{0}=(0,0,1)$.
We note that the constants we obtain by estimating coordinate descriptions
about $(0,0,1)$ remain true for all $x_{0}\in\partial\Omega$ since we can
obtain similar coordinate representations through application of a rotation.
\\

We consider $0<r_{0}<\frac{1}{4}$, which will be fixed later,
and we let $x_{0}=(0,0,1)$.
We denote $\mathcal{O}_{x_{0},r_{0}}^{2d}\subseteq\mathbb{R}^{3}$ by
\begin{equation*}
    \mathcal{O}_{x_{0},r_{0}}^{2d}\coloneqq
    \Bigl\{\bigl(t\cos(\theta),t\sin(\theta),\sqrt{1-t^{2}}\bigr):
    0\le{}t<r_{0},\hspace{10pt}\theta\in[0,2\pi)\Bigr\}
\end{equation*}
and we define
$\phi_{x_{0},r_{0}}\colon{}B_{r_{0}}^{2d}(0)\to\mathcal{O}_{x_{0},r_{0}}^{2d}$
by
\begin{equation*}
    \phi_{x_{0},r_{0}}(y_{1},y_{2})\coloneqq
    \Bigl(y_{1},y_{2},\sqrt{1-y_{1}^{2}-y_{2}^{2}}\Bigr)
    \, .
\end{equation*}
The function $\phi_{x_0,r_0}$ maps the two-dimensional flat disk of radius $r_0$ into a curved disk centered around $x_0$ that is part of a the sphere of radius $1$ centered at $0$.
We also let $\mathcal{O}_{x_{0},r_{0}}\subseteq\mathbb{R}^{3}$ denote
\begin{equation*}
\mathcal{O}_{x_{0},r_{0}}\coloneqq
    \Bigl\{(1+t\cos(\psi))\bigl(t\cos(\theta)\sin(\psi),
    t\sin(\theta)\sin(\psi),\sqrt{1-t^{2}\sin^{2}(\psi)}\bigr):
    0\le{}t<r_{0},\hspace{5pt}\theta\in[0,2\pi),\hspace{5pt}
    \psi\in[0,\pi]\Bigr\}
\end{equation*}
and define $\Phi_{x_{0},r_{0}}\colon{}B_{r_{0}}(0)\to\mathcal{O}_{x_{0},r_{0}}$
by
\begin{equation*}
    \Phi_{x_{0},r_{0}}(y_{1},y_{2},y_{3})\coloneqq
    (1+y_{3})\phi_{x_{0},r_{0}}(y_{1},y_{2}).
\end{equation*}
The map $\Phi_{x_{0},r_{0}}$ deforms  the three dimensional unit ball in a way that sends the flat equatorial disk to the spherical piece given by the image of $\phi_{x_{0},r_{0}}$. 
The sign of the third component $x_3$ decides whether the point gets mapped to the inside ($x_3<0$) or outside of the ball ($x_3>0$).
Since $x_{0}$ is fixed and $r_{0}$ will eventually be fixed, we will omit
them from the notation on $\Phi$ for convenience.
Next, we provide some information satisfied by $\Phi$.
First, we notice that
\begin{align*}
	\nabla\Phi(y)&=\begin{bmatrix}
		(1+y_{3})& 0& y_{1}\\
		0& (1+y_{3})& y_{2}\\
		\frac{-y_{1}(1+y_{3})}{\sqrt{1-y_{1}^{2}-y_{2}^{2}}}& \frac{-y_{2}(1+y_{3})}{\sqrt{1-y_{1}^{2}-y_{2}^{2}}}& \sqrt{1-y_{1}^{2}-y_{2}^{2}}
	\end{bmatrix},\\
	\Phi^{-1}(x)&=\biggl[\frac{x_{1}}{|x|},\frac{x_{2}}{|x|},|x|-1\biggr]^{T},\\
	\nabla\Phi^{-1}(x)&=
	\begin{bmatrix}
		\frac{1}{|x|}-\frac{x_{1}^{2}}{|x|^{3}}& -\frac{x_{1}x_{2}}{|x|^{3}}& -\frac{x_{1}x_{3}}{|x|^{3}}\\
		-\frac{x_{1}x_{2}}{|x|^{3}}& \frac{1}{|x|}-\frac{x_{2}^{2}}{|x|^{3}}& -\frac{x_{2}x_{3}}{|x|^{3}}\\
		\frac{x_{1}}{|x|}& \frac{x_{2}}{|x|}& \frac{x_{3}}{|x|}
	\end{bmatrix}\\
	&=\frac{1}{|x|}I_{3}
	+\frac{1}{|x|^{3}}\biggl[x_{3}\mathbf{e}_{3}-x\biggr]x^{T}
	+\frac{1}{|x|}\mathbf{e}_{3}\bigl[x-\mathbf{e}_{3}\bigr]^{T},\\
	\nabla\Phi\bigl(\Phi^{-1}(x)\bigr)&=
	\begin{bmatrix}
		|x|& 0& \frac{x_{1}}{|x|}\\
		0& |x|& \frac{x_{2}}{|x|}\\
		\frac{-x_{1}|x|}{x_{3}}& \frac{-x_{2}|x|}{x_{3}}& \frac{x_{3}}{|x|}
	\end{bmatrix},\\
	\nabla\Phi\bigl(\Phi^{-1}(x)\bigr)\Phi^{-1}(x)&=
	\begin{bmatrix}
		&2x_{1}-\frac{x_{1}}{|x|}\\
		&2x_{2}-\frac{x_{2}}{|x|}\\
		&2x_{3}-\frac{x_{3}}{|x|}-\frac{|x|^{2}}{x_{3}}
	\end{bmatrix}
	=\biggl(2-\frac{1}{|x|}\biggr)(x-x_{0})
    +\biggl(\frac{x_{3}(|x|-1)+|x|(x_{3}-|x|^{2})}{x_{3}|x|}\biggr)\mathbf{e}_{3}.
\end{align*}
We also observe that:

\begin{align*}
	\nabla^{2}\Phi^{1}(y)&=
	\begin{bmatrix}
		0& 0& 1\\
		0& 0& 0\\
		1& 0& 0
	\end{bmatrix},\hspace{10pt}
	\nabla^{2}\Phi^{2}(y)=
	\begin{bmatrix}
		0& 0& 0\\
		0& 0& 1\\
		0& 1& 0
	\end{bmatrix},\\
	\nabla^{2}\Phi^{3}(y)&=
	\begin{bmatrix}
		-\frac{(1+y_{3})}{\sqrt{1-y_{1}^{2}-y_{2}^{2}}}-\frac{y_{1}^{2}(1+y_{3})}{(1-y_{1}^{2}-y_{2}^{2})^{\frac{3}{2}}}&
		-\frac{y_{1}y_{2}(1+y_{3})}{(1-y_{1}^{2}-y_{2}^{2})^{\frac{3}{2}}}&
		\frac{-y_{1}}{\sqrt{1-y_{1}^{2}-y_{2}^{2}}}\\	
		-\frac{y_{1}y_{2}(1+y_{3})}{(1-y_{1}^{2}-y_{2}^{2})^{\frac{3}{2}}}&
		-\frac{(1+y_{3})}{\sqrt{1-y_{1}^{2}-y_{2}^{2}}}
		-\frac{y_{2}^{2}(1+y_{3})}{(1-y_{1}^{2}-y_{2}^{2})^{\frac{3}{2}}}&
		\frac{-y_{2}}{\sqrt{1-y_{1}^{2}-y_{2}^{2}}}\\
		\frac{-y_{1}}{\sqrt{1-y_{1}^{2}-y_{2}^{2}}}&
		\frac{-y_{2}}{\sqrt{1-y_{1}^{2}-y_{2}^{2}}}&
		0
	\end{bmatrix}
\end{align*}
and hence
\begin{align*}
	\nabla^{2}\Phi^{1}\bigl(\Phi^{-1}(x)\bigr)&=
	\begin{bmatrix}
		0& 0& 1\\
		0& 0& 0\\
		1& 0& 0
	\end{bmatrix},\hspace{10pt}
	\nabla^{2}\Phi^{2}\bigl(\Phi^{-1}(x)\bigr)=
	\begin{bmatrix}
		0& 0& 0\\
		0& 0& 1\\
		0& 1& 0
	\end{bmatrix},\\
	\nabla^{2}\Phi^{3}\bigl(\Phi^{-1}(x)\bigr)&=
	\begin{bmatrix}
		-\frac{|x|^{2}}{x_{3}}-\frac{x_{1}^{2}|x|^{2}}{x_{3}^{3}}&
		-\frac{x_{1}x_{2}|x|^{2}}{x_{3}^{3}}&
		\frac{-x_{1}}{x_{3}}\\	
		-\frac{x_{1}x_{2}|x|^{2}}{x_{3}^{3}}&
		-\frac{|x|^{2}}{x_{3}}
		-\frac{x_{2}^{2}|x|^{2}}{x_{3}^{3}}&
		\frac{-x_{2}}{x_{3}}\\
		\frac{-x_{1}}{x_{3}}&
		\frac{-x_{2}}{x_{3}}&
		0
	\end{bmatrix}.
\end{align*}
As a result, we notice that on $\mathcal{O}_{x_{0},r_{0}}$ we have the
following leading order expansions
\begin{equation}\label{eq:phileadingorder}
	\Phi^{-1}(x)=O(r_{0}),\hspace{10pt}
	\nabla\Phi^{-1}(x)=I_{3}+O(r_{0}),\hspace{10pt}
	\nabla\Phi\bigl(\Phi^{-1}(x)\bigr)=I_{3}+O(r_{0}).
\end{equation}

Next we choose $0<\delta,r<\frac{r_{0}}{2}$ and define $\eta_{\delta,r}\colon\mathbb{R}\to\mathbb{R}$ by
\begin{equation*}
	\eta_{\delta,r}(s)\coloneqq\min\biggl\{1,\max\biggl\{0,\frac{r+\delta-|s|}{\delta}\biggr\}\biggr\}.
\end{equation*}
We also introduce a mollifying function $\rho\in{}C_{c}^{\infty}(\mathbb{R})$ satisfying $\rho\ge0$, $\text{supp}(\rho)=[-1,1]$, and
$\int_{\mathbb{R}}\!{}\rho=1$.
We also let $\rho_{\alpha}$, for $0<\alpha<r$, denote
\begin{equation*}
	\rho_{\alpha}(s)\coloneqq\frac{1}{\alpha^{3}}\rho\biggl(\frac{s}{\alpha}\biggr).
\end{equation*}
Using this mollifying function we define $\eta_{\delta,r,\alpha}$ by
\begin{equation*}
	\eta_{\delta,r,\alpha}(s)\coloneqq\rho_{\alpha}\star\eta_{\delta,r}(s).
\end{equation*}
We notice that this function satisfies
\begin{enumerate}
	\item
		$\text{supp}(\eta_{\delta,r,\alpha})\subseteq{}(-r-\delta-\alpha,r+\delta+\alpha)$,
	\item
		$\eta_{\delta,r,\alpha}\in{}C^{\infty}(\mathbb{R})$,
	\item
		$\eta_{\delta,r,\alpha}\equiv1\text{ on }(-r+\alpha,r-\alpha)$.
\end{enumerate}
Next, we define $\varphi_{\delta,r,\alpha}$ for $0\le{}t\le\frac{\delta}{2}$ by
\begin{equation*}
	\varphi_{\delta,r,\alpha}(x,t)\coloneqq{}
	\begin{cases}
		\Phi\bigl((1+t\eta_{\delta,r,\alpha}(|\Phi^{-1}(x)|))\Phi^{-1}(x)\bigr)& \text{for }x\in\mathcal{O}_{x_{0},r_{0}}\\
		x& \text{for }x\in{}B_{2}(0)\setminus\mathcal{O}_{x_{0},r_{0}}.
	\end{cases}
\end{equation*}
For given $t$, the function $\varphi_{\delta,r,\alpha}$ essentially represents the compact perturbation of the identity given by the radial vector field $x-x_0$ but preserves the boundary $\partial\Omega$. 
This is achieved by first mapping $x$ into a ball where the boundary is flat, then applying a stretch by a factor of $1+t\eta_{\delta,r,\alpha}$ (depending on the distance from the center of the ball) which stays in the (flat) image of the boundary, and finally mapping back the resulting point.
We demonstrate in the following lemma that $\varphi_{\delta,r,\alpha}$
satisfies the requirements of Definition \ref{def:innervartan}.
This will allow us to provide a rigorous treatment of the Monotonicity
Formula.
\begin{lemma}\label{lem:monotonicityfamily}
    Suppose $x_{0}=(0,0,1)$, $0<r_{0}<\frac{1}{4}$,
    $0<\delta,r<\frac{r_{0}}{2}$, $0<\alpha<r$, and
    $0\le{}t\le\frac{\delta}{2}$.
    Then $\varphi_{\delta,r,\alpha}$ defined as above satisfies the
    requirements of Definition~\ref{def:innervartan}.
\end{lemma}

\begin{proof}
    From our definition, we see that
    $\varphi_{\delta,r,\alpha}(\cdot,t)-\text{id}_{B_{2}(0)}$ has support in
    $\overline{\Phi\Bigl(B_{r+\delta+\alpha}(0)\Bigr)}$,
    which is compact
    since $\Phi$ is continuous.
    In particular, since
    $\text{supp}(\eta_{\delta,r,\alpha})\subseteq\mathcal{O}_{x_{0},r_{0}}$
    then $\varphi_{\delta,r,\alpha}(x,t)=x$ in a neighbourhood of
    $\partial\mathcal{O}_{x_{0},r_{0}}$.
    As a result, $\varphi_{\delta,r,\alpha}$ is $C^{3}$ on $B_{2}(0)$.
    Next, we observe that
    \begin{equation*}
    	\varphi_{\delta,r\alpha}(x,0)=x=\text{id}_{B_{2}(0)}(x)
    \end{equation*}
    for all $x\in{}B_{2}(0)$.
    Next we demonstrate that
    $\varphi_{\delta,r,\alpha}(\Omega\times\{t\})=\Omega$ for all
    $0\le{}t\le\frac{\delta}{2}$.
    If $x\in\Omega\setminus\mathcal{O}_{x_{0},r_{0}}$ then by definition
    $\varphi_{\delta,r,\alpha}(x,t)=x\in\Omega$.
    If $x\in\Omega\cap\mathcal{O}_{x_{0},r_{0}}$ then observe that
    $y=\Phi^{-1}(x)$ is such that $y_{3}<0$.
    Since
    $\bigl(1+t\eta_{\delta,r,\alpha}(|\Phi^{-1}(x)|\bigr)>0$
    then
    \begin{equation*}
    	z=\bigl(1+t\eta_{\delta,r,\alpha}
        (|\Phi^{-1}(x)|)\bigr)\Phi^{-1}(x)
    	=\bigl(1+t\eta_{\delta,r,\alpha}(|\Phi^{-1}(x)|)\bigr)
        (y_{1},y_{2},y_{3})
    \end{equation*}
    has negative third component since $y_{3}<0$, $t\ge0$, and
    $|\Phi^{-1}(x)|\ge0$.
    By definition of $\Phi$ we then have
    \begin{equation*}
    	|\Phi(z)|=|1+z_{3}|\cdot
        |\phi_{x_{0},r_{0}}(z_{1},z_{2})|=|1+z_{3}|<1.
    \end{equation*}
    Hence, since $0\le{}t\le\frac{\delta}{2}$ was arbitrary,
    $\varphi_{\delta,r,\alpha}(\Omega\times{}\{t\})\subseteq\Omega$ for all
    $0\le{}t\le\frac{\delta}{2}$.
    Next, we show that
    $\varphi_{\delta,r,\alpha}(\Omega\times\{t\})\supseteq\Omega$.
    Consider $z_{0}\in\Omega$.
    If $z_{0}\in\Omega\setminus\mathcal{O}_{x_{0},r_{0}}$ then
    $\varphi_{\delta,r,\alpha}(z_{0},t)=z_{0}$ and hence we may assume that
    $z_{0}\in\Omega\cap\mathcal{O}_{x_{0},r_{0}}$.
    By definition of $\Phi$ and $\mathcal{O}_{x_{0},r_{0}}$,
    there is $z\in{}B_{r_{0}}(0)\times(0,1)$ such that
    $z_{0}=\Phi(z)$.
    If $|z|\ge{}r+\delta+\alpha$ then
    $\varphi_{\delta,r,\alpha}(z_{0},t)=z_{0}$ so we may assume that
    $|z|<r+\delta+\alpha$.
    Next, I demonstrate that the map
    \begin{equation*}
    	y\mapsto{}\bigl(1+t\eta_{\delta,r,\alpha}(|y|)\bigr)y
    \end{equation*}
    is surjective on $B_{r+\delta+\alpha}(0)$.
    First, notice that $0$ maps to $0$ under this map.
    Next, for $y\neq0$ we rewrite this map as
    \begin{equation*}
    	y\mapsto{}\bigl(1+t\eta_{\delta,r,\alpha}(|y|)\bigr)|y|\cdot
        \frac{y}{|y|}.
    \end{equation*}
    From this we see that the function factors into the form
    $f(s)\sigma$ for $s>0$ and $\sigma\in\mathbb{S}^{2}$.
    Thus, to demonstrate surjectivity it suffices to demonstrate that the
    function
    \begin{equation*}
    	s\mapsto{}\bigl(1+t\eta_{\delta,r,\alpha}(s)\bigr)s
    \end{equation*}
    is a surjective map from $(0,r+\delta+\alpha)$ to $(0,r+\delta+\alpha)$.
    To see that this is true note that the function is continuous and
    increasing since its derivative satisfies
    \begin{equation*}
    	1+t\eta_{\delta,r,\alpha}(s)+st\eta_{\delta,r,\alpha}'(s)
        \ge1-st\left\|\eta_{\delta,r,\alpha}'\right\|_{L^{\infty}}
        \ge1-\frac{s}{2}
        >1-\frac{r+\delta+\alpha}{2}>0.
    \end{equation*}
    Finally, observe that when $s=r+\delta+\alpha$ we have, by the support
    of $\eta_{\delta,r,\alpha}$, that
    \begin{equation*}
    	\bigl(1+t\eta_{\delta,r,\alpha}(r+\delta+\alpha)\bigr)
        (r+\delta+\alpha)=r+\delta+\alpha.
    \end{equation*}
    From surjectivity we can now find $y\in{}B_{r+\delta+\alpha}(0)$ such
    that $z=\bigl(1+t\eta_{\delta,r,\alpha}(|y|)\bigr)y$.
    Altogether we now have that
    $\varphi_{\delta,r,\alpha}(\Omega\times\{t\})=\Omega$.
    Notice that since $s\mapsto{}\bigl(1+t\eta_{\delta,r,\alpha}(s)\bigr)s$
    is increasing on $(0,r+\delta+\alpha)$ then it is injective
    which means that $y\mapsto{}\bigl(1+t\eta_{\delta,r,\alpha}(|y|)\bigr)y$
    is injective on $B_{r+\delta+\alpha}(0)$.
    From this we conclude that $\varphi_{\delta,r,\alpha}$ is bijective on
    $\Phi\bigl(B_{r+\delta+\alpha}(0)\bigr)$
    being the composition of injective maps.
    Since $\varphi_{\delta,r,\alpha}$ is defined as the identity map on
    $B_{2}(0)\setminus\Phi\bigl(B_{r+\delta+\alpha}(0)\bigr)$
    then we conclude that $\varphi_{\delta,r,\alpha}$ is injective on
    $B_{2}(0)$.
    Finally, we conclude that $\varphi_{\delta,t,\alpha}$ is a
    $C^{3}$-diffeomorphism
    for all $0\le{}t\le\frac{\delta}{2}$, $0<\delta,r<\frac{r_{0}}{2}$,
    $0<\alpha<r$, and $0<r_{0}<\frac{1}{4}$.
\end{proof}
Now that we have a suitable family of functions, we now demonstrate that the
Monotonicity Formula holds in our setting.

\begin{proposition}[Monotonicity Formula]\label{prop:mod_monotonicity_formula}
    Suppose $u\colon\Omega\to\mathbb{S}^{2}$ is a minimizer of
    \eqref{def:dir} and $\overline{u}$ is defined as in \eqref{eq:extension}.
    If $0<r_{0}<\frac{1}{4}$ is chosen sufficiently small
    and $f_{x_{0},r_{0}}\colon(0,\frac{r_{0}}{2})\to[0,\infty)$ is
    defined by
    \begin{equation*}
    	f_{x_{0},r_{0}}(r)\coloneqq
        \frac{1}{2r}
        \int_{\Phi(B_{r}(x_{0}))}\!{}\min\biggl\{1,\frac{1}{|x|}\biggr\}
        |\nabla_{x}\overline{u}|^{2}
        \, .
    \end{equation*}
    then there exists $C_{1},C_{2}\in\mathbb{R}$ independent of $x_{0}$
    such that
    \begin{equation*}
    	\frac{\mathrm{d}}{\mathrm{d}r}
        \biggl[e^{C_{1}r}f_{x_{0},r_{0}}(r)+C_{2}r\biggr]\ge 0
        \, ,
    \end{equation*}
    i.e.\ the map $r\mapsto{}e^{C_{1}r}f_{x_{0},r_{0}}(r)+C_{2}r$ is non-decreasing.
\end{proposition}

\begin{proof}
    In order to demonstrate this, we consider the family of maps
    $T_{\delta,r,\alpha}$ defined as in \eqref{def:infinnervar} using
    $\varphi_{\delta,r,\alpha}$ as in Lemma \ref{lem:monotonicityfamily}.
    We observe that
    \begin{align*}
    	T_{\delta,r,\alpha}(x)\coloneqq\frac{\mathrm{d}}{\mathrm{d}t}\bigg\vert_{t=0}\varphi_{\delta,r,\alpha}(x,t)
    	&=\chi_{\mathcal{O}_{x_{0},r_{0}}}(x)\eta_{\delta,r,\alpha}(|\Phi^{-1}(x)|)
    	\nabla\Phi\bigl(\Phi^{-1}(x)\bigr)\Phi^{-1}(x)\\
    	&=\eta_{\delta,r,\alpha}(|\Phi^{-1}(x)|)\nabla\Phi\bigl(\Phi^{-1}(x)\bigr)\Phi^{-1}(x)
    \end{align*}
    where we used that
    $\text{supp}(\eta_{\delta,r,\alpha})\subseteq\mathcal{O}_{x_{0},r_{0}}$
    in the final equality.

    In addition, we observe that
    \begin{align*}
    	\nabla{}T_{\delta,r,\alpha}(x)=&\eta_{\delta,r,\alpha}(|\Phi^{-1}(x)|)I_{3}
    	+\frac{\eta_{\delta,r,\alpha}'(|\Phi^{-1}(x)|)}{|\Phi^{-1}(x)|}\nabla\Phi\bigl(\Phi^{-1}(x)\bigr)\Phi^{-1}(x)
    	\Bigl[\nabla\Phi\bigl(\Phi^{-1}(x)\bigr)\Phi^{-1}(x)\Bigr]^{T}\\
    	&+\eta_{\delta,r,\alpha}(|\Phi^{-1}(x)|)\sum_{j=1}^{3}\sum_{i=1}^{3}\Phi^{-1}(x)^{T}
    	\nabla^{2}\Phi^{i}(\Phi^{-1}(x))\partial_{x_{j}}\Phi^{-1}(x)\mathbf{e}_{i}\mathbf{e}_{j}^{T}\\
    	\text{div}\bigl(T_{\delta,r,\alpha}(x)\bigr)=&3\eta_{\delta,r,\alpha}\bigl(|\Phi^{-1}(x)|\bigr)
    	+\frac{\eta_{\delta,r,\alpha}'(|\Phi^{-1}(x)|)\cdot|\nabla\Phi\bigl(\Phi^{-1}(x)\bigr)\Phi^{-1}(x)|^{2}}{|\Phi^{-1}(x)|}\\
    	&+\eta_{\delta,r,\alpha}(|\Phi^{-1}(x)|)\sum_{i=1}^{3}\Phi^{-1}(x)^{T}\nabla^{2}\Phi^{i}(\Phi^{-1}(x))\partial_{x_{i}}\Phi^{-1}(x).
    \end{align*}
    Combining this with Proposition \ref{prop:inner_variations_tang} we
    find that
    \begin{align*}
    	0
    	=&\int_{\Omega}\!{}\biggl\{-\eta_{\delta,r,\alpha}\bigl(|\Phi^{-1}(x)|\bigr)|\nabla_{x}\overline{u}(x)|^{2}
    	-\frac{\eta_{\delta,r,\alpha}'\bigl(|\Phi^{-1}(x)|\bigr)}{|\Phi^{-1}(x)|}|\nabla_{x}\overline{u}(x)\nabla\Phi\bigl(\Phi^{-1}(x)\bigr)\Phi^{-1}(x)|^{2}
    	\biggr\}\\
    	&-\sum_{i,j=1}^{3}\int_{\Omega}\!{}\biggl\{\eta_{\delta,r,\alpha}\bigl(|\Phi^{-1}(x)|\bigr)
    	\bigl[\partial_{x_{i}}\overline{u}(x)\cdot\partial_{x_{j}}\overline{u}(x)\bigr]\Bigl[\Phi^{-1}(x)^{T}
    	\nabla^{2}\Phi^{i}(\Phi^{-1}(x))\partial_{x_{j}}\Phi^{-1}(x)\Bigr]\biggr\}\\
    	&+\int_{\Omega}\!{}\biggl\{\frac{3}{2}|\nabla{}\overline{u}|^{2}\eta_{\delta,r,\alpha}\bigl(|\Phi^{-1}(x)|\bigr)
    	+\frac{1}{2}|\nabla{}\overline{u}|^{2}\frac{\eta_{\delta,r,\alpha}'(|\Phi^{-1}(x)|)|\nabla\Phi\bigl(\Phi^{-1}(x)\bigr)\Phi^{-1}(x)|^{2}}{|\Phi^{-1}(x)|}\biggr\}\\
    	&+\frac{1}{2}\sum_{i=1}^{3}\int_{\Omega}\!{}|\nabla{}\overline{u}|^{2}\eta_{\delta,r,\alpha}(|\Phi^{-1}(x)|)\bigl[\Phi^{-1}(x)^{T}\nabla^{2}\Phi^{i}(\Phi^{-1}(x))\partial_{x_{i}}\Phi^{-1}(x)\bigr]\\
    	&-\int_{B_{2}(0)\setminus\Omega}\!{}\frac{\eta_{\delta,r,\alpha}\bigl(|\Phi^{-1}(x)|\bigr)}{|x|^{2}}|\nabla_{x}\overline{u}(x)|^{2}
    	-\int_{B_{2}(0)\setminus\Omega}\!{}\frac{\eta_{\delta,r,\alpha}'\bigl(|\Phi^{-1}(x)|\bigr)}{|\Phi^{-1}(x)|}
    	|\nabla_{x}\overline{u}(x)\nabla\Phi\bigl(\Phi^{-1}(x)\bigr)\Phi^{-1}(x)|^{2}\\
    	&-\sum_{i,j=1}^{3}\int_{B_{2}(0)\setminus\Omega}\!{}\frac{1}{|x|^{2}}\biggl\{\eta_{\delta,r,\alpha}\bigl(|\Phi^{-1}(x)|\bigr)
    	\bigl[\partial_{x_{i}}\overline{u}(x)\cdot\partial_{x_{j}}\overline{u}(x)\bigr]\Bigl[\Phi^{-1}(x)^{T}
    	\nabla^{2}\Phi^{i}(\Phi^{-1}(x))\partial_{x_{j}}\Phi^{-1}(x)\Bigr]\biggr\}\\
    	&+\frac{3}{2}\int_{B_{2}(0)\setminus\Omega}
    	\frac{1}{|x|^{2}}|\nabla_{x}\overline{u}(x)|^{2}\eta_{\delta,r,\alpha}\bigl(|\Phi^{-1}(x)|\bigr)
    	+\frac{1}{2}\int_{B_{2}(0)\setminus\Omega}\!{}\frac{1}{|x|^{2}}
    	|\nabla{}\overline{u}|^{2}|\frac{\eta_{\delta,r,\alpha}'(|\Phi^{-1}(x)|)|\nabla\Phi\bigl(\Phi^{-1}(x)\bigr)\Phi^{-1}(x)|^{2}}{|\Phi^{-1}(x)|}\\
    	&+\frac{1}{2}\sum_{i=1}^{3}\int_{\Omega}\!{}|\nabla{}\overline{u}|^{2}\eta_{\delta,r,\alpha}(|\Phi^{-1}(x)|)\bigl[\Phi^{-1}(x)^{T}\nabla^{2}\Phi^{i}(\Phi^{-1}(x))\partial_{x_{i}}\Phi^{-1}(x)\bigr]\\
    	&+\int_{B_{2}(0)\setminus\Omega}\!{}
    	\mathcal{F}(x,\overline{u}(x),\nabla_{x}\overline{u}(x))\cdot{}T_{\delta,r,\alpha}(x)
    	+\int_{B_{2}(0)\setminus\Omega}\!{}\mathcal{G}(x,\overline{u}(x),\nabla_{x}\overline{u}(x)):\nabla_{x}T_{\delta,r,\alpha}(x)
    \end{align*}
    for all $0<\delta,r<\frac{r_{0}}{2}$ and $0<\alpha<r$.
    First we let $\alpha\to0^{+}$ and use the Dominated Convergence Theorem
    to remove the parameter from the above equation.
    After this, we apply the Change of Variables Theorem with
    $\Phi$, let $\delta\to0^{+}$, and using convergence of the
    mollifier in $L^{2}$ to obtain
   \begin{align*}
    	0
    	=&\frac{1}{2}\int_{B_{r}(0)}\!{}
        \min\biggl\{1,\frac{1}{|\Phi(y)|^{2}}\biggr\}|\nabla_{x}\overline{u}(\Phi(y))|^{2}|\det(\nabla\Phi(y))|\\
    	&+\frac{1}{r}\int_{\partial{}B_{r}(0)}\!{}
        \min\biggl\{1,\frac{1}{|\Phi(y)|^{2}}\biggr\}|\nabla_{x}\overline{u}(\Phi(y))\nabla\Phi\bigl(y\bigr)y|^{2}
    	\biggr\}|\det(\nabla\Phi(y))|\\
    	&-\sum_{i,j=1}^{3}\int_{B_{r}(0)}\!{}
        \min\biggl\{1,\frac{1}{|\Phi(y)|^{2}}\biggr\}\biggl\{
    	\bigl[\partial_{x_{i}}\overline{u}(\Phi(y))\cdot\partial_{x_{j}}\overline{u}(\Phi(y))\bigr]\Bigl[y^{T}
    	\nabla^{2}\Phi^{i}(y)\partial_{x_{j}}\Phi^{-1}(\Phi(y))\Bigr]\biggr\}|\det(\nabla\Phi(y))|\\
    	&-\frac{1}{2r}\int_{\partial{}B_{r}(0)}\!{}
        \min\biggl\{1,\frac{1}{|\Phi(y)|^{2}}\biggr\}\biggl\{
    	|\nabla{}\overline{u}(\Phi(y))|^{2}|\nabla\Phi\bigl(y\bigr)y|^{2}\biggr\}|\det(\nabla\Phi(y))|\\
    	&+\frac{1}{2}\sum_{i=1}^{3}\int_{B_{r}^{+}(0)}\!{}
        \min\biggl\{1,\frac{1}{|\Phi(y)|^{2}}\biggr\}|\nabla{}\overline{u}(\Phi(y))|^{2}\bigl[y^{T}\nabla^{2}\Phi^{i}(y)\partial_{x_{i}}\Phi^{-1}(\Phi(y))\bigr]|\det(\nabla\Phi(y))|\\
    	&+\int_{B_{r}^{-}(0)}\!{}
    	\mathcal{F}(\Phi(y),\overline{u}(\Phi(y)),\nabla_{x}\overline{u}(\Phi(y)))\cdot{}\nabla\Phi\bigl(y\bigr)y|\det(\nabla\Phi(y))|\\
    	&+\int_{B_{r}^{-}(0)}\!{}\mathcal{G}(\Phi(y),\overline{u}(\Phi(y)),\nabla_{x}\overline{u}(\Phi(y))):I_{3}|\det(\nabla\Phi(y))|\\
    	&-\frac{1}{r}\int_{\partial{}B_{r}^{-}(0)}\!{}\mathcal{G}(\Phi(y),\overline{u}(\Phi(y)),\nabla_{x}\overline{u}(\Phi(y))):
    	\nabla\Phi\bigl(y\bigr)y\bigl[\nabla\Phi\bigl(y\bigr)y\bigr]^{T}|\det(\nabla\Phi(y))|\\
    	&+\sum_{i,j=1}^{3}
    	\int_{B_{r}^{-}(0)}\!{}\mathcal{G}(\Phi(y),\overline{u}(\Phi(y)),\nabla_{x}\overline{u}(\Phi(y))):\bigl[y^{T}
    	\nabla^{2}\Phi^{i}(y)\partial_{x_{j}}\Phi^{-1}(\Phi(y))\bigr]\mathbf{e}_{i}\mathbf{e}_{j}^{T}|\det(\nabla\Phi(y))|.
    \end{align*}
    Rearranging, dividing by $r^{2}$ and adding
    \begin{equation*}
    	\frac{1}{2r}\int_{\partial{}B_{r}(0)}\!{}
        \min\biggl\{1,\frac{1}{|\Phi(y)|^{2}}\biggr\}|\nabla_{x}\overline{u}(\Phi(y))|^{2}|\det(\nabla\Phi(y))|
    \end{equation*}
    to both sides  we obtain 
    \begin{align*}
    	&\frac{\mathrm{d}}{\mathrm{d}r}\biggl[
    	\frac{1}{2r}\int_{B_{r}(0)}\!{}
        \min\biggl\{1,\frac{1}{|\Phi(y)|^{2}}\biggr\}|\nabla_{x}\overline{u}(\Phi(y))|^{2}|\det(\nabla\Phi(y))|\biggr]\\
    	&+\sum_{i,j=1}^{3}\frac{1}{r^{2}}\int_{B_{r}(0)}\!{}
        \min\biggl\{1,\frac{1}{|\Phi(y)|^{2}}\biggr\}\biggl\{
    	\bigl[\partial_{x_{i}}\overline{u}(\Phi(y))\cdot\partial_{x_{j}}\overline{u}(\Phi(y))\bigr]\Bigl[y^{T}
    	\nabla^{2}\Phi^{i}(y)\partial_{x_{j}}\Phi^{-1}(\Phi(y))\Bigr]\biggr\}|\det(\nabla\Phi(y))|\\
    	&-\frac{1}{2r^{2}}\sum_{i=1}^{3}\int_{B_{r}(0)}\!{}
        \min\biggl\{1,\frac{1}{|\Phi(y)|^{2}}\biggr\}|\nabla{}\overline{u}(\Phi(y))|^{2}\bigl[y^{T}\nabla^{2}\Phi^{i}(y)\partial_{x_{i}}\Phi^{-1}(\Phi(y))\bigr]|\det(\nabla\Phi(y))|\\
    	&=
        \frac{1}{2r}\int_{\partial{}B_{r}(0)}\!{}
        \min\biggl\{1,\frac{1}{|\Phi(y)|^{2}}\biggr\}
        |\nabla_{x}\overline{u}(\Phi(y))|^{2}|\det(\nabla\Phi(y))|\\
    	&\quad
        -\frac{1}{2r^{3}}\int_{\partial{}B_{r}(0)}\!{}
        \min\biggl\{1,\frac{1}{|\Phi(y)|^{2}}\biggr\}
    	|\nabla{}\overline{u}(\Phi(y))|^{2}|\nabla\Phi\bigl(y\bigr)y|^{2}|\det(\nabla\Phi(y))|\\
    	&\quad
    	+\frac{1}{r^{3}}\int_{\partial{}B_{r}(0)}\!{}
        \min\biggl\{1,\frac{1}{|\Phi(y)|^{2}}\biggr\}|\nabla_{x}\overline{u}(\Phi(y))\nabla\Phi\bigl(y\bigr)y|^{2}|\det(\nabla\Phi(y))|\\
    	&\quad
        +\frac{1}{r^{2}}\int_{B_{r}^{-}(0)}\!{}
    	\mathcal{F}(\Phi(y),\overline{u}(\Phi(y)),\nabla_{x}\overline{u}(\Phi(y)))\cdot{}\nabla\Phi\bigl(y\bigr)y|\det(\nabla\Phi(y))|\\
    	&\quad
        +\frac{1}{r^{2}}\int_{B_{r}^{-}(0)}\!{}\mathcal{G}(\Phi(y),\overline{u}(\Phi(y)),\nabla_{x}\overline{u}(\Phi(y))):I_{3}|\det(\nabla\Phi(y))|\\
    	&\quad
        -\frac{1}{r^{3}}\int_{\partial{}B_{r}^{-}(0)}\!{}\mathcal{G}(\Phi(y),\overline{u}(\Phi(y)),\nabla_{x}\overline{u}(\Phi(y))):
    	\nabla\Phi\bigl(y\bigr)y\bigl[\nabla\Phi\bigl(y\bigr)y\bigr]^{T}|\det(\nabla\Phi(y))|\\
    	&\quad
        +\sum_{i,j=1}^{3}
    	\frac{1}{r^{2}}\int_{B_{r}^{-}(0)}\!{}\mathcal{G}(\Phi(y),\overline{u}(\Phi(y)),\nabla_{x}\overline{u}(\Phi(y))):\bigl[y^{T}
    	\nabla^{2}\Phi^{i}(y)\partial_{x_{j}}\Phi^{-1}(\Phi(y))\bigr]\mathbf{e}_{i}\mathbf{e}_{j}^{T}|\det(\nabla\Phi(y))|.
    \end{align*}
    Using that on $B_{r}(0)$ we have
    \begin{equation*}
    	\bigl|y^{T}\nabla^{2}\Phi^{i}(y)\partial_{x_{j}}\Phi^{-1}(\Phi(y))\bigr|\le{}C_{1}r\hspace{15pt}
        |\nabla\Phi(y)y|\le{}(1+C_{2}r)|y|\le{}(1+C_{2}r)r
    \end{equation*}
    for $i=1,2,3$ and combining this with the inequality for the derivative
    in $r$ we find that there are constants $C_{3},C_{4}>0$ such that
    \begin{align*}
    	&(1+C_{3}r)\frac{\mathrm{d}}{\mathrm{d}r}\biggl[
    	\frac{1}{2r}\int_{B_{r}(0)}\!{}\min\biggl\{1,\frac{1}{|\Phi(y)|^{2}}\biggr\}
        |\nabla_{x}\overline{u}(\Phi(y))|^{2}|\det(\nabla\Phi(y))|
    	\biggr]\\
    	\ge
    	&-C_{4}\biggl[\frac{1}{2r}\int_{B_{r}(0)}\!{}
        \min\biggl\{1,\frac{1}{|\Phi(y)|^{2}}\biggr\}
    	|\nabla{}\overline{u}(\Phi(y))|^{2}|\det(\nabla\Phi(y))|
    	\biggr]\\
    	&
    	+\frac{1}{r^{3}}\int_{\partial{}B_{r}(0)}\!{}
        \min\biggl\{1,\frac{1}{|\Phi(y)|^{2}}\biggr\}|\nabla_{x}\overline{u}(\Phi(y))\nabla\Phi\bigl(y\bigr)y|^{2}|\det(\nabla\Phi(y))|\\
    	&+\frac{1}{r^{2}}\int_{B_{r}^{-}(0)}\!{}
    	\mathcal{F}(\Phi(y),\overline{u}(\Phi(y)),\nabla_{x}\overline{u}(\Phi(y)))\cdot{}\nabla\Phi\bigl(y\bigr)y|\det(\nabla\Phi(y))|\\
    	&+\frac{1}{r^{2}}\int_{B_{r}^{-}(0)}\!{}\mathcal{G}(\Phi(y),\overline{u}(\Phi(y)),\nabla_{x}\overline{u}(\Phi(y))):I_{3}|\det(\nabla\Phi(y))|\\
    	&-\frac{1}{r^{3}}\int_{\partial{}B_{r}^{-}(0)}\!{}\mathcal{G}(\Phi(y),\overline{u}(\Phi(y)),\nabla_{x}\overline{u}(\Phi(y))):
    	\nabla\Phi\bigl(y\bigr)y\bigl[\nabla\Phi\bigl(y\bigr)y\bigr]^{T}|\det(\nabla\Phi(y))|\\
    	&+\sum_{i,j=1}^{3}
    	\frac{1}{r^{2}}\int_{B_{r}^{-}(0)}\!{}\mathcal{G}(\Phi(y),\overline{u}(\Phi(y)),\nabla_{x}\overline{u}(\Phi(y))):\bigl[y^{T}
    	\nabla^{2}\Phi^{i}(y)\partial_{x_{j}}\Phi^{-1}(\Phi(y))\bigr]\mathbf{e}_{i}\mathbf{e}_{j}^{T}|\det(\nabla\Phi(y))|.
    \end{align*}
    Next, using that $|\mathcal{F}(x,z,p)|\le{}C\bigl[1+|p|^{2}\bigr]$, $|\mathcal{G}(x,z,p)|\le{}C(1+|p|)$, we can estimate this from
    below and rewrite it as
    \begin{align*}
    	&(1+C_{3}r)\frac{\mathrm{d}}{\mathrm{d}r}\biggl[
    	\frac{1}{2r}\int_{B_{r}(0)}\!{}
        \min\biggl\{1,\frac{1}{|\Phi(y)|^{2}}\biggr\}|\nabla_{x}\overline{u}(\Phi(y))|^{2}|\det(\nabla\Phi(y))|
    	\biggr]\\
    	\ge
    	&-C_{4}\biggl[\frac{1}{2r}\int_{B_{r}(0)}\!{}
        \min\biggl\{1,\frac{1}{|\Phi(y)|^{2}}\biggr\}
    	|\nabla{}\overline{u}(\Phi(y))|^{2}|\det(\nabla\Phi(y))|
    	\biggr]\\
    	&
    	+\frac{1}{r^{3}}\int_{\partial{}B_{r}(0)}
        \min\biggl\{1,\frac{1}{|\Phi(y)|^{2}}\biggr\}|\nabla_{x}\overline{u}(\Phi(y))\nabla\Phi\bigl(y\bigr)y|^{2}|\det(\nabla\Phi(y))|\\
    	&-C_{5}-C_{6}\frac{1}{2r}\int_{B_{r}^{-}(0)}\!{}\frac{1}{|\Phi(y)|^{2}}
    	|\nabla_{x}\overline{u}\bigl(\Phi(y)\bigr)|^{2}|\det(\nabla\Phi(y))|\\
    	&-C_{7}\biggl[\frac{1}{2r}\int_{B_{r}(0)}\!{}
        \min\biggl\{1,\frac{1}{|\Phi(y)|^{2}}\biggr\}
        |\nabla_{x}\overline{u}\bigl(\Phi(y)\bigr)|^{2}|\det(\nabla\Phi(y))|\biggr]\\
    	\ge&\frac{1}{r^{3}}\int_{\partial{}B_{r}(0)}\!{}
        \min\biggl\{1,\frac{1}{|\Phi(y)|^{2}}\biggr\}|\nabla_{x}\overline{u}(\Phi(y))\nabla\Phi\bigl(y\bigr)y|^{2}|\det(\nabla\Phi(y))|\\
    	&-(C_{4}+C_{6}+C_{7})
    	\biggl[\frac{1}{2r}\int_{B_{r}(0)}\!{}
        \min\biggl\{1,\frac{1}{|\Phi(y)|^{2}}\biggr\}
        |\nabla_{x}\overline{u}\bigl(\Phi(y)\bigr)|^{2}|\det(\nabla\Phi(y))|\biggr]\\
    	&-C_{5}.
    \end{align*}
    Choosing $r_{0}>0$ sufficiently small, dividing both sides by $1+C_{3}r$, and bounding from below gives
    \begin{align}\label{eq:estim}
    	&\frac{\mathrm{d}}{\mathrm{d}r}\biggl[
    	\frac{1}{2r}\int_{B_{r}(0)}\!{}\biggl\{
        \min\biggl\{1,\frac{1}{|\Phi(y)|^{2}}\biggr\}|\nabla_{x}\overline{u}(\Phi(y))|^{2}\biggr\}|\det(\nabla\Phi(y))|
    	\biggr]\\ \nonumber
    	\ge&\overline{C}_{1}
    	\biggl[\frac{1}{r^{3}}\int_{\partial{}B_{r}(0)}\!{}
        \min\biggl\{1,\frac{1}{|\Phi(y)|^{2}}\biggr\}
        |\nabla_{x}\overline{u}(\Phi(y))\nabla\Phi\bigl(y\bigr)y|^{2}|\det(\nabla\Phi(y))|
    	\biggr]\\ \nonumber
    	&-\overline{C}_{2}
    	\biggl[\frac{1}{2r}\int_{B_{r}(0)}\!{}
        \min\biggl\{1,\frac{1}{|\Phi(y)|^{2}}\biggr\}
        |\nabla_{x}\overline{u}\bigl(\Phi(y)\bigr)|^{2}|\det(\nabla\Phi(y))|\biggr]\\ \nonumber
    	&-\overline{C}_{3}.
    \end{align}
    Similar to the formal calculation we let $F_{x_{0}}\colon(0,\frac{r_{0}}{2})\to\mathbb{R}$ be defined by
    \begin{equation*}
    	F_{x_{0}}(r)\coloneqq\frac{1}{2r}\int_{B_{r}(0)}\!{}
        \min\biggl\{1,\frac{1}{|\Phi(y)|^{2}}\biggr\}|\nabla_{x}\overline{u}(\Phi(y))|^{2}|\det(\nabla\Phi(y))|.
    \end{equation*}
    Then our above calculation gives
    \begin{equation*}
    	F_{x_{0}}'(r)\ge-\overline{C}_{2}F_{x_{0}}(r)-\overline{C}_{3}.
    \end{equation*}
    As in the formal calculation this gives that the function
    \begin{equation*}
    	r\mapsto{}e^{\overline{C}_{2}r}F_{x_{0}}(r)+\overline{C}_{3}r
    \end{equation*}
    is non-decreasing.
    Applying the Change of Variables Theorem with $\Phi^{-1}$ now shows that the following function is non-decreasing
    \begin{align*}
    	r\mapsto{}e^{\overline{C}_{2}r}\biggl[\frac{1}{2r}\int_{\Phi(B_{r}(x_{0}))\cap\Omega}\!{}
        \min\biggl\{1,\frac{1}{|x|^{2}}\biggr\}|\nabla_{x}\overline{u}|^{2}\biggr]+\overline{C}_{3}r.
    \end{align*}
    Notice that the constants we obtained only depended on estimates of
    $\Phi$ which are independent of $x_{0}$ since similar local
    representations about any point on $\partial\Omega$ can be obtained by
    applying a rotation to $\Phi$.
\end{proof}
Next we provide a Corollary that will be useful for proving partial regularity.
Before we begin we introduce the following notation for the $1$-density
of the energy measure
\begin{align*}
    \Theta(\overline{u},x_{0},s)\coloneqq\frac{1}{2s}\int_{B_{s}(x_{0})}\!{}
|\nabla{}\overline{u}|^{2}
\, .
\end{align*}

\begin{corollary}\label{corr:densbehav}
    Suppose $u\colon\Omega\to\mathbb{S}^{2}$ is a minimizer of
    \eqref{def:dir}, $\overline{u}$ is defined by \eqref{eq:extension},
    and $x_{0}\in\partial\Omega$.
    Then,
    \begin{equation}\label{eq:dirdenslimit}
        \lim_{s\to0^{+}}\Theta(\overline{u},x_{0},s)
        =\lim_{s\to0^{+}}f_{x_{0},r_{0}}(s)
    \end{equation}
    exists and there exists $\overline{r}_{0}>0$ such that
    \begin{equation}\label{eq:dirdensbound}
        \sup_{0<r<\overline{r}_{0}}
        \biggl\{\Theta(\overline{u},x_{0},r)\biggr\}
        \le{}\overline{C}
    \end{equation}
    where $\overline{C}$ is independent of $x_{0}$.
\end{corollary}

Corollary~\ref{corr:densbehav} justifies the definition of
\begin{align*}
    \Theta(\overline{u},x_{0})
    \ \coloneqq \
    \lim_{s\to0^{+}}\Theta(\overline{u},x_{0},s)
    \, .
\end{align*}

\begin{proof}[Proof of Corollary~\ref{corr:densbehav}]
    First, observe that \eqref{eq:dirdensbound} follows from
    \eqref{eq:dirdenslimit} and so it suffices to demonstrate that
    \eqref{eq:dirdenslimit} holds.
    Let $0<r_{0}<\frac{1}{4}$ be chosen as in Proposition
    \ref{prop:mod_monotonicity_formula}.
    Since $\Phi$ satisfies \eqref{eq:phileadingorder} then there exists
    $c_{0}>0$
    such that
    \begin{equation*}
        |\Phi^{-1}(x^{2})-\Phi^{-1}(x^{1})|\le{}(1+c_{0}s)|x^{2}-x^{1}|
    \end{equation*}
    for all $x^{1},x^{2}\in{}\Phi\bigl(B_{s}(0)\bigr)$ where
    $0<s<\frac{r_{0}}{4}$.
    In addition, we have that $\Phi\bigl(B_{\frac{r_{0}}{4}}(0)\bigr)$ is
    open and hence there exists $s>0$ such that
    $B_{s}(x_{0})\subseteq\Phi\bigl(B_{\frac{r_{0}}{4}}(0)\bigr)$.
    Next we define $s_{0}>0$ by
    \begin{equation*}
        s_{0}\coloneqq\sup\bigl\{
        s>0:B_{s}(x_{0})\subseteq{}\Phi\bigl(B_{\frac{r_{0}}{4}}(0)\bigr)
        \bigr\}.
    \end{equation*}
    For $0<s<\min\bigl\{s_{0},\frac{r_{0}}{4(1+c_{0})}\bigr\}$ we now have that
    $\Phi^{-1}\bigl(B_{s}(x_{0})\bigr)\subseteq{}B_{(1+c_{0}s)s}(0)$
    and hence
    \begin{equation*}
        B_{s}(x_{0})\subseteq\Phi\bigl(B_{(1+c_{0}s)s}(0)\bigr).
    \end{equation*}
    Next, we see that for all
    $0<s<\min\bigl\{s_{0},\frac{r_{0}}{4(1+c_{0})}\bigr\}$ we have
    \begin{align*}
        &e^{C_{1}s}\biggl\{\frac{1}{2s}\int_{B_{s}(x_{0})}\!{}
        |\nabla{}\overline{u}|^{2}\biggr\}+C_{2}s\\
        \le&(1+s)^{2}e^{C_{1}s}\biggl\{\frac{1}{2s}
        \int_{B_{s}(x_{0})}\!{}
        \min\biggl\{1,\frac{1}{|x|^{2}}\biggr\}|\nabla\overline{u}|^{2}
        \biggr\}+C_{2}s\\
        \le&(1+s)^{2}e^{C_{1}s}\biggl\{\frac{1}{2s}
        \int_{\Phi(B_{(1+c_{0}s)s}(0))}\!{}
        \min\biggl\{1,\frac{1}{|x|^{2}}\biggr\}|\nabla\overline{u}|^{2}
        \biggr\}+C_{2}s\\
        \le&(1+s)^{2}e^{|C_{1}c_{0}|s}
        \max\biggl\{1+c_{0}s,\frac{1}{1+c_{0}s}\biggr\}f_{x_{0},r_{0}}((1+c_{0}s)s).
    \end{align*}
    Similarly, we have, since $\Phi$ satisfies \eqref{eq:phileadingorder},
    that there is $c_{1}>0$ such that
    \begin{equation*}
        |\Phi(y^{2})-\Phi(y^{1})|\le{}(1+c_{1}s)|y^{2}-y^{1}|
    \end{equation*}
    for all $y^{1},y^{2}\in{}B_{s}(0)$.
    Hence, for $0<s<\frac{r_{0}}{4}$ we have
    \begin{equation*}
        \Phi\bigl(B_{s}(0)\bigr)\subseteq{}B_{(1+c_{1}s)s}(x_{0}).
    \end{equation*}
    We conclude that for all $0<s<\frac{r_{0}}{4}$ we have
    \begin{equation*}
        f_{x_{0},r_{0}}(s)\le{}
        e^{|C_{1}c_{1}|s}\max\biggl\{1+c_{1}s,\frac{1}{1+c_{1}s}\biggr\}
        \biggl[{}e^{C_{1}(1+c_{1}s)s}\biggl\{
        \frac{1}{2(1+c_{1}s)s}\int_{B_{(1+c_{1}s)s}(0)}\!{}
        |\nabla\overline{u}|^{2}
        \biggr\}+C_{2}(1+c_{1}s)s\biggr].
    \end{equation*}
    For $0<s<\min\bigl\{\frac{s_{0}}{1+c_{1}},\frac{r_{0}}{4(1+c_{0})(1+c_{1})
    }\bigr\}$ we find, after combining the above inequalities with
    Proposition \ref{prop:mod_monotonicity_formula} and the Squeeze Theorem
    we now have that
    \begin{equation*}
        \lim_{s\to0^{+}}\biggl\{\frac{1}{2(1+c_{1}s)s}
        \int_{B_{(1+c_{1}s)s}(0)}\!{}
        |\nabla{}\overline{u}|^{2}\biggr\}
        =\lim_{s\to0^{+}}f_{x_{0},r_{0}}(s).
    \end{equation*}
\end{proof}

The Monotonicity Formula can now be used to establish regularity results pertaining to minimizers of \eqref{def:dir}.
In order to do this, we use the arguments of
\cite{DiMiPi} in order to obtain H\"o{}lder regularity when the $1$-density of the energy measure is small.

\begin{proposition}\label{prop:part_reg}
    Let the singular set $\Sigma(u)\subset\overline{B_1(0)}$ be defined as
    \begin{equation}\label{def:singset}
    \Sigma(u)
    \ \coloneqq \
    \Big\{ x\in\overline{B_1(0)} \sd \Theta(u,x) > 0 \Big\}
    \, .
    \end{equation}
    Then $\mathcal{H}^1(\Sigma(u))=0$ and $u\in C^\infty(B_1(0)\setminus\Sigma(u))$. 
\end{proposition}

\begin{proof}
    The goal of this proof is to apply Theorem 2.12 in \cite{DiMiPi}.
    We therefore need to check the assumptions of this Theorem. 
    From Lemma~\ref{lem:global_PDE} we know that $u$ solves a PDE of the required form \cite[(2.44)]{DiMiPi}.
    By Corollary~\ref{corr:densbehav} we know that $\sup_{0<r<r_0}\Theta(u,x,r)\leq \overline{C}$ for some constant $\overline{C}$, which implies \cite[(2.45)]{DiMiPi}. 
    We therefore conclude that $u$ is $C^{0,\alpha}-$regular outside of $\Sigma(u)$ and by Corollary~2.19 in \cite{DiMiPi}, one obtains that $u$ is analytic away from $\Sigma(u)$.

    The fact that $\mathcal{H}^1(\Sigma(u))=0$ follows from a covering argument, see e.g.\ \cite[Proposition 9.21]{GiMe}.  
\end{proof}

We conclude this subsection with a corollary demonstrating that we
can obtain $L^{2}$ control of the radial derivative on a set diffeomorphic
to an annuli centered on $\partial\Omega$.
This will be important for establishing full regularity in Subsection
\ref{subsec:fullreg}. To simplify again the notation, we will drop the dependence on $r_0$ in the definition of $f_{x_0,r_0}$.

\begin{corollary}\label{cor:radialmono}
    Suppose $u\colon\Omega\to\mathbb{S}^{2}$ is a minimizer for
    \eqref{def:dir} and $\overline{u}$ is an extension defined as in
    \eqref{eq:extension}.
    Suppose also that $x_{0}\in\partial\Omega$.
    Then for $0<r_{0}<\frac{1}{4}$ chosen sufficiently small there exists
    $C>0$ such that for $0<r_{1}<r_{2}<\frac{r_{0}}{2}$ we have
    \begin{align*}
        C\int_{\Phi(B_{r_{2}}(0))\setminus\Phi(B_{r_{1}}(0))}\!{}
        &\frac{1}{|x-x_{0}|}
        \biggl|\nabla\overline{u}(x)\cdot\frac{x-x_{0}}{|x-x_{0}|}\biggr|^{2}\\
        &\le{}f_{x_{0}}(r_{2})-f_{x_{0}}(r_{1})
        +E(\overline{u},
        \Phi(B_{r_{2}}(0))\setminus{}\Phi(B_{r_{1}}(0))).
    \end{align*}
\end{corollary}

\begin{proof}
    From the inequality \eqref{eq:estim} in the proof of Proposition~\ref{prop:mod_monotonicity_formula} we have that
    there exists a constant $c_{0}$ such that
    \begin{equation*}
        c_{0}\biggl[\frac{1}{r^{3}}\int_{\partial{}B_{r}(0)}
        \min\biggl\{1,\frac{1}{|\Phi(y)|^{2}}\biggr\}|\nabla_{x}\overline{u}(\Phi(y))\nabla\Phi\bigl(y\bigr)y|^{2}|\det(\nabla\Phi(y))|\biggr]
        \le\frac{\mathrm{d}f_{x_{0}}}{\mathrm{d}r}(r)
    \end{equation*}
    for $0<r<r_{0}$ where $r_{0}>0$ is chosen sufficiently small.
    Considering $0<r_{1}<r_{2}<r_{0}$, using absolute continuity of
    $r\mapsto{}\frac{1}{r^{3}}\int_{B_{r}(0)}\!{}
    \min\bigl\{1,\frac{1}{|\Phi(y)|^{2}}\bigr\}
    |\nabla_{x}\overline{u}(\Phi(y))\nabla\Phi(y)y|^{2}|\det(\nabla\Phi(y))|$, as well as integrating this inequality over
    $[r_{1},r_{2}]$ gives
    \begin{align*}
        c_{0}\biggl[\int_{A_{r_{1},r_{2}}(0)}
        \min\biggl\{1,\frac{1}{|\Phi(y)|^{2}}\biggr\}\frac{1}{|y|}\biggl|\nabla_{x}\overline{u}(\Phi(y))\nabla\Phi\bigl(y\bigr)\frac{y}{|y|}\biggr|^{2}
        |\det(\nabla\Phi(y))|\biggr]
        \le{}f_{x_{0}}(r_{2})-f_{x_{0}}(r_{1})
    \end{align*}
    where $A_{r_{1},r_{2}}(0)\coloneqq{}B_{r_{2}}(0)\setminus{}B_{r_{1}}(0)$.
    Making the change of variables $y=\Phi^{-1}(x)$ gives
    \begin{equation*}
        c_{0}\biggl[\int_{\Phi(B_{r_{2}}(0))\setminus
        \Phi(B_{r_{1}}(0))}\!{}
        \min\biggl\{1,\frac{1}{|x|^{2}}\biggr\}
        \frac{1}{|\Phi^{-1}(x)|}\biggl|\nabla_{x}\overline{u}(x)\nabla\Phi\bigl(\Phi^{-1}(x)\bigr)\frac{\Phi^{-1}(x)}{|\Phi^{-1}(x)|}\biggr|^{2}
        \biggr]
        \le{}f_{x_{0}}(r_{2})-f_{x_{0}}(r_{1}).
    \end{equation*}
    By the identities established at the beginning of
    Subsubsection~\ref{subsubsec:rigorous_calc} as well as that
    \begin{align*}
        |\Phi^{-1}(x)|^{2}&=
        \frac{x_{1}^{2}+x_{2}^{2}+|x|^{2}(|x|-1)^{2}}{|x|^{2}}
        \le\frac{(2+|x|^{2})|x-x_{0}|^{2}}{|x|^{2}}
        \le\frac{(2+r_{0})^{2}|x-x_{0}|^{2}}{|x|^{2}},\\
        x_{3}(|x|-1)&=(x_{3}-1)+O(|x-x_{0}|^{2}),\\
        |x|(x_{3}-|x|^{2})&=-(x_{3}-1)+O(|x-x_{0}|^{2}),
    \end{align*}
    where the last two identities follow from Taylor's Theorem applied
    near $x=x_{0}$,
    we have
    \begin{align*}
        \biggl|\nabla{}\overline{u}(x)\nabla\Phi\bigl(\Phi^{-1}(x)\bigr)
        \frac{\Phi^{-1}(x)}{|\Phi^{-1}(x)|}\biggr|
        &\ge\frac{2|x|}{2+r_{0}}\biggl|\nabla\overline{u}(x)\frac{x-x_{0}}{|x-x_{0}|}\biggr|
        -c_{1}|x-x_{0}||\nabla\overline{u}(x)|
    \end{align*}
    and hence, since $\frac{1}{2}x^{2}-y^{2}\le(x-y)^{2}$, we have
    \begin{equation*}
        \biggl|\nabla\overline{u}(x)\nabla\Phi\bigl(\Phi^{-1}(x)\bigr)
        \frac{\Phi^{-1}(x)}{|\Phi^{-1}(x)|}\biggr|^{2}
        \ge\frac{2|x|^{2}}{(2+r_{0})^{2}}\biggl|\nabla\overline{u}(x)\frac{x-x_{0}}{|x-x_{0}|}\biggr|^{2}
        -c_{1}^{2}|x-x_{0}|^{2}|\nabla\overline{u}(x)|^{2}.
    \end{equation*}
    Notice that, considering only the error term and using that
    $\frac{1}{|x|}\le1$ for $|x|\ge1$, we have that
    \begin{align*}
        c_{1}^{2}\int_{\Phi(B_{r_{2}}(0))\setminus\Phi(B_{r_{1}}(0))}\!{}
        \min\biggl\{1,\frac{1}{|x|^{2}}\biggr\}
        \frac{|x-x_{0}|^{2}}{|\Phi^{-1}(x)|}|\nabla\overline{u}(x)|^{2}
        \le{}c_{1}^{2}(1+r_{0})^{2}E(\overline{u},
        \Phi(B_{r_{2}}(0))\setminus\Phi(B_{r_{1}}(0))).
    \end{align*}
    Altogether this gives that there is $c_{2}>0$
    \begin{align*}
        c_{2}\biggl[\int_{\Phi(B_{r_{2}}(0))\setminus\Phi(B_{r_{1}}(0))}\!{}
        &\min\biggl\{1,\frac{1}{|x|^{2}}\biggr\}
        \frac{1}{|\Phi^{-1}(x)|}\biggl|\nabla\overline{u}(x)\cdot{}\frac{x-x_{0}}{|x-x_{0}|}\biggr|^{2}\\
        \le&(f_{x_{0}}(r_{2})-f_{x_{0}}(r_{1}))
        +c_{1}^{2}(1+r_{0})^{2}\overline{E}(\overline{u},
        \Phi(B_{r_{2}}(0))\setminus{}\Phi(B_{r_{1}}(0)))\\
        \le&\max\biggl\{1,c_{1}^{2}(1+r_{0})^{2}\biggr\}
        \Bigl[(f_{x_{0}}(r_{2})-f_{x_{0}}(r_{1}))
        +E(\overline{u},
        \Phi(B_{r_{2}}(0)_\setminus{}\Phi(B_{r_{1}}(0)))\Bigr].
    \end{align*}
    Using that $\min\bigl\{1,\frac{1}{|x|^{2}}\bigr\}\ge\frac{1}{4}$ on $B_{2}(0)$,
    that $|\Phi^{-1}(x)|\le\frac{|x|}{2+r_{0}}|x-x_{0}|$,
    as well as rearranging constants that there is $c_{3}>0$ such
    that
    \begin{align*}
        c_{2}\biggl[\int_{\Phi(B_{r_{2}}(0))\setminus\Phi(B_{r_{1}}(0))}\!{}
        &
        \frac{1}{|x-x_{0}|}\biggl|\nabla\overline{u}(x)\cdot{}\frac{x-x_{0}}{|x-x_{0}|}\biggr|^{2}\\
        &\le{}(f_{x_{0}}(r_{2})-f_{x_{0}}(r_{1}))
        +E(\overline{u},
        \Phi(B_{r_{2}}(0))\setminus{}\Phi(B_{r_{1}}(0))).
    \end{align*}
    Rearranging constants gives the desired conclusion.
\end{proof}

\subsection{Full Regularity}\label{subsec:fullreg}

In this subsection, we improve the regularity of Proposition~\ref{prop:part_reg} to the following \emph{full} regularity result:

\begin{proposition}\label{prop:full_reg}
    Let $u\in W^{1,2}(\Omega;\mathbb{S}^2)$ be a minimizer of \eqref{def:dir}.
    Then $H^0(\Sigma)<+\infty$ and $\Sigma$ is discrete.
\end{proposition}

We follow the procedure in \cite[Section 10.3.2]{GiMe}. 
Most of the results therein carry over to our setting without modification.
The only change is the addition of the constraint $u\cdot\nu=0$ to the minimization in \cite[Theorem 10.25]{GiMe} and the inclusion of the extended energy \eqref{def:extenergy} in addition to the standard Dirichlet energy.

For this reason, and in view of the blow up procedure that this lemma will be applied to, we introduce the energy for $\overline{u}\in W^{1,2}(B_1(0);\mathbb{S}^2)$
    \, ,
\begin{align*}
    \overline{E}_{x_0,r}(\overline{u})
    \ &\coloneqq \
    \frac12\int_{\Phi_r(P^-)} |\nabla \overline{u}|^2 \dx x \\
    &\qquad
    + \frac{1}{2}\int_{\Phi_r(P^+)} \frac{1}{|rx-x_0|^{2}}
    	\biggl[
    	|\nabla{}\overline{u}|^{2}
    	+\frac{4 r^2}{|rx-x_0|^{4}}(\overline{u}\cdot{}(rx-x_0))^{2}\\
    &\qquad\qquad\qquad\qquad\qquad
        +\frac{4 r}{|rx-x_0|^{2}}\overline{u}^{T}\nabla{}\overline{u}^{T}(rx-x_0)
        -\frac{4 r}{|rx-x_0|^{2}}(\overline{u}\cdot{}(rx-x_0))\text{div}(\overline{u})\biggr]
    \, ,
\end{align*}
where $\Phi_r$ parametrizes $B_1(0)\cap[\frac1r(\Omega-x_0)] = \Phi_r(P^-)$ and $B_1(0)\setminus\overline{[\frac1r(\Omega-x_0)]} = \Phi_r(P^+)$ for $P^\pm\subset 
 \{\pm x\cdot x_0>0\}$.

Such $\Phi_r$ can be constructed as in Subsection~\ref{subsubsec:rigorous_calc} with the modification that $r_0=1$ and $\phi(y_1,y_2) = (y_1,y_2,\sqrt{\frac{1}{r^2} - y_1^2 - y_2^2})-\frac{1}{r}e_3$.

To improve readability, we do not write the dependence of $\Phi$ on $r$ and $x_0$.
Note that the energy $\overline{E}_{x_0,r}(u)$ corresponds to the combined energies $\frac1r (E(u)+\widetilde{E}(\widetilde{u}))$ on the ball $B_r(x_0)$ with a rescaling $\overline{u}(z)=\widetilde{u}(rz+x_0)$ to $B_1(0)$.

\begin{lemma}\label{lem:full_reg_cptness}
    Let $\{r_{k}\}_{k\in\mathbb{N}}$ be a sequence of positive real numbers
    such that $r_k\searrow 0$ as $k\to\infty$.
    Furthermore, let $\overline{u}_k\in W^{1,2}(B_1(0);\mathbb{S}^2)$ be a sequence of energy minimizing \emph{tangential} harmonic maps with equibounded energies, i.e.\ 
    \begin{align*}
        \overline{u}_k
        \in
        \argmin_{u} \overline{E}_{x_0,r_k}(u)
        \, ,
    \end{align*}
    where the minimum is taken over maps $u\in W^{1,2}(B_1(0);\mathbb{S}^2)$ with $u\cdot\nu=0$ on $\Phi_{r_k}(\{x\cdot x_0=0\})$ and we assume there exists $C>0$ such that
    \begin{align*}
        \sup_{k\in\mathbb{N}} \int_{B_1(0)} |\nabla \overline{u}_k|^2 \dx x
        \ \leq \
        C
        \, .
        \end{align*}
    Then, a subsequence $\overline{u}_{k_\ell}$ converges in $W^{1,2}(B_1(0);\mathbb{S}^2)$ to an energy minimizing harmonic map $\overline{u}$ with $\overline{u}\cdot\nu=0$ on $\{x\cdot x_0=0\}\cap B_1(0)$.
\end{lemma}

\begin{figure}
\begin{center}
\scalebox{0.7}{\begin{tikzpicture}[scale=1]
\pgfmathsetmacro{\r}{2}  
\pgfmathsetmacro{\R}{10} 

\draw[black,line width=1] (0,0) circle (\r cm);
\fill[black] (0,0) circle (0.05 cm) node[below] {$0$};
\draw[black, line width=0.5] (0,0) -- (45:\r);
\node at (55:0.5*\r) {$1$};

\draw[black,line width=1] (1.6*\r,-0.3*\r) arc (180-110:180-70:\R);

\draw[black,line width=1] (\R,0) circle (\r cm);
\fill[black] (\R,0) circle (0.05 cm) node[below] {$0$};

\draw[black, line width=0.5] (\R,0) -- ++(45:\r);
\node at (\R+0.3*\r,0.5*\r) {$1$};

\draw[black,line width=1] (\R+1.5*\r,0) -- (\R-1.5*\r,0);
\node[] at (1.3*\R,0.3) {$\{x\cdot x_0=0\}$};

\node[] at (0.95*\R, 0.9) {$P^+$};
\node[] at (0.95*\R,-0.9) {$P^-$};

\draw[black,line width=2, ->] (\R-1.8*\r,0) -- (1.8*\r,0);
\node[black] at (0.5*\R,0.3) {$\Phi_{r_k}$};

\end{tikzpicture}}
\caption{The diffeomorphism $\Phi_{r_k}$ maps the flat boundary $\{x\cdot x_0=0\}$ to a curved surface (that is part of a sphere of radius $\frac{1}{r_k}$) on which the tangentiality condition is satisfied.}
\label{fig:straighten_bdry}
\end{center}
\end{figure}
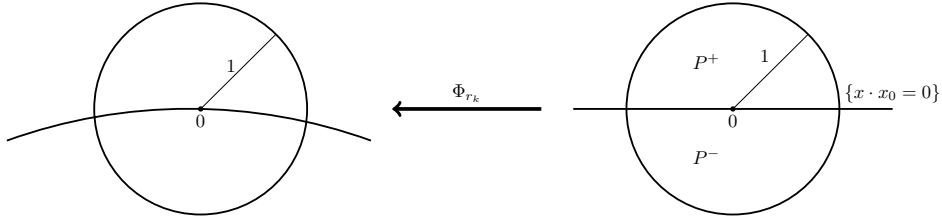

\begin{proof}
    Mutatis mutandis, the proof of Theorem 10.25 in \cite{GiMe} and Lemma 1 in \cite{Luck} can be applied.
    More precisely the following changes are necessary:
    First, observe that the tangentiality condition $\overline{u}_{k}\cdot\nu=0$ on $\Phi_{r_k}(\{x\cdot x_0=0\})) =: P_k^0$ is preserved by weak $W^{1,2}-$convergence.
    Second, in order to construct a competitor $v_k$ for the $k-$problem from a competitor of the limit energy $v$ (which is tangential to $P^0 \coloneqq \{x\cdot x_0=0\}$), one needs to take into account the changing set $P_k^0$ and hence alter the function to stay tangential w.r.t.\ $P_k^0$. 
    One way to accomplish this, is to look at the function
    \begin{equation*}
        \widetilde{v_k}(x)
        \ = \
        \frac{\nabla\Phi_{k}(\Phi_{k}^{-1}(x)) v(\Phi_{k}^{-1}(x))}
        {|\nabla\Phi_{k}(\Phi_{k}^{-1}(x)) v(\Phi_{k}^{-1}(x))|},
    \end{equation*}
    rather than $v$.
    This modification ensures tangentiality of $\widetilde{v_k}$ to $P_{k}^0 = \Phi_{r_k}(P^0)$ since
    \begin{align*}
    \widetilde{v_k}(y)\cdot\nu_{P_k^0}(y)
        \ &= \ 
        \widetilde{v_k}(y)\cdot\nu_{\Phi_{r_k}(P^0)}(y) 
        \ = \
        0 \, , 
    \end{align*}
    by Lemma A.1 in \cite{KiMaSt} ( see also the calculation in \eqref{lem:tang-var-admissible:eq} in the proof of lemma~\ref{lem:tang-var-admissible}), and does not change the energy of $v$ up to a negligible error as $k\to\infty$.
    \ds{
    The construction to connect this tangential map $\widetilde{v_k}$ with the map $u_k$ is carried out in Lemma 1 of \cite{Luck} and we adopt the notation from there.
    One can choose the decomposition into cells of the sphere in a way that respects $P^0_k$, namely that $P_k^0$ intersected with the sphere is given by $\bigcup_{j=0}^1 \bigcup_{i=0}^{m_j} e_i^j$, for some $m_j\in\mathbb{N}$, $m_j\leq k_j$.
    Up to a bilipschitz isomorphism, we can additionally assume that $P_k^0$ in the spherical shell can be written as $P_k^0 = \bigcup_{j=0}^1 \bigcup_{i=0}^{m_j} \hat{e}_i^j$.
    In the construction of the extension $\varphi$, one needs to add a projection onto $TP_k^0$ for $j=0,1$ to preserve the tangentiality of $\varphi$ on $P_k^0$. 
    This modification is small since $u_j$ and $v_j$ are both tangential to $P_k^0$.
    For $j=2$, we use the $0-$homogeneous extension without modification.
    At the very end, to obtain $v_k$, we project $\varphi$ at each point onto $\mathbb{S}^2$ which introduces another small error (see e.g.\ \cite[Cor. 10.24]{GiMe}).
    } 
\end{proof}

Next, we include a lemma which highlights a weak upper semi-continuity property satisfied by the density $\Theta$.
This is analogous to Proposition 10.26 of \cite{GiMe}.
\ds{
Note that we cannot obtain the exact upper semi-continuity $\limsup_{j\to\infty}\Theta(u_{j},x_{j})\leq\Theta(u,x)$ for boundary points $x\in\partial\Omega$ since the available monotonicity formulas are only applicable either in the inside of $\Omega$ with radii $r$ at most the distance to the boundary (classical), or directly on the boundary (Proposition~\ref{prop:mod_monotonicity_formula}).
The lack of a monotonicity formula that is uniform in the distance to the boundary can be explained through the inadmissibility of a radial deformation in view of our tangential boundary conditions, i.e.\ there is no diffeomorphism $\varphi_t$ that looks like $x-x_0$ in a ball $B_r(x_0)$ for some $x_0$ close to $\partial\Omega$ and $r>\dist(x_0,\partial\Omega)$ and at the same time preserves the boundary.
This means that for a sequence $x_j$ inside $\Omega$ that approaches $x\in\partial\Omega$, we cannot compare energies and use convergence of the Dirichlet integral for a small but uniformly chosen radius as e.g.\ in the proof of \cite[Proposition 10.26]{GiMe}. 

Once the regularity of our minimizer is proven, one can a posteriori show by testing the Euler-Lagrange equations from Lemma~\ref{lem:global_PDE} with $\varphi=\nabla{}u\zeta$ where $\zeta\in{}C_{c}^{\infty}(B_{2}(0);\mathbb{R}^{3})$, that a monotonicity formula also holds close to the boundary and thus $\Theta$ satisfies the upper semi-continuity with the precise constant $1$ instead of $2$.
} 
\begin{lemma}\label{lem:Theta_upper_semi_continuity}
    Suppose that
    $\{u_{j}\}_{j\in\mathbb{N}}\subseteq{}W_{T}^{1,2}(\Omega;\mathbb{S}^{2})$
    is a sequence of locally energy minimizing maps to \eqref{def:dir}
    with locally equibounded energies and $u_{j}\rightharpoonup{}u$ in
    $L^{2}$.
    Suppose also that we define extensions of each member of this sequence,
    $\overline{u}_{j}\in\mathcal{W}_{u_{j}}$ as in \eqref{eq:extension}.
    Suppose also that we have a sequence
    $\{x_{j}\}_{j\in\mathbb{N}}\subseteq{}\overline{\Omega}$
    converging to $x\in{}\overline{\Omega}$.
    Then
    \begin{equation*}
        \limsup_{j\to\infty}\Theta(\overline{u}_{j},x_{j})
        \ \leq \
        2 \ \Theta(\overline{u},x)
        \, .
    \end{equation*}
\end{lemma}

\begin{proof}
    If $x\in\Omega$ then, by perhaps passing to a subsequence, we may assume that
    $\{x_{j}\}_{j\in\mathbb{N}}\subseteq\Omega$
    and apply the proof of Proposition $10.26$ of \cite{GiMe}.
    As a result, we may assume that $x\in\partial\Omega$.
    \ds{Appealing to Lemma~\ref{lem:full_reg_cptness} (with $\Phi=I$)}
    we have that $\overline{u}_{j}\to{}\overline{u}$ strongly in
    $W^{1,2}(\Omega;\mathbb{S}^{2})$.
    For each $j\in\mathbb{N}$ we let $z_{j}\coloneqq\frac{x_{j}}{|x_{j}|}$.
    Fix $0<\rho,\varepsilon<\frac{r_{0}}{4(1+c_{0})}$, where $c_{0}$ is the constant
    from Corollary \ref{corr:densbehav}, and consider $j$ sufficiently
    large that $|x-x_{j}|<\varepsilon$.
    Observe that
    \begin{equation*}
        B_{s}(x_{j})\subseteq{}B_{s+\varepsilon}(x),\hspace{15pt}
        B_{s}(x_{j})\subseteq{}B_{s+\varepsilon}(z_{j}),\hspace{15pt}
        B_{s}(z_{j})\subseteq{}B_{s+2\varepsilon}(x)
    \end{equation*}
    for all $s>0$.
    In addition, for $0<s\le\text{dist}(x_{j},\partial\Omega)$ we have
    \begin{equation*}
        B_{s}(x_{j})\subseteq{}B_{2\text{dist}(x_{j},\partial\Omega)}(z_{j}).
    \end{equation*}
    We first consider the case that $x_{j}\in\partial\Omega$ where we have
    $x_{j}=z_{j}$.
    Appealing to Corollary \ref{corr:densbehav} we have that there are
    smooth functions $h_{1},h_{2}\colon[0,\infty)\to[0,\infty)$,
    independent of the point $x_{0}\in\partial\Omega$, satisfying $h_{1}(0)=1=h_{2}(0)$ as well as
    \begin{equation*}
        f_{z_{j},r_{0}}(u_{j},\rho)\le{}
        h_{1}(\rho)\biggl[e^{C_{1}h_{2}(\rho)}\Theta(\overline{u}_{j},z_{j},h_{2}(\rho)\rho)
        +C_{2}h_{2}(\rho)\rho\biggr].
    \end{equation*}
    From this and \eqref{eq:dirdenslimit} we see that if $\text{}$
    \begin{align*}
        \Theta(\overline{u}_{j},x_{j})
        &\le{}f_{x_{j},r_{0}}(\overline{u}_{j},\rho)\\
        &\le{}h_{1}(\rho)\biggl[e^{C_{1}h_{2}(\rho)\rho}\Theta(\overline{u}_{j},x_{j},h_{2}(\rho)\rho)
        +C_{2}h_{2}(\rho)\rho\biggr]\\
        &\le{}h_{1}(\rho)\biggl[e^{C_{1}h_{2}(\rho)\rho}
        \cdot\frac{1}{2h_{2}(\rho)\rho}\int_{B_{h_{2}(\rho)\rho+\varepsilon}(x)}\!{}
        |\nabla\overline{u}_{j}|^{2}
        +C_{2}h_{2}(\rho)\rho\biggr].
    \end{align*}
    Next we consider the case $x_{j}\in\Omega$.
    Appealing to the classical Monotonicity Formula we have that
    $s\mapsto\Theta(\overline{u}_{j},x_{j},s)$ is non-decreasing for
    $0<s\le\text{dist}(x_{j},\partial\Omega)$.
    Observe that if $0<\rho<\text{dist}(x_{j},\partial\Omega)$ then
    \begin{equation*}
        \Theta(\overline{u}_{j},x_{j})
        \le\Theta(\overline{u}_{j},x_{j},\rho)
        \le\frac{1}{2\rho}\int_{B_{\rho+2\varepsilon}(x)}\!{}|\nabla\overline{u}_{j}|^{2}.
    \end{equation*}
    If $\text{dist}(x_{j},\partial\Omega)\le\rho$ then we have, by following the
    proof of Corollary \ref{corr:densbehav}, that
    \begin{align*}
        \Theta(\overline{u}_{j},x_{j})
        &\le\Theta(\overline{u}_{j},x_{j},\text{dist}(x_{j},\partial\Omega))\\
        &\le\frac{1}{2\text{dist}(x_{j},\partial\Omega)}
        \int_{B_{2\text{dist}(x_{j},\partial\Omega)}(z_{j})}\!{}
        |\nabla\overline{u}_{j}|^{2}\\
        &=2\Theta(\overline{u}_{j},z_{j},2\text{dist}(x_{j},\partial\Omega))\\
        &\le2e^{2C_{1}\text{dist}(x_{j},\partial\Omega)}
        \Theta(\overline{u}_{j},z_{j},2\text{dist}(x_{j},\partial\Omega))
        +2C_{2}\text{dist}(x_{j},\partial\Omega)\\
        &\le2h_{3}(2\text{dist}(x_{j},\partial\Omega))
        f_{z_{j},r_{0}}(\overline{u}_{j},(1+2c_{0}\text{dist}(x_{j},\partial\Omega))\text{dist}(x_{j},\partial\Omega))\\
        &\le2h_{3}(2\text{dist}(x_{j},\partial\Omega))
        f_{z_{j},r_{0}}(\overline{u}_{j},(1+c_{0})\rho)\\
        &\le2h_{3}(2\text{dist}(x_{j},\partial\Omega))h_{1}((1+c_{0})\rho)
        \biggl[\\
        &\cdot\frac{e^{C_{1}h_{2}((1+c_{0})\rho)(1+c_{0})\rho}}{2h_{2}((1+c_{0})\rho)(1+c_{0})\rho}
        \int_{B_{h_{2}((1+c_{0})\rho)(1+c_{0})\rho+2\varepsilon}(x)}\!{}
        |\nabla\overline{u}_{j}|^{2}
        +C_{2}h_{2}((1+c_{0})\rho)(1+c_{0})\rho\biggr]
    \end{align*}
    where $h_{3}\colon[0,\infty)\to[0,\infty)$ is a smooth function appearing in the proof of Corollary
    \ref{corr:densbehav} satisfying $h_{3}(0)=1$.
    We now see that, appealing to strong $W^{1,2}$ convergence we have, for all $j$
    sufficiently large and for some fixed $s>0$
    \begin{equation*}
        \frac{1}{2s}\int_{B_{s+2\varepsilon}(x)}\!{}
        |\nabla\overline{u}_{j}|^{2}
        \le\frac{1}{2s}\int_{B_{s+2\varepsilon}(x)}\!{}
        |\nabla\overline{u}|^{2}+\varepsilon.
    \end{equation*}
    Using this in either case gives, after taking the limit supremum in $j$ followed
    by the limit $\varepsilon\to0^{+}$, that
    \begin{align*}
        &\limsup_{j\to\infty}\Theta(\overline{u}_{j},x_{j})
        \le{}\max\biggl\{
        \frac{1}{2\rho}\int_{B_{\rho}(x)}\!{}|\nabla\overline{u}|^{2},\\
        &2h_{1}((1+c_{0})\rho)\biggl[
        \cdot\frac{e^{C_{1}h_{2}((1+c_{0})\rho)(1+c_{0})\rho}}{2h_{2}((1+c_{0})\rho)(1+c_{0})\rho}
        \int_{B_{h_{2}((1+c_{0})\rho)(1+c_{0})\rho}(x)}\!{}
        |\nabla\overline{u}|^{2}
        +C_{2}h_{2}((1+c_{0})\rho)(1+c_{0})\rho\biggr]\biggr\}.
    \end{align*}
    Finally, letting $\rho\to0^{+}$ we obtain
    \begin{equation*}
        \limsup_{j\to\infty}\Theta(\overline{u}_{j},x_{j})
        \le2\Theta(\overline{u},x).
    \end{equation*}
\end{proof}

We now have all necessary ingredients to prove the full regularity result and hence also Theorem~\ref{thm:main_regularity}:

\begin{proof}[Proof of Proposition \ref{prop:full_reg}]
    In the interior of $B_1(0)$, this result is well known \cite{GiGi, GiMe, SU}.
    So we only need to prove regularity at boundary points $x_0\in\partial B_1(0)$.
   So assume for the sake of contradiction that there is a sequence $x_j\in \overline{B_1(0)}\cap\Sigma(u)$ of singular points that converge to $x_0$.
   Performing a blow up with the radius $\rho_j = 2|x_j|$ (see Figure \ref{fig:dim_reduction}), we obtain a sequence $u_j(z) = u(x_0+\rho_j z)$ which is equibounded in energy and tangential, hence satisfying the assumptions of Lemma~\ref{lem:full_reg_cptness}. 
   Therefore, we can assume, up to passing to a subsequence, that $u_j\to u$ in $W^{1,2}(\Omega;\mathbb{S}^2)$ with $u$ energy minimizing and tangential.
   Note that since the boundary of our initial domain is smooth, the blow up converges to a hyperplane $\{x\cdot x_0=0\}$.
   By compactness of $\partial B_{\frac12}(0)$, we can furthermore assume that the singular points $z_j = \frac{1}{2}\frac{x_j-x_0}{|x_j-x_0|}$ in the rescaled picture converge to a limiting point $z_0$ with $|z_0|=\frac12$.
   Then the weak upper semi-continuity of the density $\Theta$ (Lemma~\ref{lem:Theta_upper_semi_continuity})    with the definition of a singular point, \eqref{def:singset}, yields that $z_0$ is also singular. 
   In the limit $j\to\infty$ (i.e.\ $\rho_j\to 0$) $f_{x_0, r_0}(u)$ is constant, as in the classical case, and all error terms in the monotonicity formula vanish,  and thus we can conclude from Corollary~\ref{cor:radialmono}, as in the classical case, that the limiting function $u$ must be homogeneous of degree $0$ in a neighborhood of the origin.
   This implies that a whole segment $\mathfrak{S}\subset\{\lambda z_0\sd \lambda>0\}\cap B_1(0)$ is singular.
   This contradicts Proposition~\ref{prop:part_reg}.
\end{proof}

\begin{figure}
\begin{center}
\begin{tikzpicture}[scale=1]
\pgfmathsetmacro{\r}{3} 	

\draw[black,line width=1] (0,0) circle (\r cm);
\fill[black] (0,0) circle (0.05 cm) node[left] {$x_0$};

\draw[blue] (0,0) -- (50*0.9:0.9*0.9*\r);
\node[blue] at (65*0.9:{0.54*0.9*0.9*\r}) {$\frac12\rho_j$};

\foreach \t in {0.9,0.8,...,0.1}{ 
	\draw[blue,line width=0.5] (0,0) circle (\t*\t*\r cm);
	\fill[purple] ({50*\t}:{\t*\t*\r}) circle (0.05 cm);
	} 
\node[purple] at ({50*0.9}:{(0.9*0.9+0.1)*\r}) {$x_j$};

\end{tikzpicture}  
\caption{The radii $\rho_j$ for the blow up are chosen such that $\partial B_{\frac12\rho_j}(x_0)$ contains a singular point $x_j$.}
\label{fig:dim_reduction}
\end{center}
\end{figure}
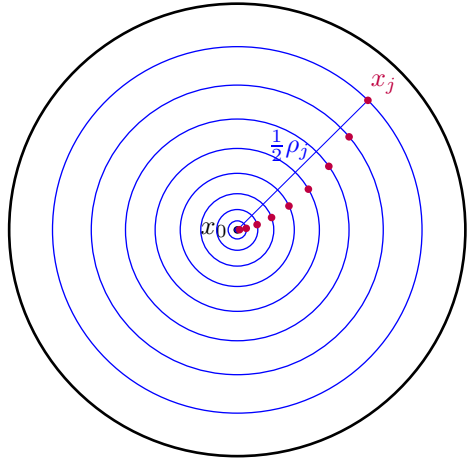

\section{Symmetrization and Equivariance}
\label{sec:sym_equiv}

In this section we're concerned with proving partial results on the symmetry of minimizers.
The main difficulty lies in the fact that standard techniques such as symmetric rearrangements are not applicable, since they do not preserve the norm constraint or the tangency condition on the boundary.

The main result in this section is Proposition~\ref{prop:symmetrization} in which we prove the equivariance of minimizers under one additional assumption. 

\begin{remark}\label{rem:e_theta=0_is_difficult_too}
    A first idea to prove that $u_\theta=0$ might consist in comparing the minimizer $u = u_r\ee_r + u_\theta \ee_\theta + u_\varphi\ee_\varphi$ with the competitor
    \begin{align*}
        \widetilde{u}
        \ \coloneqq \
        u_r\:\ee_r \ \pm \ \sqrt{u_\varphi^2+u_\theta^2}\:\ee_\varphi
        \,,
    \end{align*}
    where we have used spherical coordinates to express the functions.
    Note that $\widetilde{u}$ is $\mathbb{S}^2-$valued and tangential, provided that $u$ is tangential.
    However, it is not obvious that $\widetilde{u}$ has a lower energy, even if one assumes $u$ to be equivariant since certain mixed terms in the energy do not carry a sign.
    It would be interesting to prove that $u_\theta$ must be zero using an explicit competitor.
\end{remark}

We introduce some notation we will
use for the rest of the section.
We note that the unit vectors defined in Subsection~\ref{subsec:equivariance} also satisfy
\begin{align*}
    \partial_{\rho}\mathbf{e}_{\rho}&=0,\hspace{15pt}
    \partial_{\theta}\mathbf{e}_{\rho}=\mathbf{e}_{\theta},\hspace{10pt}
    \phantom{-0}\partial_{z}\mathbf{e}_{\rho}=0,\\
    \partial_{\rho}\mathbf{e}_{\theta}&=0,\hspace{15pt}
    \partial_{\theta}\mathbf{e}_{\theta}=-\mathbf{e}_{\rho},\phantom{0}\hspace{10pt}
    \partial_{z}\mathbf{e}_{\theta}=0,\\
    \partial_{\rho}\mathbf{e}_{z}&=0,\hspace{15pt}
    \partial_{\theta}\mathbf{e}_{z}=0,\phantom{-\mathbf{e}_{\rho}}\hspace{10pt}
    \partial_{z}\mathbf{e}_{z}=0.
\end{align*}
It follows that we have
\begin{align*}
    \partial_{\rho}u(\rho,\theta,z)&=
    \partial_{\rho}u_{\rho}\mathbf{e}_{\rho}+\partial_{\rho}u_{\theta}\mathbf{e}_{\theta}
    +\partial_{\rho}u_{z}\mathbf{e}_{z},\\
    \partial_{\theta}u(\rho,\theta,z)&=
    \partial_{\theta}u_{\rho}\mathbf{e}_{\rho}+\partial_{\theta}u_{\theta}\mathbf{e}_{\theta}
    +\partial_{\theta}u_{z}\mathbf{e}_{z}
    +u_{\rho}\mathbf{e}_{\theta}
    -u_{\theta}\mathbf{e}_{\rho}\\
    \partial_{z}u(\rho,\theta,z)&=
    \partial_{z}u_{\rho}\mathbf{e}_{\rho}
    +\partial_{z}u_{\theta}\mathbf{e}_{\theta}
    +\partial_{z}u_{z}\mathbf{e}_{z}.
\end{align*}
We also see that
\begin{align*}
    \partial_{x_{1}}u&=\cos(\theta)\partial_{\rho}u(\rho,\theta,z)
    -\frac{\sin(\theta)}{\rho}\partial_{\theta}u(\rho,\theta,z),\\
    \partial_{x_{2}}u&=\sin(\theta)\partial_{\rho}u(\rho,\theta,z)
    +\frac{\cos(\theta)}{\rho}\partial_{\theta}u(\rho,\theta,z),\\
    \partial_{x_{3}}u&=\partial_{z}u(\rho,\theta,z),
\end{align*}
and hence
\begin{align*}
    |\nabla{}u|^{2}&=
    |\partial_{\rho}u(\rho,\theta,z)|^{2}+\frac{1}{\rho^{2}}
    |\partial_{\theta}u(\rho,\theta,z)|^{2}
    +|\partial_{z}u(\rho,\theta,z)|^{2}\\
    &=|\partial_{\rho}u_{\rho}|^{2}+|\partial_{\rho}u_{\theta}|^{2}
    +|\partial_{\rho}u_{z}|^{2}
    +\frac{1}{\rho^{2}}
    \bigl[|\partial_{\theta}u_{\rho}|^{2}+|\partial_{\theta}u_{\theta}|^{2}
    +|\partial_{\theta}u_{z}|^{2}\bigr]\\
    &+\bigl[|\partial_{z}u_{\rho}|^{2}
    +|\partial_{z}u_{\theta}|^{2}+|\partial_{z}u_{z}|^{2}\bigr]
    +\frac{2}{\rho^{2}}\bigl[u_{\rho}\partial_{\theta}u_{\theta}
    -u_{\theta}\partial_{\theta}u_{\rho}\bigr].
\end{align*}

The following proposition adapts the proof of Theorem 1 from \cite{SaSha} to our setting.

\begin{proposition}\label{prop:u_theta_0_S2_implies_u_theta_0_equiv}
    Let $u\in W^{1,2}(\Omega;\mathbb{S}^2)$ satisfy \eqref{eq:tanbc}.
    Assume furthermore that $u_\theta=0$ on $\mathbb{S}^2$.
    Then there exists a function $\widetilde{u}\in W^{1,2}(\Omega;\mathbb{S}^2)$ satisfying the boundary condition \eqref{eq:tanbc} and
    \begin{enumerate}
        \item $\widetilde{u}\cdot \ee_\theta = 0$ in $\Omega$
        \item $\widetilde{u}$ is equivariant.
    \end{enumerate}
    with $E(\widetilde{u})\leq E(u)$ \ds{with equality if and only if $u$ is  equivariant and orthogonal to $\ee_\theta$ in $\Omega$.}
\end{proposition}

\begin{proof}
    We decompose $u$ into components in cylindrical coordinates $(\rho,\theta,z)$, i.e. $u = u_\rho\ee_\rho + u_\theta \ee_\theta + u_z\ee_z$.
    Inspired by \cite{SaSha}, we pose the competitor
    \begin{align*}
        \widetilde{u}(\rho,z)
        \ &\coloneqq \
        \widetilde{u}_\rho(\rho,z)\ee_\rho + \widetilde{u}_z(\rho,z)\ee_z
        \, ,
    \end{align*}
    where
    \begin{align*}
        \widetilde{u}_\rho(\rho,z)
        \ \coloneqq \
        \frac{1}{2\pi \rho} \int_{\rho\mathbb{S}^1 + z\ee_z} u_\rho \dx s
        \qquad\text{ and }\qquad
        \widetilde{u}_z(\rho,z)
        \ \coloneqq \ \sqrt{1 - (\widetilde{u_\rho}(\rho,z))^2}
        \, ,
    \end{align*}
    where the circle $\mathbb{S}^1$ lies in the plane $\{z=0\}$.
    While this competitor in general will not satisfy the boundary condition $\widetilde{u}\cdot\nu=0$, we can use the assumption that $u_\theta=0$ on the boundary and the $W^{1,2}-$regularity to deduce that $u=\pm\ee_\varphi$ on the boundary and thus $u_\rho$ (and also $u_z$) are constant along circles $\rho\mathbb{S}^1 + z\ee_z$ with $\rho^2+z^2=1$.
    Hence, the averaging does not change the boundary condition and we have $\widetilde{u}=u$ on $\mathbb{S}^2$.

    Next, we compute $|\partial_\alpha \widetilde{u}|^2$ and $|\partial_\alpha {u}|^2$ for $\alpha\in\{\theta,\rho,z\}$ to obtain the energy estimate.

    We find that
    \begin{align*}
        |\partial_\theta u|^2
        \ &= \ 
        |\partial_\theta v|^2 + |\partial_\theta u_z|^2 
        \qquad\qquad
        |\partial_\theta \widetilde{u}|^2
        \ = \ 
        |\partial_\theta \widetilde{v}|^2
        \, ,
    \end{align*}
    where $v = u_\rho\ee_\rho + u_\theta \ee_\theta$ and $\widetilde{v}=\widetilde{u}_\rho\ee_\rho$.
    A straightforward computation then shows that
    \begin{align*}
        \partial_\theta\widetilde{v} \cdot \partial_\theta(v-\widetilde{v})
        \ &= \
        \widetilde{u}_\rho ((u_\rho - \widetilde{u}_\rho) + \partial_\theta u_\theta)
        \, .
    \end{align*}
    Integrating this equality over $\rho\mathbb{S}^1 + z\ee_z$ gives zero since $\widetilde{u}_\rho$ is the average of $u_\rho$ and $u_\theta$ is periodic.
    By Fubini, we get that
    \begin{align*}
        \int_{B_1(0)} \frac{1}{\rho^2}|\partial_\theta v|^2 \dx x
        \ &= \
        \int_{B_1(0)} \frac{1}{\rho^2}|\partial_\theta (\widetilde{v} - v)|^2 \dx x
        +
        \int_{B_1(0)} \frac{1}{\rho^2}|\partial_\theta \widetilde{v}|^2 \dx x
        \, ,
    \end{align*}
    which implies
    \begin{align}\label{prop:u_theta_0_S2_implies_u_theta_0_equiv:theta}
        \int_{B_1(0)} \frac{1}{\rho^2}|\partial_\theta \widetilde{u}|^2 \dx x
        \ &\leq \
        \int_{B_1(0)} \frac{1}{\rho^2}|\partial_\theta u|^2 \dx x
        \, ,
    \end{align}
    with equality if and only if $v-\widetilde{v}$ is a constant on each circle and $\partial_\theta u_z=0$.

    Now we turn to the remaining derivatives. 
    Let $\alpha\in\{\rho,z\}$.
    Then, by Cauchy-Schwartz inequality it holds 
    \begin{align*}
        \bigl|\partial_{\alpha}(\sqrt{u_{\theta}^{2}+u_{z}^{2}})\bigr|
        &=\biggl|(\partial_{\alpha}u_{\theta},\partial_{\alpha}u_{z})
        \cdot\frac{(u_{\theta},u_{z})}{\sqrt{u_{\theta}^{2}+u_{z}^{2}}}\biggr|\\
        &\le\sqrt{\partial_{\alpha}u_{\theta}^{2}+\partial_{\alpha}u_{z}^{2}}
        \sqrt{\biggl(\frac{u_{\theta}}{\sqrt{u_{\theta}^{2}+u_{z}^{2}}}\biggr)^{2}
        +\biggl(\frac{u_{z}}{\sqrt{u_{\theta}^{2}+u_{z}^{2}}}\biggr)^{2}}\\
        &=\sqrt{\partial_{\alpha}u_{\theta}^{2}+\partial_{\alpha}u_{z}^{2}}
    \end{align*}
    and hence
    \begin{align*}
        |\partial_\alpha u|^2
        \ &= \
        |\partial_\alpha u_\rho|^2
        + |\partial_\alpha u_\theta|^2
        + |\partial_\alpha u_z|^2
        \ \geq \
        |\partial_\alpha u_\rho|^2
        + |\partial_\alpha \sqrt{u_\theta^2+u_z^2}|^2 \\
        \ &= \
        |\partial_\alpha u_\rho|^2
        + |\partial_\alpha \sqrt{1-u_\rho^2}|^2
        \ = \
        |\partial_\alpha u_\rho|^2
        + \bigg|\frac{\partial_\alpha u_\rho}{\sqrt{1-u_\rho^2}}\bigg|^2
        \, ,
    \end{align*}
    which implies
    \begin{align*}
        \int_{\rho\mathbb{S}^1+z\ee_z}
        |\partial_\alpha u|^2 
        \dx s
        \ &\geq \
        \int_{\rho\mathbb{S}^1+z\ee_z}
        \bigg(|\partial_\alpha u_\rho|^2
        + 
        \bigg|\frac{\partial_\alpha u_\rho}{\sqrt{1-u_\rho^2}}\bigg|^2
        \bigg)\dx s
        \, ,
    \end{align*}
    At the same time
    \begin{align*}
        \int_{\rho\mathbb{S}^1+z\ee_z}
        |\partial_\alpha \widetilde{u}|^2
        \dx s
        \ &= \
        \int_{\rho\mathbb{S}^1+z\ee_z}
        \bigg(|\partial_\alpha \widetilde{u}_\rho|^2
        + \bigg|\frac{\partial_\alpha \widetilde{u}_\rho}{\sqrt{1-(\widetilde{u}_\rho)^2}}\bigg|^2
        \bigg)\dx s
        \, .
    \end{align*}
    Since the map $\beta\mapsto\beta^2$ is convex on $\mathbb{R}$ and the map $(\gamma,\zeta)\mapsto \frac{\gamma^2}{1-\zeta^2}$ is convex on $\mathbb{R}\times [-1,1]$, Jensen's inequality yields
    \begin{align} \label{prop:u_theta_0_S2_implies_u_theta_0_equiv:alpha}
        \int_{B_1(0)} |\partial_\alpha \widetilde{u}|^2 \dx x
        \ &\leq \
        \int_{B_1(0)} |\partial_\alpha u|^2 \dx x
        \, .
    \end{align}
    \ds{If equality holds in \eqref{prop:u_theta_0_S2_implies_u_theta_0_equiv:alpha}, then $\partial_\alpha u_\rho$ and $u_\rho$ are constant on each circle $\rho\mathbb{S}^1+z\ee_z$.
    This implies that $u_\rho=\widetilde{u}_\rho$.
    From the equality case for the $\theta-$derivative we recall that $v-\widetilde{v}$ must be constant on each circle.
    Using that $u_\rho=\widetilde{u}_\rho$ this implies that $u_\theta=0$. 
    From the norm constraint we then deduce that $|u_z|=|\widetilde{u}_z|$, so $u$ is equivariant and $u\cdot\ee_\theta=0$.
    On the other hand one can check that if $u_\theta=0$ and $u_\rho$ is independent of $\theta$ (i.e.\ constant on the circles $\rho\mathbb{S}^1+z\ee_z$), then it holds equality in \eqref{prop:u_theta_0_S2_implies_u_theta_0_equiv:alpha}.
    } 

    Combining \eqref{prop:u_theta_0_S2_implies_u_theta_0_equiv:theta} and \eqref{prop:u_theta_0_S2_implies_u_theta_0_equiv:alpha}
    gives
    \begin{align*}
        \int_{B_1(0)} |\nabla \widetilde{u}|^2 \dx x
        \ &= \
        \int_{B_1(0)} 
        \Big(|\partial_\rho \widetilde{u}|^2 
        + \frac{1}{\rho^2}|\partial_\theta \widetilde{u}|^2 
        + |\partial_z \widetilde{u}|^2 
        \Big)\dx x \\
        \ &\leq \
        \int_{B_1(0)} 
        \Big(|\partial_\rho {u}|^2 
        + \frac{1}{\rho^2}|\partial_\theta {u}|^2 
        + |\partial_z {u}|^2 
        \Big)\dx x
        \ = \
        \int_{B_1(0)} |\nabla u|^2 \dx x
        \, .
    \end{align*}
\end{proof}

The main result of this section is the following proposition which constructs an equivariant function from a given function $u$ with energy smaller or equal to that of $u$.
As we will explain in Remark~\ref{rem:symm_cover_diff_cases},
Proposition~\ref{prop:u_theta_0_S2_implies_u_theta_0_equiv} and Proposition~\ref{prop:symmetrization} cover different cases and are of independent interest.

\begin{proposition}[Symmetrization] \label{prop:symmetrization}
Let $u\in W^{1,2}(\Omega;\mathbb{S}^2)$ satisfy \eqref{eq:tanbc}.
Assume furthermore that
\begin{align}\label{prop:symmetrization:T_=_0}
\T 
\ = \
\int_{B_1(0)} \frac{1}{\rho^2}(u_\rho\partial_\theta u_\theta - u_\theta\partial_\theta u_\rho) \dx x
\ \geq \ 
0
\, .
\end{align}
Then there exists an equivariant function $u_\infty\in W^{1,2}(\Omega;\mathbb{S}^2)$ satisfying the boundary condition \eqref{eq:tanbc} with $E(u_\infty)\leq E(u)$ and equality if and only if $u$ is equivariant.
\end{proposition}

The proof relies on an iterative procedure of decreasing the energy by symmetrization of the coefficient functions of $u$.
The general idea is as follows:
Write the function $u$ and the energy $E$ in  cylindrical coordinates and choose a half-ball containing the $\ee_z-$axis of lower energy than the remaining half-ball.  
Reflect the coefficient function $u_a$ for $a\in \{r,\theta,\varphi\}$ across the plane separating the half-balls to define a new function $u_1$.
By construction, this new function satisfies the same boundary conditions as $u$ and preserves the norm.
From the energy one can see that this procedure decreases the energy only if the integral $\T$ in \eqref{prop:symmetrization:T_=_0} is non-negative.
Then choose one half of the lower energy half-ball and symmetrize again.
Repeating this symmetrization leads to a sequence of functions $u_k$ which will converge to $u_\infty$ with the desired properties.

\begin{proof}
We first note that the gradient of a function $u=u_{\ds{\rho}}\mathbf{e}_{\ds{\rho}}+u_{\theta}\mathbf{e}_{\theta}+u_{z}\mathbf{e}_{z}$ can be written as    
\begin{align*}
|\nabla u|^2
\ &= \
\left(
(\partial_\rho u_\rho)^{2}
+ (\partial_\rho u_\theta)^{2}
+ (\partial_\rho u_z)^{2}
\right) 
+\frac{1}{\rho^2}\left(
(\partial_{\theta}u_\rho)^{2}
+(\partial_{\theta}u_\theta)^{2}
+(\partial_{\theta}u_z)^{2}
\right) \\
&\qquad+ \left(
(\partial_{z}u_\rho)^{2}
+ (\partial_{z}u_\theta)^{2}
+ (\partial_{z}u_z)^{2}
\right) 
+\frac{2}{\rho^2}(u_{\rho}\partial_{\theta}u_{\theta}
-u_{\theta}\partial_{\theta}u_{\rho})
+\frac{1}{\rho^2}(u_{\rho}^{2}+u_{\theta}^{2})
\, .
\end{align*}
We introduce the modified energy $\widetilde{E}$ defined in cylindrical coordinates by
\begin{align*}
\tilde{E}(u)
\ &= \
\int_{\Omega_{\rho,\theta,z}} \: \bigg(
\frac{\rho}{2}
\left(
(\partial_\rho u_\rho)^{2}
+ (\partial_\rho u_\theta)^{2}
+ (\partial_\rho u_z)^{2}
\right) 
+\frac{1}{4\rho}\left(
(\partial_{\theta}u_\rho)^{2}
+(\partial_{\theta}u_\theta)^{2}
+(\partial_{\theta}u_z)^{2}
\right) \\
&\qquad\qquad+ \frac{\rho}{2}\left(
(\partial_{z}u_\rho)^{2}
+ (\partial_{z}u_\theta)^{2}
+ (\partial_{z}u_z)^{2}
\right) 
+\frac{1}{2\rho}(u_{\rho}^{2}+u_{\theta}^{2})
\bigg) \dx\rho \dx\theta \dx z.
\end{align*}
Compared to $E$, there are two modifications. Firstly, we omit the term $\T$ and secondly, we add an additional prefactor of $\frac12$ in front of all $\theta-$derivatives \ds{so that
\begin{align}\label{prop:symmetrization:tildeE_E_T_theta}
    \tilde{E}(u)
    \ &= \
    E(u) 
    - \T 
    - \frac{1}{2}\sum_{a\in\{\rho,\theta,z\}} \int_{\Omega_{\rho,\theta,z}} \: \frac{1}{2\rho} (\partial_\theta u_a)^{2}
\end{align}
} 
Note that for equivariant functions $v$ it holds $\tilde{E}(v)=E(v)$.

\vspace{0.5cm}
\textit{Step 1: Construction of $u_k$.}
We construct a sequence of functions which are increasingly symmetric but have non-increasing energy.
Specifically, we construct a sequence $\{u_{k}\}_{k=1}^{\infty}$ such that for each $k\in\mathbb{N}$ we have 
\begin{enumerate}
\item\label{eq:BasicConditions}
Admissibility:
$u_{k}\in{}W^{1,2}(\Omega,\mathbb{R}^{3})$ and 
$u_{k}\cdot\nu=0$ on $\partial\Omega$,
\item\label{eq:LInftyPreservation}
Norm preservation:
$|u_{k}(x)| = 1$ for $\mathcal{L}^{3}$-almost every $x\in \Omega$.
\item\label{eq:EnergyDecreasing}
Energy decreasing: $\tilde{E}(u_{k+1})\leq \tilde{E}(u_{k})$,
\item\label{eq:Even}
Reflection symmetry:
the components of $u_{k}$ are even relative to $\frac{2\pi}{2^{k}}$
for $\theta\in[0,\frac{2\pi}{2^{k-1}}]$,
\item\label{eq:Periodic}
Periodicity: 
the components of $u_{k}$ are $\frac{2\pi}{2^{k-1}}$-periodic
functions of $\theta$,
\end{enumerate}
To construct this sequence we proceed by induction.

For $0\le\theta_{1}<\theta_{2}\le2\pi$, we introduce the notation $\Omega_{[\theta_{1},\theta_{2}]}$ by
\begin{equation*}
    \Omega_{[\theta_{1},\theta_{2}]}\coloneqq
    \{(\rho,\theta,z)\in\Omega_{\rho,\theta,z}:
    \rho^2+z^2\leq 1,\,\theta\in[\theta_{1},\theta_{2}]\}.
\end{equation*}

We set $u_{0}=u$ to start the sequence and we define $u_1$.
Without loss of generality, we assume that $\tilde{E}(u_{0},\Omega_{[0,\pi]}) \leq \tilde{E}(u_{0},\Omega_{[\pi,2\pi]})$.
We then define $u_1$ for $\theta\in [0,\pi]$ by setting $u_1=u_0$.
We extend $u_{1}$ to the rest of $B_{1}(0)$, i.e.\ for $\theta\in [\pi,2\pi]$ as
\begin{equation*}
    u_{1}(\rho,\theta,z)
    \coloneqq{}u_{0,\rho}(\rho,2\pi-\theta,z)\mathbf{e}_{\rho}+
    u_{0,\theta}(\rho,2\pi-\theta,z)\mathbf{e}_{\theta}+
    u_{0,z}(\rho,2\pi-\theta,z)\mathbf{e}_{z}.
\end{equation*}
Notice that $u_{1}\in{}W^{1,2}(B_{1}(0),\mathbb{R}^{3})$, $u_{1}\cdot\nu=0$ on $\mathbb{S}^{2}$,
$|u_{1}| = 1$
and $u_{1}$ satisfies requirements \ref{eq:Even} and \ref{eq:Periodic} by construction.
In order to check \ref{eq:EnergyDecreasing}, we observe that since $\tilde{E}(u_{0},\Omega_{[0,\pi]}) \leq \tilde{E}(u_{0},\Omega_{[\pi,2\pi]})$ and the energy $\tilde{E}$ is invariant under the exchange of $\theta$ against $2\pi-\theta$ in the coefficients, it also holds that
\begin{equation*}
    \tilde{E}(u_{1})
    \ = \ 
    2\: \tilde{E}(u_{0},\Omega_{[0,\pi]})
    \ \leq \ 
    \tilde{E}(u_{0}).
\end{equation*}
It is important to note that only $\tilde{E}$ is invariant under replacing $\theta$ by $2\pi-\theta$, but not $E$. 
The function $u_1$ therefore satisfies all requirements \ref{eq:BasicConditions}-\ref{eq:Periodic}.

Now we assume that such a function has been constructed for some $k$.   Assuming that      
\begin{equation*}
    \tilde{E}(u_{k},\Omega_{[0,\frac{2\pi}{2^{k+1}}]})
    \ \leq \
    \tilde{E}(u_{k},\Omega_{[\frac{2\pi}{2^{k+1}},\frac{2\pi}{2^{k}}]})
    \, ,
\end{equation*}            
a similar argument will work in the case where the other sector $\Omega_{[\frac{2\pi}{2^{k+1}},\frac{2\pi}{2^{k}}]}$ has lowest energy. 
The requirement \ref{eq:Even} implies that    
\begin{align*}
\tilde{E}(u_{k},\Omega_{[3\frac{2\pi}{2^{k+1}},\frac{2\pi}{2^{k-1}}]})
\ &\leq \
\tilde{E}(u_{k},\Omega_{[\frac{2\pi}{2^{k}},3\frac{2\pi}{2^{k+1}}]}),
\end{align*}
while requirement  \ref{eq:Periodic} gives
\begin{align*}
\tilde{E}(u_{k},\Omega_{[4m\frac{2\pi}{2^{k+1}},(4m+1)\frac{2\pi}{2^{k+1}}]}) 
\ &\leq \
\tilde{E}(u_{k},\Omega_{[(4m+1)\frac{2\pi}{2^{k+1}},(4m+2)\frac{2\pi}{2^{k+1}}]}), \\
\tilde{E}(u_{k},\Omega_{[(4m+3)\frac{2\pi}{2^{k+1}},(4m+4)\frac{2\pi}{2^{k+1}}]})
\ &\leq \
\tilde{E}(u_{k},\Omega_{[(4m+2)\frac{2\pi}{2^{k+1}},(4m+3)\frac{2\pi}{2^{k+1}}]}),
\end{align*}
for $m=1,2,\ldots,2^{k-1}-1$.
We define $u_{k+1}:\Omega\to\mathbb{R}^{3}$ on $[0,\frac{2\pi}{2^{k+1}}]$ by $u_{k+1}=u_k$ and  then extend $u_{k}$ to the rest of $\Omega$ by reflecting the coefficient functions at the plane at angle $\theta=\frac{2\pi}{2^{k+1}}$ to extend to sector $\Omega_{[0,\frac{2\pi}{2^{k}}]}$ and then extend $u_{k+1}$ periodically. 
More precisely, we define $u_{k+1}$ for $\theta\in (\frac{2\pi}{2^{k+1}},\frac{2\pi}{2^{k}}]$ via
\begin{align*}
u_{k+1}(\rho,\theta,z)
\ &\coloneqq \
u_{k,\rho}(\rho,\tfrac{2\pi}{2^{k}}-\theta,z)\mathbf{e}_{\rho}
+ u_{k,\theta}(\rho,\tfrac{2\pi}{2^{k}}-\theta,z)\mathbf{e}_{\theta}
+ u_{k,z}(\rho,\tfrac{2\pi}{2^{k}}-\theta,z)\mathbf{e}_{z}
\, .
\end{align*}
For the periodic extension we note that for $\theta\in (\frac{2\pi}{2^{k}},2\pi]$ there exists $\llbracket \theta\rrbracket_k\in [0,\frac{2\pi}{2^{k}})$ such that $\theta = \llbracket \theta\rrbracket_k + \ell \frac{2\pi}{2^{k}}$ for some $\ell\in\mathbb{N}$. We then set
\begin{align*}
u_{k+1}(\rho,\theta,z)
\ &\coloneqq \
u_{k+1,\rho}(\rho,\llbracket \theta\rrbracket_k,z)\mathbf{e}_{\rho}
+ u_{k+1,\theta}(\rho,\llbracket \theta\rrbracket_k,z)\mathbf{e}_{\theta}
+ u_{k+1,z}(\rho,\llbracket \theta\rrbracket_k,z)\mathbf{e}_{z}
\, .
\end{align*}
It is clear from this definition that $u_{k+1}\in{}W^{1,2}(\Omega;\mathbb{R}^{3})$, and that we preserved the norm $|u_{k+1}| = 1$ a.e.\
and boundary condition $u_{k+1}\cdot\nu=0$ on $\partial\Omega$ so that requirements \ref{eq:BasicConditions} and \ref{eq:LInftyPreservation} hold.
By construction, the requirements \ref{eq:Even} and \ref{eq:Periodic} are also satisfied.
Finally, notice that
\begin{equation*}
    \tilde{E}(u_{k+1},\Omega)
    \ = \ 
    2^{k+1} \tilde{E}(u_{k},\Omega_{[0,\frac{2\pi}{2^{k+1}}]})
    \ \leq \
    \tilde{E}(u_{k},\Omega).
\end{equation*}
which gives requirement \ref{eq:EnergyDecreasing}.

\begin{figure}
\begin{center}
\begin{tikzpicture}[scale=1]
\pgfmathsetmacro{\r}{3} 	
\pgfmathsetmacro{\k}{4} 	
\pgfmathsetmacro{\kk}{2^\k} 
\pgfmathsetmacro{\kkm}{2^\k-1} 
\pgfmathsetmacro{\d}{0.1} 	

\fill[gray!20!blue] (0,0) -- (90:\r) arc (90:{90-360/\kk}:{\r}) -- cycle;
\fill[gray!20!cyan] (0,0) -- ({90-360/\kk}:{\r}) arc ({90-360/\kk}:{90-2*360/\kk}:{\r}) -- cycle;
\draw[] (90:{\r-\d})--(90:{\r+\d});
\node[] at (90:{\r+3*\d}) {$\theta=0$};

\draw[] ({90-360/\kk}:{\r-\d})--({90-360/\kk}:{\r+\d});
\node[] at ({90-360/\kk}:{\r+4*\d}) {$\theta=\frac{2\pi}{2^{k+1}}$};
\draw[] ({90-2*360/\kk}:{\r-\d})--({90-2*360/\kk}:{\r+\d});
\node[] at ({90-2*360/\kk}:{\r+5*\d}) {$\theta=\frac{2\pi}{2^{k}}$};

\foreach \m in {2,4,...,\kkm}{
	\fill[gray!40] (0,0) -- ({90-\m*360/\kk}:\r) arc ({90-\m*360/\kk}:{90-(\m+1)*360/\kk}:{\r}) -- cycle;
	\fill[gray!20] (0,0) -- ({90-(\m+1)*360/\kk}:\r) arc ({90-(\m+1)*360/\kk}:{90-(\m+2)*360/\kk}:{\r}) -- cycle;
	} 

\draw[black,line width=1] (0,0) circle (\r cm);

\end{tikzpicture}  
\caption{Schematic representation of the symmetrization procedure for $k=4$ on a slice through $\Omega$ for constant $z$. Dark blue represents the energy minimizing sector, light blue its reflection, in grey the periodic extension.}
\label{fig:thin_film_constr}
\end{center}
\end{figure}
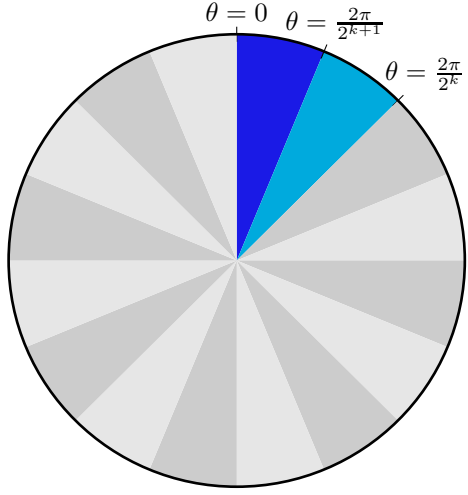

\vspace{0.5cm}
\textit{Step 2: Passing to the limit in $k$.}
With the sequence $\{u_{k}\}_{k=1}^{\infty}$ constructed in Step 1, we notice that by \ds{ definition of $\tilde{E}$ and \eqref{prop:symmetrization:tildeE_E_T_theta}}
\begin{align*}
\left\|\nabla{}u_{k}\right\|_{L^{2}(\Omega;\mathbb{R}^{3})}^{2}
\ &\leq \ 8\widetilde{E}(u_{k},\Omega)
\ \leq \ 8\widetilde{E}(u,\Omega)
\ \leq \ 8\left\|\nabla{}u\right\|_{L^{2}(\Omega;\mathbb{R}^{3})}^{2}
 \, , \\
\left\|u_{k}\right\|_{L^{2}(\Omega;\mathbb{R}^{3})}^{2}
\ &= \
\frac{4}{3}\pi
\, .
\end{align*}
We conclude that $\{u_{k}\}_{k=1}^{\infty}$ is uniformly bounded in
$W^{1,2}(\Omega;\mathbb{R}^{3})$.
By Banach-Alaoglu, there exists a subsequence (not relabelled)
which converges weakly in $W^{1,2}(\Omega;\mathbb{R}^{3})$ to
$u_{\infty}\in{}W^{1,2}(\Omega;\mathbb{R}^{3})$.
By compactness of the Sobolev Embedding we can find a further subsequence such
that $u_{k}\to{}u_{\infty}$ in $L^{4}(\Omega;\mathbb{R}^{3})$.
By weak continuity of the trace operator and compactness of the embedding $W^{\frac{1}{2},2}(\partial\Omega;\mathbb{R}^{3})$ into
$L^{p}(\partial\Omega;\mathbb{R}^{3})$ for $1\le{}p<4$, we then also have that for $k\to\infty$
\begin{align*}
    \int_{\partial\Omega} |u_\infty\cdot\nu|^2 \dx\sigma
    \ \xleftarrow{} \
    \int_{\partial\Omega} |u_k\cdot\nu|^2 \dx\sigma
    \ = \
    0
    \, ,
\end{align*}
i.e.\ $u_{\infty}$ satisfies the boundary conditions $u_\infty\cdot\nu=0$ almost everywhere on $\partial\Omega$.
Since $u_{k}\to{}u_{\infty}$ in $L^{4}(\Omega;\mathbb{R}^{3})$ then
\begin{equation*}
    0=\lim_{k\to\infty}
    \int_{\Omega_{\rho,\theta,z}}\!{}
    (|u_{k}|^{2}-1)^{2}
    \dx x
    \ = \
    \int_{\Omega_{\rho,\theta,z}}\!{}
    (|u_{\infty}|^{2}-1)^{2}
    \dx x
    \, ,
\end{equation*}
which implies that $|u_\infty|^2 = 1$ almost everywhere in $\Omega$.

It remains to show that $u_\infty$ is equivariant.
We start by noticing that if $k\ge{}m$ then
$\frac{2\pi}{2^{m}}$ is an integer multiple of $\frac{2\pi}{2^{k}}$.
Thus, $u_{k}$ has components which are $\frac{2\pi}{2^{m-1}}$-periodic in
$\theta$ for $k\ge{}m$.
Since $u_{k}\to{}u_{\infty}$ in $L^{2}(\Omega;\mathbb{R}^{3})$ then
\begin{equation*}
    \left\|u_{\infty,a}\biggl(\rho,\theta+\frac{2\pi}{2^{m-1}},z\biggr)-
    u_{\infty,a}(\rho,\theta,z)\right\|
    _{L^{2}(\Omega_{\rho,\theta,z})}=0
\end{equation*}
for $a=\rho,\theta,z$ and all $m=1,2,\ldots$.
We conclude that
$u_{\infty}\bigl(\rho,\theta+\frac{2\pi}{2^{m-1}},z\bigr)=
u_{\infty}(\rho,\theta,z)$ for almost every point in $\Omega_{\rho,\theta,z}$ for $m=1,2,\ldots$.
Hence, $u_{\infty,a}$ is $\frac{2\pi}{2^{m-1}}$-periodic in $\theta$ for all
$m\in\mathbb{N}$ and $a\in\{\rho,\theta,z\}$.

Next, we consider a test function $\varphi\in{}C_{c}^\infty(\Omega_{\rho,\theta,z})$
and we choose $m\in\mathbb{N}$ large enough that
\begin{equation}\label{eq:bddist}
    \text{dist}(\text{supp}(\varphi),\partial\Omega_{\rho,\theta,z})
    >\frac{2\pi}{2^{m-1}}.
\end{equation}
We notice that since
\begin{equation*}
    \partial_{\theta}\varphi(\rho,\theta,z)=
    \lim_{m\to\infty}-\frac{\varphi(\rho,\theta-\frac{2\pi}{2^{m-1}},z)-
    \varphi(\rho,\theta,z)}{\frac{2\pi}{2^{m-1}}}
\end{equation*}
for each $(\rho,\theta,z)\in\Omega_{\rho,\theta,z}$ as well as that
\begin{equation*}
    \biggl|\frac{\varphi(\rho,\theta-\frac{2\pi}{2^{m-1}},z)-
    \varphi(\rho,\theta,z)}{\frac{2\pi}{2^{m-1}}}\biggr|
    \le\left\|\nabla\varphi\right\|_{L^{\infty}(\Omega_{\rho,\theta,z})}
\end{equation*}
then by the Dominated Convergence Theorem
used with dominant function $g\coloneqq{}u_{\infty,a}\left\|\nabla\varphi\right\|_{L^{\infty}(\Omega_{\rho,\theta,z})}$, we have, for $a\in\{\rho,\theta,z\}$ that
\begin{align*}
    \int_{\Omega_{\rho,\theta,z}}\!{}
    u_{\infty,a}(\rho,\theta,z)\partial_{\theta}\varphi(\rho,\theta,z)
    =&\lim_{m\to\infty}-\int_{\Omega_{\rho,\theta,z}}\!{}
    u_{\infty,a}(\rho,\theta,z)\cdot
    \frac{\varphi(\rho,\theta-\frac{2\pi}{2^{m-1}},z)
    -\varphi(\rho,\theta,z)}{\frac{2\pi}{2^{m-1}}}\\
    =&\lim_{m\to\infty}
    \biggl[-\int_{\text{supp}(\varphi)+(0,\frac{2\pi}{2^{m-1}},0)}\!{}
    u_{\infty,a}(\rho,\theta,z)\cdot
    \frac{\varphi(\rho,\theta-\frac{2\pi}{2^{m-1}},z)
    }{\frac{2\pi}{2^{m-1}}}\\
    &+\int_{\text{supp}(\varphi)}\!{}u_{\infty,a}(\rho,\theta,z)\cdot
    \frac{\varphi(\rho,\theta,z)}{\frac{2\pi}{2^{m-1}}}\biggr]
\end{align*}
where we have restricted to $\text{supp}(\varphi)+\bigl(0,\frac{2\pi}{2^{m-1}},0\bigr)$
after using \eqref{eq:bddist}.
Next we argue by discrete integration by parts in order to make use of the
$\theta$-periodicity of $u_{\infty,a}$ for $a\in\{\rho,\theta,z\}$.
By the Change of Variables Theorem applied with the map
$(\rho,\theta,z)\mapsto\bigl(\rho,\theta+\frac{2\pi}{2^{m-1}},z\bigr)$
we have, after recombining the integrals, that
\begin{align*}
    \int_{\Omega_{\rho,\theta,z}}\!{}
    u_{\infty,a}(\rho,\theta,z)\partial_{\theta}\varphi(\rho,\theta,z)
    &=\lim_{m\to\infty}-\int_{\text{supp}(\varphi)}\!{}
    \frac{u_{\infty,a}(\rho,\theta+\frac{2\pi}{2^{m-1}},z)
    -u_{\infty,a}(\rho,\theta,z)}{\frac{2\pi}{2^{m-1}}}\cdot\varphi(\rho,\theta,z)\\
    &=\lim_{m\to\infty}-\int_{\Omega_{\rho,\theta,z}}\!{}
    \frac{u_{\infty,a}(\rho,\theta+\frac{2\pi}{2^{m-1}},z)
    -u_{\infty,a}(\rho,\theta,z)}{\frac{2\pi}{2^{m-1}}}\cdot\varphi(\rho,\theta,z).
\end{align*}
Since $u_{\infty,a}\bigl(\rho,\theta+\frac{2\pi}{2^{m-1}},z\bigr)
=u_{\infty,a}(\rho,\theta,z)$ for $m\in\mathbb{N}$ and $a\in\{\rho,\theta,z\}$
for $\mathcal{L}^{3}$-almost every point in $\Omega_{\rho,\theta,z}$ then
we have
\begin{equation*}
    \int_{\Omega_{\rho,\theta,z}}\!{}
    \frac{u_{\infty,a}(\rho,\theta+\frac{2\pi}{2^{m-1}},z)
    -u_{\infty,a}(\rho,\theta,z)}{\frac{2\pi}{2^{m-1}}}\cdot\varphi(\rho,\theta,z)
    =0
\end{equation*}
for all $m\in\mathbb{N}$.
Taking the limit as $m\to\infty$ we see that
\begin{equation*}
    \int_{\Omega_{\rho,\theta,z}}\!{}u_{\infty,a}(\rho,\theta,z)
    \partial_{\theta}\varphi(\rho,\theta,z)
    =0.
\end{equation*}
Thus, for all $\varphi\in{}C_{c}^{\infty}(\Omega_{\rho,\theta,z})$ we have that
\begin{equation*}
    -\int_{\Omega_{\rho,\theta,z}}\!{}\partial_{\theta}u_{\infty,a}(\rho,\theta,z)
    \varphi(\rho,\theta,z)
    =\int_{\Omega_{\rho,\theta,z}}\!{}u_{\infty,a}(\rho,\theta,z)
    \partial_{\theta}\varphi(\rho,\theta,z)
    =0.
\end{equation*}
By the Fundamental Lemma of the Calculus of Variations we have that
$\partial_{\theta}u_{\infty,a}=0$ for $a\in\{\rho,\theta,z\}$ and for
$\mathcal{L}^{3}$-almost every point in $\Omega_{\rho,\theta,z}$.
Let $\eta\colon{}\mathbb{R}^{3}\to\mathbb{R}$ be a smooth, non-negative function
with compact support in $\overline{B_{1}(0)}$ satisfying $\int_{\mathbb{R}^{3}}\!{}\eta=1$.
Consider $\delta>0$ and set $\eta_{\delta}(x)\coloneqq\frac{1}{\delta^{3}}\eta\bigl(\frac{x}{\delta}\bigr)$
and define $u_{\infty,a,\delta}\coloneqq\eta_{\delta}\star{}u_{\infty,a}$ for
$a\in\{\rho,\theta,z\}$.
Observe that $u_{\infty,a,\delta}$ is smooth for $\delta>0$ and
\begin{equation*}
    \partial_{\theta}u_{\infty,a,\delta}
    =\eta_{\delta}\star(\partial_{\theta}u_{\infty,a})
    =0.
\end{equation*}
We conclude for each $\delta>0$ that $u_{\infty,a,\delta}$ is independent of
$\theta$ for all $a\in\{\rho,\theta,z\}$.
Letting $\delta\to0^{+}$ and using that $u_{\infty,a,\delta}$ converges pointwise
almost everywhere, in an $\mathcal{L}^{3}$ sense, to $u_{\infty,a}$ we find that
$u_{\infty,a}$ is independent of $\theta$ for each $a\in\{\rho,\theta,z\}$.

We conclude that $u_{\infty,a}$ is independent of $\theta$ for each
$a\in\{\rho,\theta,z\}$ and hence equivariant.

In order to show that $E(u_\infty)\leq E(u)$, we note that by equivariance of $u_\infty$ it holds that
\begin{align}\label{prop:symmetrization:E_ineqs_chain}
    E(u_\infty)
    \ = \
    \tilde{E}(u_\infty)
    \ \leq \
    \tilde{E}(u)
    \ \leq \
    E(u) - \T
    \ \leq \
    E(u)
    \, ,
\end{align}
where we used assumption \eqref{prop:symmetrization:T_=_0} that $\T\geq 0$.

\vspace{0.5cm}
\textit{Step 3: Equality case}
If $u$ is equivariant, then $\T=0$ and by construction $u_k=u$ for all $k\in\mathbb{N}$, giving $u_\infty=u$. 
Furthermore, if $u$ is equivariant it holds $E(u)=\tilde{E}(u)$ and thus $E(u_\infty)=\tilde{E}(u_\infty)=\tilde{E}(u)=E(u)$.

For the reverse implication, i.e.\ if $E(u_\infty) = E(u)$, we see that we must have equality in all of the inequalities in \eqref{prop:symmetrization:E_ineqs_chain}.
Therefore it holds that $\tilde{E}(u)=E(u)$ and thus \ds{by \eqref{prop:symmetrization:tildeE_E_T_theta}}
\begin{align*}
\int_{\Omega_{\rho,\theta,z}} 
\frac{1}{4\rho}\left(
(\partial_{\theta}u_\rho)^{2}
+(\partial_{\theta}u_\theta)^{2}
+(\partial_{\theta}u_z)^{2}
\right) \dx\rho \dx\theta \dx z
\ &= \
E(u) - \tilde{E}(u) - \T
\ \leq \ 0
\, ,
\end{align*}
i.e.\ $u$ is equivariant.
\end{proof}

\begin{proof}[Proof of Theorem~\ref{thm:main_symmetry}]
      The first part of Theorem~\ref{thm:main_symmetry} follows from Proposition~\ref{prop:symmetrization} since if $u$ is a minimizer, then $E(u)=E(u_\infty)$ and thus $u$ must be equivariant.

      The second part, is an application of Proposition~\ref{prop:u_theta_0_S2_implies_u_theta_0_equiv} since by minimality of $u$ it holds $E(\widetilde{u})=E(u)$, \ds{which implies that $u$ must be equivariant and $u\cdot\ee_\theta=0$ in $B_1(0)$.}
\end{proof}

\begin{remark}\label{rem:symm_cover_diff_cases}\hspace{5pt}\\
\vspace{-15pt}
\begin{itemize}
    \item Note that there exist functions $u\in W^{1,2}(\Omega,\mathbb{S}^2)$ such that $u(\Omega)$ is not included in a hemisphere, and that are not equivariant, but still $\T=0$, see Example~\ref{example:non-equiv-fct-T=0}.
    So in particular, the results of Proposition~\ref{prop:symmetrization} apply although the hypothesis \ds{to use the methods from} \cite{BBCH,KaSha,Sa,SaSha2,SaSha} and Proposition~\ref{prop:u_theta_0_S2_implies_u_theta_0_equiv} are not satisfied. 
    \item In \cite{diFraSlaZar}, it is shown for a more general energy that if $\Omega$ is a two-dimensional surface of revolution, then the existence of axially symmetric energy minimizers is equivalent to the existence of axially \emph{null-average} minimizers which is a condition reminiscent of our condition.
    Proposition~\ref{prop:symmetrization} states that if a minimizer satisfies $\T=0$, then it must be equivariant.
    On the other hand, equivariant minimizers trivially satisfy $\T=0$.
\end{itemize}
\end{remark}

\begin{example}\label{example:non-equiv-fct-T=0}
Here we construct a non-equivariant function $\widetilde{u}$ in the admissible class that satisfies $\T=0$.
Appealing to spherical coordinates on 
the boundary $\partial B_1(0)$ we can define a function
\begin{align*}
    u(\theta,\varphi)
    \ \coloneqq \ 
    \sin(\varphi)\ee_\theta + \cos(\varphi)\ee_\varphi
    \, ,
\end{align*}
and note that $u\cdot\nu=0$ and $u(\partial B_1(0))$ contains an open neighborhood of the equator. 
We extend $u$ into $B_1(0)$ by interpolating $u_\rho,u_\theta,u_z$ from the boundary along $\rho$ to the $z-$axis on which we impose $u_\rho= 0 = u_\theta$ and $u_z=1$ and then project the resulting function onto $\mathbb{S}^2$.
Then $u\in W^{1,2}(\Omega;\mathbb{S}^2)$.
Note that $u$ is equivariant and thus satisfies $\T=0$.
We can now perturb $u$ by choosing a compactly supported test function $\eta\in C_c^\infty(B_{\frac12}\setminus \{\rho>\frac14\})$.
Define
\begin{align*}
    \widetilde{u}
    \ &\coloneqq \
    u_\rho\ee_\rho
    + (u_\theta + \eta)\ee_\theta
    + \widetilde{u}_z\ee_z
    \, ,
\end{align*}
where $\widetilde{u}_z^2 = u_z^2 - \eta^2 - 2u_\theta\eta$ and thus
$\widetilde{u}_z = \sqrt{u_z^2 - \eta^2 - 2u_\theta\eta}$.
By choosing $\eta=0$ on $\{u_z\leq 0\}$ and $\eta\geq 0$ small enough outside this set, this is well-defined.
Since $u_\rho$ is independent of $\theta$, so is $\widetilde{u}_\rho$ and thus $\widetilde{u}_\theta\partial_\theta\widetilde{u}_\rho = 0$.
Furthermore, $\widetilde{u}_\rho\partial_\theta\widetilde{u}_\theta$ integrated over circles $\rho\mathbb{S}^1 + z\ee_z$ yields zero because of periodicity of $\widetilde{u}_\theta$.
We conclude that $\T=0$ also for $\widetilde{u}$. 
\end{example}

\section{Excluding Interior Defects}
\label{sec:bdry_defects}

We recall that we are interested in finding minimizers $u\in W^{1,2}(\Omega,\mathbb{R}^3)$ of the energy
\begin{equation*}
\frac{1}{2}\int_\Omega |\nabla u|^2 \dx x\, ,
\end{equation*}
subject to the constraint $|u|=1$ and the boundary condition $u\cdot\nu=0$.
In Proposition \ref{subsec:fullreg} we saw that that minimizers $u$ have only finitely many singular points.

From here on, we will assume that $u_\theta=0$ for some choice of axis. 
This allows us to apply Proposition~\ref{prop:symmetrization} an conclude that $u$ is rotationally equivariant around the $\ee_3-$axis.
Therefore, singularities cannot occur on the boundary, except on the axis of equivariance.
This implies in particular that $u$ is continuous on the boundary $\partial\Omega$, except on the two points of intersection with the axis.

The main result of this section is the following proposition which directly implies Theorem~\ref{thm:main_no_interior_defects}:

\begin{proposition}\label{prop:sing_only_on_bdry}
Let $u$ be a minimizer of \eqref{def:dir} satisfying \eqref{eq:tanbc} with $u_\theta=0$. 
Then $u$ has exactly two singularities that appear on the boundary at the North and South pole.
\end{proposition}

The proof proceeds similarly to \cite[Section 4]{ABL}. 

\begin{proof}
We express $u$ in cylindrical coordinates $(\rho,\theta,z)$ and use the fact that $u_\theta=0$ and equivariance to write $u = \sin(\psi)\ee_\rho + \cos(\psi)\ee_z$, where $\psi=\psi(\rho,z)$ is a real-valued $W^{1,2}$ function.
\begin{equation}\label{def:E:psi}
E_0(\psi)
\ \coloneqq \
\int_\Omega \left( |\partial_\rho \psi|^2 + |\partial_z\psi|^2 + \frac{1}{\rho^2}\sin^2(\psi) \right)\dx x\, ,
\end{equation}

The boundary condition \eqref{eq:tanbc} translates to
\begin{equation*}
0
\ = \ u\cdot\nu
\ = \ \sin(\psi)\sin(\varphi) + \cos(\psi)\cos(\varphi)
\ = \ \cos(\varphi-\psi)
\end{equation*}
or in other words, $\psi = \varphi \pm \tfrac{\pi}{2} + 2\pi \mathbb{Z}$ on $\partial\Omega$.
We therefore consider the boundary condition
\begin{equation}\label{def:lim_tang:bc_psi}
\psi
\ = \ \varphi - \frac{\pi}{2} \, .
\end{equation}
Indeed, a minimizer has to be continuous on the boundary according to the previous section (or use $W^{1,2}-$regularity as in ABL).

It therefore holds that 
\begin{equation*}
-\frac{\pi}{2}
\ \leq \ \psi
\ \leq \ \frac{\pi}{2}
\qquad \text{ on }\partial\Omega.
\end{equation*}

Given a function $\psi$, we can define the competitors $\psi_1 \coloneqq \max\{\psi,-\pi\}$ and $\psi_2 \coloneqq \min\{\psi,\pi\}$.
Both $\psi_1$ and $\psi_2$ satisfy the same boundary conditions \eqref{def:lim_tang:bc_psi}
and decrease the energy \eqref{def:E:psi}.
We therefore conclude that for any minimizer $\psi$ the inequalities
\begin{equation*}
-\pi
\ \leq \ \psi
\ \leq \ \pi
\qquad \text{in }\Omega
\end{equation*}
hold.

By finiteness of the energy \eqref{def:E:psi}, $\psi$ is forced to take values in $\{-\pi,0,\pi\}$ on the axis $\rho=0$.
In order to conclude that there are no singularities in the interior of $\Omega$, we have to show that $\psi$ is constant on the line segment $\rho=0$ for $r<1$.
The goal is to prove that $\psi=0$ on this line segment.\\

\noindent\textit{Step 1.} 
We define 
\begin{equation*}
\begin{aligned}
X_+ 
\ &\coloneqq \ \{ x\in\Omega\sd \psi > 0 \} \\
X_-
\ &\coloneqq \ \{ x\in\Omega\sd \psi < 0 \} 
\end{aligned}
\end{equation*}
and 
\begin{equation} \label{eq:lim_tang:A_pm}
\begin{aligned}
A_+ 
\ &\coloneqq \ \{ x\in\partial\Omega\sd \varphi\in (\tfrac{\pi}{2},\pi)  \} \subseteq \overline{X}_+ \\
A_-
\ &\coloneqq \ \{ x\in\partial\Omega\sd \varphi\in (0,\tfrac{\pi}{2}) \} \subseteq \overline{X}_- 
\end{aligned}
\end{equation}
The inclusions in \eqref{eq:lim_tang:A_pm} follow from the boundary condition for $\psi$ in \eqref{def:lim_tang:bc_psi}.

\begin{figure}
\begin{center}
\begin{tikzpicture}[scale=1]
\pgfmathsetmacro{\R}{3} 	
\pgfmathsetmacro{\h}{2} 	

\draw[fill=gray!10] (0,\R) arc (90:-90:\R) -- cycle;

\draw[red, fill=red!10!white] plot [smooth, tension=0.5] coordinates {(0,\R) (\R/4,\R/2) (\R/2,\R/2) (\R/3,\R/4) (\R/2,-\R/5) (2*\R/3,\R/10) (\R,0)} arc (0:90:\R) -- cycle;
\draw[red, fill=red!10!white] plot [smooth, tension=0.5] coordinates {(0,\R/10) (\R/5,\R/20) (\R/6,-\R/4) (0,-\R/3)} -- cycle;
\node[red] at (1.2*\R/2,0.5*\R/2) {$\omega_{-}$};
\node[red] at (\R/10,-\R/10) {$X_{-}$};
\node[] at (\R/1.5,-\R/2) {$\omega_{+}$};
\node[] at (\R/5,\R/3) {$X_{+}$};

\draw[->] (0,0)--(\R*1.10,0) node[right] {$\rho$};
\draw[->] (0,-1.05*\R)--(0,\R*1.10) node[above] {$z$};
\draw[line width=1.5] (0,\R) -- (0,-\R);

\draw[line width=1.5, blue] (0,\R) arc (90:0:\R);
\draw[line width=1.5, green!50!black] (\R,0) arc (0:-90:\R);
\node[blue] at (1.5*\R/2,1.6*\R/2) {$A_{-}$};
\node[green!50!black] at (1.5*\R/2,-1.6*\R/2) {$A_{+}$};

\node[] at (0.75,\R+0.25) {$\psi=-\frac{\pi}{2}$};
\node[] at (\R+0.5,0.25) {$\psi=0$};
\node[] at (0.6,-\R-0.25) {$\psi=\frac{\pi}{2}$};
\node[] at (-0.5,0) {$\psi=0$};
\end{tikzpicture}  
\caption{Domain and notation used in the proof of Proposition \ref{prop:sing_only_on_bdry}.}
\label{fig:excl_int_defects}
\end{center}
\end{figure}
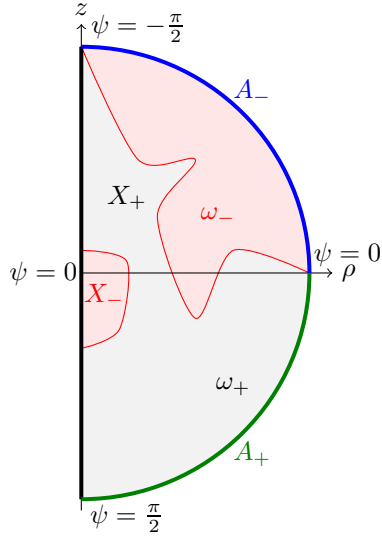

Let $\omega_\pm$ be the connected component of $X_\pm$ containing $A_\pm$. 
We want to show that $X_+ = \omega_+$.
Consider the function
\begin{equation*}
    \tilde{\psi} 
    \ \coloneqq \ \begin{cases}
    \psi & \omega_+ \, , \\
    \min\{\psi,-\psi\} & \Omega\setminus \omega_+ \, .
    \end{cases}
\end{equation*}
Then $\tilde{\psi}\in H^1(\Omega)$ and $E_0(\tilde{\psi}) = E_0(\psi)$ because of the symmetries of $\sin^2(\psi)$.
Furthermore, $\tilde{\psi}$ satisfies the boundary conditions \eqref{def:lim_tang:bc_psi} since 
on $A_+\subset\overline{\omega}_+$
it holds $\tilde{\psi}=\psi$ 
and on $\partial\Omega\setminus A_+$ we have $\psi = \varphi-\tfrac{\pi}{2}\leq 0$, i.e.\ $\min\{\psi,-\psi\}=\psi$. 
Therefore $\tilde{\psi}$ is a minimizer of $E_0$ and is analytic away from the axis $\rho=0$. 
Since $X_-$ is open and non-empty, there is an open subset of $\Omega\setminus \overline{\omega_+}$ such that $\psi < 0$. In this open subset, the analytic functions $\psi$ and $\min\{\psi,-\psi\}$ coincide and hence they coincide everywhere in $\Omega\setminus \overline{\omega_+}$.
It follows that $\psi \leq -\psi$ in $\Omega\setminus \overline{\omega_+}$, or in other words $\psi\leq 0$. 
We conclude that $X_+$ is connected. 

A similar argument allows us to show that $X_-$ is connected. 
We consider the function
\begin{equation*}
    \tilde{\psi} 
    \ \coloneqq \ \begin{cases}
    \psi & \omega_- \, , \\
    \max\{\psi,-\psi\} & \Omega\setminus \omega_- \, .
    \end{cases}
\end{equation*}
As before, one can show that $\tilde{\psi}$ satisfies the boundary conditions \eqref{def:lim_tang:bc_psi} and has the same energy as $\psi$, and is therefore a minimizer of $E_0$. 
Again by analyticity, we show that $\psi\geq -\psi$ on $\Omega\setminus \overline{\omega_-}$, i.e.\ $\psi\geq 0$ and therefore $X_-=\omega_-$ is connected.

\textit{Step 2.}
Let's define the competitor
\begin{equation*}
    \tilde{\psi} 
    \ \coloneqq \ \begin{cases}
    \min\{\psi,\pi-\psi\} & X_+ \, , \\
    \psi & \Omega\setminus X_+ \, .
    \end{cases}
\end{equation*}
We have that $\tilde{\psi}\in H^1(\Omega)$ and $E_0(\tilde{\psi}) = E_0(\psi)$ because $\sin^2(\pi-\psi)=\sin^2(\psi)$.
Furthermore, $\tilde{\psi}$ satisfies the boundary conditions \eqref{def:lim_tang:bc_psi} since on $A_+\subset \overline{X}_+$ it holds $0\leq \psi\leq\tfrac{\pi}{2}$, i.e.\ $\min\{\psi,\pi-\psi\}=\psi$ and on $A_-$ we have $\tilde{\psi}=\psi$.
Therefore $\tilde{\psi}$ is a minimizer of $E_0$ and is analytic away from the axis $\rho=0$. 
We observe that $0<\tfrac{\pi}{8}\leq \psi\leq \tfrac{3\pi}{8}<\tfrac{\pi}{2}$ in a small open neighborhood of the boundary point $\psi(\sqrt{2}/2,-\sqrt{2}/2)$. 
In particular, this neighborhood is contained in $X_+$.
In this open subset, the analytic functions $\psi$ and $\tilde{\psi}$ coincide and hence they coincide everywhere in $X_+$.
It follows that $\psi = \min\{\psi,\pi-\psi\}$ in $X_+$ from which $\psi \leq \pi-\psi$ in $X_+$, or in other words $\psi\leq \tfrac{\pi}{2}$.

It is possible to do a similar comparison argument with the competitor
\begin{equation*}
    \tilde{\psi} 
    \ \coloneqq \ \begin{cases}
    \psi & X_+ \, , \\
    \max\{\psi,-\pi-\psi\} & \Omega\setminus X_+ \, .
    \end{cases}
\end{equation*}
Indeed, $E_0(\tilde{\psi}) = E_0(\psi)$ and $\tilde{\psi}$ satisfies the boundary conditions \eqref{def:lim_tang:bc_psi}.
Therefore, on $X_-$, $\psi=\max\{\psi,-\pi-\psi\}$ which gives $\psi\geq -\pi-\psi$, i.e.\ $\psi\geq -\tfrac{\pi}{2}$.

So in total, we conclude that $-\tfrac{\pi}{2}\leq \psi \leq \tfrac{\pi}{2}$ on the whole of $\Omega$. 
It therefore follows that $\psi\equiv 0$ on the axis $\rho=0$, i.e.\ there are no singularities in the interior of $\Omega$.
\end{proof}

\section{Estimating the Energy of Minimizers}
\label{sec:profile}

The goal of this section is to bound the energy of equivariant minimizers of $E$ of the form $u = u_\rho\ee_\rho + u_z\ee_z$.

In the next lemma, we derive an upper bound on the minimial energy of $E$ by using an explicit competitor $u_0$.
The vector field $u_0$ is obtained by constructing the unit normal vector field of spherical shells that intersect the boundary of $\Omega$ perpendicularly, see Figure~\ref{fig:competitor}.
Note that in the case of Dirichlet boundary data $u=\frac{x}{|x|}$ on $\mathbb{S}^2$, this construction yields precisely the minimizer $u=\frac{x}{|x|}$.

\begin{figure}
\begin{center}
\scalebox{1.5}{\begin{tikzpicture}[scale=1]
\pgfmathsetmacro{\r}{2} 	
\pgfmathsetmacro{\d}{0.15} 	
\pgfmathsetmacro{\a}{0.2} 	

\draw[black,line width=1] (0,0) circle (\r cm);

\draw[dotted] (0,-\r) -- (0,\r);

\draw[blue] (-\r,0) -- (\r,0);
\foreach \t in {-1,-0.5,...,1}{ 
	\draw[black, ->] ({\t*\r},{0}) -- +({0},{\a*1});
	} 

\foreach \z in {-0.2,-0.4,...,-0.8}{ 
	\pgfmathsetmacro{\R}{sqrt((1+\z*\z)*(1+\z*\z)/(4*\z*\z)-1)} 
	\pgfmathsetmacro{\Cz}{(1+\z*\z)/(2*\z)} 
	\pgfmathsetmacro{\angle}{asin((2*\z)/(1+\z*\z))} 
	\draw[domain=-\angle:\angle, smooth, variable=\t, blue] plot ({\r*\R*sin(\t)}, {\r*\R*cos(\t)+\r*\Cz});
	\foreach \s in {-1,-0.5,...,1}{ 
		\draw[black, ->] ({\r*\R*sin(\s*\angle)}, {\r*\R*cos(\s*\angle)+\r*\Cz}) -- +({\a*sin(\s*\angle)},{\a*cos(\s*\angle)});
		} 
	} 

\foreach \z in {0.2,0.4,...,0.8}{ 
	\pgfmathsetmacro{\R}{sqrt((1+\z*\z)*(1+\z*\z)/(4*\z*\z)-1)} 
	\pgfmathsetmacro{\Cz}{(1+\z*\z)/(2*\z)} 
	\pgfmathsetmacro{\angle}{asin((2*\z)/(1+\z*\z))} 
	\draw[domain=-\angle:\angle, smooth, variable=\t, blue] plot ({\r*\R*sin(\t)}, {-\r*\R*cos(\t)+\r*\Cz});
	\foreach \s in {-1,-0.5,...,1}{ 
		\draw[black, ->] ({\r*\R*sin(\s*\angle)}, {-\r*\R*cos(\s*\angle)+\r*\Cz}) -- +({-\a*sin(\s*\angle)},{\a*cos(\s*\angle)});
		} 
	} 

\fill[red] (-\d,-\r-\d/5) -- (\d,-\r-\d/5) arc(0:180:\d) -- cycle;
\fill[red] (-\d, \r+\d/5) -- (\d, \r+\d/5) arc(0:-180:\d) -- cycle;


\end{tikzpicture}}
\caption{Competitor $u_0$ used in Lemma~\ref{lem:upper_bound_competitor}. 
The vector field (black arrows) is the unit normal vector field of level sets (blue lines). 
The level sets are included in a family of spherical caps of varying radius and center such that they intersect the boundary $\partial\Omega$ and the axis $\ee_3$ perpendicularly, thus making $u_0$ tangential to $\partial\Omega$. 
The two red half disks represent the boundary defects of $u_0$ created by the radius of the level sets shrinking to zero at the North and South pole. 
On the equator the radius is infinite and thus $u_0=\ee_3$.}
\label{fig:competitor}
\end{center}
\end{figure}
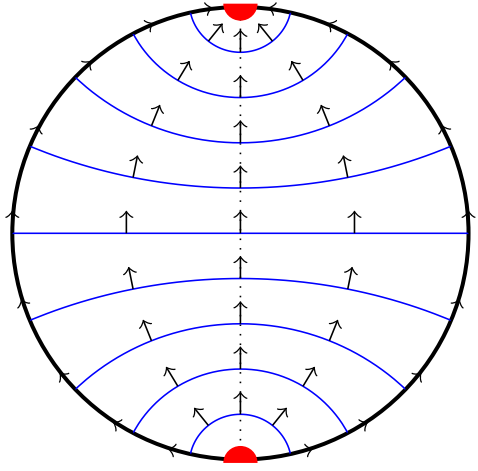

More precisely, we have:

\begin{lemma}\label{lem:upper_bound_competitor}
The vector field
\begin{align*}
    u_0(x)
    \ = \
    \frac{(x^2+y^2-z^2+1)\ee_z-2z\sqrt{x^2+y^2}\ee_\rho}{\sqrt{4z^2(x^2+y^2) + (1+x^2+y^2-z^2)^2}}
\end{align*}
satisfies $|u_0|=1$ and $u_0\cdot\nu=0$ on $\mathbb{S}^2$ and is thus in the admissible class for $E$.
It holds that $E(u_0)=5\pi - \tfrac{\pi^3}{4}\approx 7.956<\infty$.
\end{lemma}

\begin{proof}
From definition and since $\ee_\rho\cdot\ee_z=0$ it is clear that $|u_0|=1$.
Furthermore, we note that $\nu = \rho\ee_\rho + z\ee_z$ and since $\rho^2+z^2=1$ on $\partial\Omega$ we find
\begin{align*}
    u_0\cdot\nu
    \ &= \
    \frac{(\rho^2-z^2+1)z - 2z\rho^2}{\sqrt{4z^2\rho^2 + (1+\rho^2-z^2)^2}}
    \ = \
    \frac{(-\rho^2-z^2+1)z}{\sqrt{4z^2\rho^2 + (1+\rho^2-z^2)^2}}
    \ = \
    0
    \, .
\end{align*}

One can compute that for $\rho^2+z^2\leq 1$, $\rho\geq 0$ it holds
\begin{align*}
    |\nabla u_0|^2
    \ &= \
    \frac{4\rho^2+8z^2}{\rho^2+2\rho^2 z^2 + 2\rho^2 + z^4 - 2z^2+1}
    \, .
\end{align*}
Multiplying this with the Jacobian of cylindrical coordinates, $\rho$, and integrating over the unit ball yields the energy $10\pi - \tfrac{\pi^3}{2}$.  
\end{proof}

\begin{remark}\label{rem:upper_bound_competitor}
    The competitor $u_0$ from Lemma \ref{lem:upper_bound_competitor} is \emph{not} a minimizer of $E$.
    This can be seen by plugging $u_0$ into the Euler-Lagrange equations.
    It turns out that $u_0$ satisfies these equations only in the direction of the linearly independent vector fields $\ee_\theta$ and $u_0$, but not $(u_0)_z\ee_\rho - (u_0)_\rho \ee_z$.
\end{remark}

In the remainder of this section, we compute a lower bound on the energy using \emph{level set coordinates}.

For a level set $\{u_\rho=c\}$ for some constant $c\in [-1,1]$ we introduce the parametrization $\lambda=(\lambda_\rho,\lambda_z)$, i.e.\ a function $\lambda\colon(0,L(c))\rightarrow \{(\rho,z)\in [0,\infty)\times\mathbb{R}\sd \rho^2+z^2\leq 1 \}$ solution to the ODEs
\begin{align*}
    \frac{\dx}{\dx t}\lambda \ = \ (\nabla^\perp u_\rho)\circ\lambda
    \quad \text{where } \nabla^\perp = (-\partial_z, \partial_\rho)
    \, .
\end{align*}
We want to change coordinates from $(\rho,z)$ to $(t,c)$ where $t$ is the parameter for $\lambda$ on the level set $\{u_\rho=c\}$.
We define $\Phi(t,c) \coloneqq (\rho,z)$.

\begin{lemma}\label{lem:jacobian_level_set_coord}
    For $(\rho,z)\neq(0,0)$ we have $\det(\nabla\Phi) = -1$.
\end{lemma}

\begin{proof}
    In order to calculate $\nabla t$, we note that
\begin{align*}
(1,0)
\ &= \ \partial_\rho (\rho,z)
\ = \ \partial_\rho \lambda(t,c)
\ = \ \dot{\lambda}(t) \partial_\rho t \ + \ \partial_c \lambda\:\partial_\rho c
\ = \ \nabla^\perp u_\rho \: \partial_\rho t \ + \ \partial_c \lambda\:\partial_\rho c
\end{align*}
and 
\begin{align*}
(0,1)
\ &= \ \partial_z (\rho,z)
\ = \ \partial_z \lambda(t)
\ = \ \dot{\lambda}(t) \partial_z t \ + \ \partial_c \lambda\:\partial_z c
\ = \ \nabla^\perp u_\rho \: \partial_z t \ + \ \partial_c \lambda\:\partial_z c
\end{align*}
By definition $\nabla c = \nabla u_\rho$.

Solving the four equations for the four unknowns $\partial_\rho t,\partial_z t,\partial_c\lambda_\rho,\partial_c\lambda_z$, we obtain that 
\begin{align*}
\begin{pmatrix}
\partial_\rho t \\ 
\partial_z t \\
\partial_c\lambda_\rho \\
\partial_c\lambda_z
\end{pmatrix}
\ &= \
\begin{pmatrix}
0 \\ 
\frac{1}{\partial_\rho u_\rho} \\
\frac{1}{\partial_\rho u_\rho} \\
0
\end{pmatrix}
\ + \ q 
\begin{pmatrix}
-1 \\ 
-\frac{\partial_z u_\rho}{\partial_\rho u_\rho} \\
-\frac{\partial_z u_\rho}{\partial_\rho u_\rho} \\
1
\end{pmatrix}
\, ,
\end{align*}
where $q=q(t,c)$ is some unknown function.

This unknown function cancels in the Jacobian and we obtain
\begin{align*}
\det(\nabla\Phi^{-1})
\ &= \ 
-1 - q\left( \partial_z u_\rho - \partial_z u_\rho \right)
\ = \
-1 
\, ,
\end{align*}
which gives $\det(\nabla\Phi)=-1$.
\end{proof}

\begin{proposition}\label{prop:lower_bound_level_set_coord}
    Let $u$ be an admissible minimizer of \eqref{def:dir}. 
    Assume also that $u$ is equivariant and of the form $u = u_\rho\ee_\rho + u_z\ee_z$.
    Then
    \begin{align*}
    E(u)
    \ &\geq \
    \frac{4}{3}\pi(2\sqrt{2}-1)
    \ \approx \
    7.6589
    \, .
\end{align*}
\end{proposition}

\begin{proof}
We note the useful identities
\begin{align*}
|\nabla u_\rho|^2
\ &= \ |\nabla^\perp u_\rho|^2 
\  = \ |\dot{\lambda}|^2 
\ = \ |(\dot{\lambda}\cdot\nabla^\perp) u_\rho|
\end{align*}
Using Lemma \ref{lem:jacobian_level_set_coord}, we can rewrite the integral over $\Omega = \Phi(D)$ as
\begin{align*}
E_0(u) 
\ &= \ 
\int_\Omega \left(\frac{\rho}{2}|\nabla u|^2 + \frac{1}{2\rho}|u_\rho|^2\right) \dx \rho \dx z 
\ = \ 
\int_D \left(\frac{\lambda_\rho}{2}|\nabla u|^2 + \frac{1}{2\lambda_\rho}|u_\rho|^2\right) \dx t \dx c \\
\ &= \ \int_{-1}^1 \int_0^{L(c)} \left(\frac{\lambda_\rho}{2(1-c^2)}|\Dot{\lambda}|^2 + \frac{c^2}{2\lambda_\rho}\right) \dx t \dx c
 \, .
\end{align*}
For fixed $c$, minimizing the inner integral w.r.t.\ $\lambda$ yields the Euler-Lagrange equations
\begin{align*}
0
\ &= \ 
\frac{\dx}{\dx t}\left( \frac{\lambda_\rho \Dot{\lambda_z}}{1-c^2} \right) \\
0
\ &= \
\frac{\dx}{\dx t}\left( \frac{\lambda_\rho \Dot{\lambda_\rho}}{1-c^2} \right) - \frac{1}{2}\frac{|\Dot{\lambda}|^2}{1-c^2} + \frac{c^2}{2\lambda_\rho^2}
\, .
\end{align*}
The first equation means that there is a constant $\alpha(c)$ such that $\lambda_\rho\Dot{\lambda_z}=\alpha(c)$. 
It will later turn out to be useful to write $\alpha(c) = \gamma c\sqrt{1-c^2}$. 
Note however that $\gamma$ might still be dependent on $c$.

Substituting this into the equation for $\lambda_\rho$, multiplying by $1-c^2$ and dividing by $\lambda_\rho$, we get
\begin{align}\label{eq:ode_lam_rho_final}
0
\ &= \
\Ddot{\lambda}_\rho 
\ + \ \frac{1}{2 \lambda_\rho}\Dot{\lambda}_\rho^2 
\ + \ \frac{(1-\gamma^2)c^2(1-c^2)}{2\lambda_\rho^3} 
\, .
\end{align}
This ODE admits the implicit solution 
\begin{align}\label{eq:lam_rho_poly_cubic_impl}
(c_2 + t)^2
\ &= \
\frac{1}{c_1} 
\left( \lambda_\rho + \frac{3b}{c_1} \right)
\left( \lambda_\rho - \frac{6b}{c_1} \right)^2
\, ,
\end{align}
for some constants $c_1,c_2$ and $b=\frac{3}{4}(1-\gamma^2)c^2(1-c^2)$.

If we parametrize $\lambda_\rho$ such that $\lambda_\rho(0)=0$, then we obtain the relation $c_2^2 = \frac{108 b^3}{c_1^4}$. 
Solving the cubic polynomial equation \eqref{eq:lam_rho_poly_cubic_impl} and using the fact that $\Im(\lambda_\rho)=0$, we can write 
\begin{align}\label{eq:lam_rho_poly_cubic_expl}
    \lambda_\rho(t)
\ &= \ 
\frac{3b}{c_1} \Bigg(
1 
- \cos\left( \frac{1}{3}\arctan\left( \frac{\sqrt{(c_2+t)^2 t (-2c_2 - t)}}{(c_2+t)^2 - \frac12 c_2^2} \right) \right) \\
&\hspace{2cm}
+ \sqrt{3}\sin\left( \frac{1}{3}\arctan\left( \frac{\sqrt{(c_2+t)^2 t (-2c_2 - t)}}{(c_2+t)^2 - \frac12 c_2^2} \right)  \right) 
\Bigg) \nonumber
\, ,
\end{align}
which has the Taylor expansion in $t=0$
\begin{align*}
\lambda_\rho(t)
\ &= \ 
\frac{2\sqrt{6} b}{c_1\sqrt{|c_2|}}\sqrt{t} + O(t) 
\quad\text{and}\quad
\dot{\lambda}_\rho(t)
\ = \
\frac{\sqrt{6} b}{c_1\sqrt{|c_2|}}\frac{1}{\sqrt{t}} + O(1)
\, .
\end{align*}
Since $\lambda_\rho\dot{\lambda_z}=\alpha(c)$, we get that
\begin{align*}
\dot{\lambda}_z(t)
\ &= \
\frac{\alpha(c) c_1 \sqrt{|c_2|}}{2\sqrt{6} b}\frac{1}{\sqrt{t}} + O(1) 
\, ,
\end{align*}
In order to determine $\alpha$, we calculate the angle under which the level set leaves the singularity at the North/South pole.

On the one hand, the angle between $\dot{\lambda}$ and the $z-$axis is then given by
\begin{align*}
\tan(\varsigma)
\ &= \
\frac{\dot{\lambda}_\rho}{\dot{\lambda}_z}
\ = \
\frac{\frac{\sqrt{6} b}{c_1\sqrt{|c_2|}}\frac{1}{\sqrt{t}}}{\frac{\alpha c_1 \sqrt{|c_2|}}{2\sqrt{6} b}\frac{1}{\sqrt{t}}}
\ = \
\frac{\frac{\sqrt{6} b}{c_1\sqrt{|c_2|}}}{\frac{\alpha c_1 \sqrt{|c_2|}}{2\sqrt{6} b}}
\ = \
\frac{12 b^2}{\alpha c_1^2 |c_2|}
\ = \
\frac{2 \sqrt{b}}{\sqrt{3} \alpha}
\, ,
\end{align*}
where the last inequality is due to $\lambda_\rho(0)=0$, i.e.\ $c_2^2 = \frac{108 b^3}{c_1^4}$.

On the other hand, from the blow up argument we know that the profile around a defect is of the form $\rho\ee_\rho+(z+\mathrm{sign}(c))\ee_z$. 
Thus
\begin{align*}
\tan(\varsigma)
\ &= \
\frac{c}{\sqrt{1-c^2}}
\, .
\end{align*}
Setting those two expression equal to each other and using $b=\frac{3}{4}(c^2(1-c^2)-\alpha^2)$, one gets $\alpha = c(1-c^2) $.
This allows us to reformulate $b$ as $b=\frac34 \alpha c^3$.

The information about $\alpha$ can now be used to calculate a lower bound on the energy:
First we note
\begin{align*}
|\Dot{\lambda}|^2
\ &= \
|\dot{\lambda}_\rho|^2 + |\dot{\lambda}_z|^2
\ = \
|\dot{\lambda}_\rho|^2 + \frac{\alpha^2}{\lambda_\rho^2}
\, ,
\end{align*}
which gives
\begin{align*}
E_0(u) 
\ &= \  
\int_{-1}^1 \int_0^{L(c)} \left(\frac{\lambda_\rho}{2(1-c^2)}|\Dot{\lambda}|^2 + \frac{c^2}{2\lambda_\rho}\right) \dx t \dx c \\
\ &= \ 
\int_{-1}^1 \int_0^{L(c)} \left(\frac{\lambda_\rho}{2(1-c^2)} |\dot{\lambda}_\rho|^2 
+ \frac{\alpha^2 + c^2(1-c^2)}{2(1-c^2)\lambda_\rho}\right) \dx t \dx c 
\, .
\end{align*}
Since $\alpha^2 = c^2(1-c^2)^2$, it holds $\alpha^2 + c^2(1-c^2) = c^2(1-c^2)(2-c^2)$. 
Applying Young's inequality we get
\begin{align*}
E_0(u) 
\ &= \ 
\int_{-1}^1 \int_0^{L(c)} \left(\frac{\lambda_\rho}{2(1-c^2)} |\dot{\lambda}_\rho|^2 
+ \frac{c^2(2-c^2)}{2\lambda_\rho}\right) \dx t \dx c 
\ \geq \
\int_{-1}^1 \left(\sqrt{\frac{c^2(2-c^2)}{1-c^2}}
\int_0^{L(c)} |\dot{\lambda}_\rho| \dx t \right) \dx c
\, .
\end{align*}
Then we use that $\int_0^{L(c)} |\dot{\lambda}_\rho| \dx t\geq |\lambda_\rho(L(c)) - \lambda_\rho(0)|=\sqrt{1-c^2}$ to obtain
\begin{align*}
E_0(u) 
\ &\geq \
\int_{-1}^1 \left(\sqrt{\frac{c^2(2-c^2)}{1-c^2}}
\int_0^{L(c)} |\dot{\lambda}_\rho| \dx t \right) \dx c \\
\ &\geq \
\int_{-1}^1 \sqrt{c^2(2-c^2)} \dx c 
\ = \
\frac{2}{3}(2\sqrt{2} - 1)
\ \approx \ 
1.21895
\, .
\end{align*}
Integrating over $\theta\in (0,2\pi)$, this gives 
\begin{align*}
    \min_{u\in W_{T}^{1,2}(\Omega;\mathbb{S}^2) \text{ with } u_\theta=0} E(u)
    \ &\geq \
    \frac{4}{3}\pi(2\sqrt{2}-1)
    \ \approx \
    7.6589
    \, .
\end{align*}

\end{proof}

\begin{remark}\label{rem:ELE_level_set}
    Note that the Euler-Lagrange equations for $\lambda$ obtained in the proof of Proposition~\ref{prop:lower_bound_level_set_coord} do not preserve the character of $\lambda$ being a level set. 
    Since we're minimizing in a strictly larger class of functions $\lambda$, this means that we cannot expect our lower bound to be optimal.
\end{remark}

\bibliographystyle{acm}
\bibliography{Harmonic_Map_Bibliography}
                
\end{document}